\documentclass[reqno]{amsart}
\usepackage{amssymb,amscd,graphicx,stmaryrd,mathrsfs}
\usepackage[all]{xy}
\usepackage{color}
\usepackage{url}

\newtheorem{theorem}{Theorem}[section]
\newtheorem{proposition}[theorem]{Proposition} 
\newtheorem{corollary}[theorem]{Corollary}
\newtheorem{lemma}[theorem]{Lemma}
\newtheorem{remark}[theorem]{Remark}

\newtheorem{remark&definition}[theorem]{Remark and Definition}

\newtheorem{definition}[theorem]{Definition}

\makeatletter
  
  \@addtoreset{equation}{section}
\makeatother

\begin{document}
\title[Spherical representations of $C^*$-flows I]{Spherical representations of $C^*$-flows I
} 
\author{Yoshimichi UEDA}
\address{
Graduate School of Mathematics, Nagoya University, 
Furocho, Chikusaku, Nagoya, 464-8602, Japan
}
\email{ueda@math.nagoya-u.ac.jp}
\date{\today}
\thanks{Supported by Grant-in-Aid for Scientific Research (B) JP18H01122.}

\maketitle
\begin{abstract}
We propose an abstract framework of a kind of representation theory for $C^*$-flows, i.e., $C^*$-algebras equipped with one-parameter automorphism groups, as a proper generalization of Olshanski's formalism of unitary representation theory for  infinite-dimensional groups such as the infinite-dimensional unitary group $\mathrm{U}(\infty)$. The present framework, in particular, clarifies some overlaps and/or similarities between a certain unitary representation theory of infinite-dimensional groups and existing works in operator algebras, and captures arbitrary projective chains arising from links. 
\end{abstract}

\allowdisplaybreaks{

\tableofcontents

\section{Introduction} It is well known that the concept of unitary representations of locally compact groups and that of $*$-representations of $C^*$-algebras are essentially of the same kind, and the latter includes the former essentially. More precisely, there is a one-to-one correspondence between the unitary representations of a locally compact group $G$ and the $*$-representations of the corresponding (universal) $C^*$-algebra $C^*(G)$. We refer the reader to Dixmier's classical treatise \cite{Dixmier:Book2}, where this point of view is systematically explained. From this point of view, it is quite a natural attempt to re-formulate any type of theories on unitary representations of groups as a theory for operator algebras. 

In the early 70s, some representation theorists, e.g.\ A.A.\ Kirillov, started to look at unitary representations of infinite-dimensional groups, including inductive limits of (locally) compact groups,  whose group $C^*$-algebras are no longer available because of the lack of Haar measures. In the direction, it was probably the first remarkable result that Voiculescu \cite{Voiculescu:CRParis74},\cite{Voiculescu:JMathPuresAppl76} gave a list of extreme characters (which correspond to finite factor unitary representations) of the infinite-dimensional unitary group $\mathrm{U}(\infty) = \varinjlim \mathrm{U}(n)$. After that, Voiculescu pushed his work further with Stratila in a series of papers (see \cite{StratilaVoiculescu:BanachPubl82} as its summary) utilizing operator algebra tools crucially. At the beginning of the 80's, Vershik--Kerov \cite{VershikKerov:SovMathDokl82} and Boyer \cite{Boyer:JOT83} independently established that Voiculescu's list of extreme characters is indeed complete. Moreover, Vershik--Kerov's method is probabilistic and created the so-called asymptotic representation theory (see e.g.\ \cite{Kerov:Book},\cite{BorodinOlshanski:Book}), which grew up to an important part of the so-called integrable probability theory, which we call \emph{the representation-theoretic integrable probability theory}; see \cite{BorodinPetrov:ProbSurv14} for a survey.     

Around the time when Vershik--Kerov's work appeared, Olshanski started his rather original harmonic analysis on infinite-dimensional groups based on the notion of Gelfand pairs. See \cite{Olshanski:Proc90} for the early stage of his study and also \cite{Olshanski:LNM03} for a more recent overview focusing on the infinite symmetric group case. He has used two notions: spherical representations and admissible representations. The former includes finite factor unitary representations of groups as a special case, and the viewpoint of spherical representations has own merit even in the study of finite factor unitary representations; see \cite{Olshanski:JFA03}.  

In the context of integrable probability theory, Gorin introduced in \cite{Gorin:AdvMath12} and has investigated partly with Olshanski \cite{GorinOlshanski:JFA16} a $q$-analogue of the coherent systems of probability measures arising from finite characters of $\mathrm{U}(\infty)$ without any explicit group(-like) objects. Nevertheless, Gorin's probabilistic objects themselves admit an interpretation in terms of random tilings, and moreover, he gave a little speculation on a connection to a certain quantum group. As complements to those works, Ryosuke Sato \cite{Sato:JFA19} gave a successful interpretation of Gorin's $q$-analogue in terms of the inductive system of (duals of) quantum unitary groups $\mathrm{U}_q(n)$ in the sense of Woronowicz (\cite{Woronowicz:CMP87}, one of the most original references). Woronowicz's notion of compact quantum groups is formulated in terms of operator algebras, and thus Sato used the idea of Stratila--Voiculescu using AF-algebras. Besides detailed representation-theoretic stuffs, the key feature of Sato's work is the use of KMS states as a correct quantum group analogue of normalized traces, and the canonical flow is generated by $S\circ S$ ($\neq \mathrm{id}$ when the quantum group in question is not Kac type like the $q$-deformations of  classical groups), where $S$ denotes the standard (not necessarily unitary) antipode (i.e., the `dual' of inverse operation) of the quantum group in question. Then, Sato went on his study further and proposed in \cite{Sato:preprint19} an explicit concept of inductive limits of compact quantum groups in terms of Hopf $W^*$-algebras after the idea due to Masuda--Nakagami \cite{MasudaNakagami:PRIMS94}. As a bonus he could formulate the notion of tensor products of unitary representations for `infinite-dimensional quantum groups'. Study of fusion rules for tensor product representations has basically been missing in the direction, and his further study is clearly promising.           

Unlike Sato's direction mentioned above, our initial motivation was, on one hand, to re-organize the existing works on harmonic analysis on infinite-dimensional groups and/or to justify Sato's works from a rather general point of view like the spirit of Vershik--Kerov's paper \cite{VershikKerov:JSovietMath87}, and on the other hand, to translate the representation-theoretic integrable probability theory into operator algebras. The reason why we wanted to do so was that almost all the existing works in the representation-theoretic integrable probability theory seems (to us) not to need the whole group structure. Moreover, the $q$-analogue should naturally be assisted with \emph{the Tomita--Takesaki theory} similarly to Masuda--Nakagami's approach to quantum groups. These altogether suggested us to start with re-formulating the unitary representation theory for infinite-dimensional groups mainly attributed to Olshanski in such a way that the resulting reformulation includes $q$-deformed ones. As the natural consequences, we realized that the existing harmonic analysis on infinite-dimensional groups has overlaps with and/or connections to several existing works on operator algebras in rather different contexts, and moreover, that any projective chains arising from \emph{links} (see \cite[section 7]{BorodinOlshanski:Book} for this notion) over branching graphs, which are main objects in the representation-theoretic integrable probability theory nowadays, naturally arise from operator algebras. 

In this paper we would like to report the first abstract part of our studies based on the motivations explained above with emphasis on an operator algebraic interpretation of the representation-theoretic integrable probability theory. As the reader might notice, the present work concentrates on only giving an abstract framework and establishing basic facts in a (hopefully) comprehensive way. In the next paper \cite{Ueda:Preprint}, we generalize both the dimension group of the infinite symmetric group $\mathfrak{S}_\infty$ and Olshanski's representation ring of $\mathrm{U}(\infty)$ to the general setup of \S7 in a unified way. It may be remarkable that the generalization also works even for inductive limits of compact quantum groups. We also give, in the next paper, a quantum group version of Olshanski's theory of spherical representations based on the present work to justify Sato's approach to asymptotic representation theory for quantum groups.  

\section{General setup and notations}  

Let $A$ be a unital $C^*$-algebra and $\alpha^t$ be a flow on $A$, i.e., an action $t \in \mathbb{R} \mapsto \alpha^t \in \mathrm{Aut}(A)$. We denote by $K_\beta(\alpha^t)$ all the $(\alpha^t,\beta)$-KMS states, where $\beta \in \mathbb{R}$ means the inverse temperature. When the flow $\alpha^t$ is pointwise norm-continuous, we denote by $A_{\alpha^t}^\infty$ the norm-dense $*$-subalgebra of $A$ consisting of all $\alpha^t$-analytic elements (\cite[subsection 2.5.3]{BratteliRobinson:Book1}). When either $\alpha^t$ is the trivial action or $\beta=0$, the reader should notice that $(\alpha^t,\beta)$-KMS states are nothing but tracial states on $A$. Thus, the formalism below includes the analysis of tracial states. 

Let  $A^\mathrm{op}$ be the opposite $C^*$-algebra with canonical $*$-anti-isomorphism $a \in A \mapsto a^\mathrm{op} \in A^\mathrm{op}$. Let $A^{(2)}:=A\otimes_\mathrm{max}A^\mathrm{op}$ denote the maximal $C^*$-tensor product of $A$ and $A^\mathrm{op}$, in which the algebraic tensor product $A\otimes A^\mathrm{op}$ faithfully sits. Once passing to the conjugate $C^*$-algebra of $A^\mathrm{op}$ one can construct a unique isometric, $*$-preserving, multiplicative, conjugate-linear map $j : A^{(2)} \to A^{(2)}$ such that $j(a\otimes(b^*)^\mathrm{op}) = b\otimes(a^*)^\mathrm{op}$ for all $a, b \in A$. It is clear that $j^2 = \mathrm{id}$ holds and hence $j$ must be bijective. 

In what follows, we will consider only \emph{non-degenerate} $*$-representations of $C^*$ or $W^*$-algebras. Our main references on operator algebras are Bratteli--Robinson's two-volume book \cite{BratteliRobinson:Book1,BratteliRobinson:Book2}, but we also refer to Takesaki's second volume \cite{Takesaki:Book2} at a few points about the Tomita--Takesaki theory.  

\section{$(\alpha^t,\beta)$-spherical representations}

Here is one of the main concepts we introduce and study throughout. Throughout this and the next sections, we assume that the flow $\alpha^t$ is pointwise norm-continuous. This continuity assumption on the flow $\alpha^t$ is natural in the usual setup, and makes the definition intuitive. Later, we will relax this assumption to deal with inductive limits of $W^*$-algebras that naturally appear in the study of infinite-dimensional (quantum) groups. 

\begin{definition} \label{D3.1} 
Let $\Pi : A^{(2)} \curvearrowright \mathcal{H}_\Pi$ be a $*$-representation. A vector $\xi \in \mathcal{H}_\Pi$ is called \emph{an $(\alpha^t,\beta)$-spherical vector in $\Pi \curvearrowright \mathcal{H}_\Pi$}, if 
\begin{equation}\label{Eq3.1} 
\Pi(a\otimes1^\mathrm{op})\xi = \Pi(1\otimes (\alpha^{-i\beta/2}(a))^\mathrm{op})\xi, \quad a \in A_{\alpha^t}^\infty. 
\end{equation}
The subspace consisting of all $(\alpha^t,\beta)$-spherical vectors in $\Pi \curvearrowright \mathcal{H}_\Pi$ is denoted by $\mathcal{H}_\Pi^{(\alpha^t,\beta)}$. 

A $*$-representation $\Pi : A^{(2)} \curvearrowright \mathcal{H}_\Pi$ equipped with a distinguished unit vector $\xi \in \mathcal{H}_\Pi$ is called an \emph{$(\alpha^t,\beta)$-spherical representation of $A$}, if $\xi$ is cyclic for $\Pi(A^{(2)})$ and an $(\alpha^t,\beta)$-spherical vector in $\Pi \curvearrowright \mathcal{H}_\Pi$. 
\end{definition}

We note that condition \eqref{Eq3.1} is equivalent to a more well-balanced formula:  
\begin{equation}\label{Eq3.2}
\Pi(\alpha^{i\beta/4}(a)\otimes1^\mathrm{op})\xi = \Pi(1\otimes(\alpha^{-i\beta/4}(a))^\mathrm{op})\xi, \quad a \in A_{\alpha^t}^\infty. 
\end{equation}

\medskip
The reason why we call the pair $(\Pi : A^{(2)} \curvearrowright \mathcal{H}_\Pi,\xi)$ an $(\alpha^t,\beta)$-spherical representation was already explained in \cite[Appendix]{Sato:preprint19}, where Ryosuke Sato borrowed/adapted our idea presented in this paper to his framework capturing the unitary representation theory of `infinite-dimensional quantum groups' newly introduced by him. However, we will explain, in a clearer way, how our definition can be regarded as a natural generalization of spherical representations for groups in Olshanski's sense. (See also Remark \ref{R9.4}.) Let us consider the tractable case when $A = C^*(\Gamma)$, the group $C^*$-algebra of a countable discrete group $\Gamma$ (e.g.\ the infinite symmetric group $\mathfrak{S}_\infty = \varinjlim\mathfrak{S}_N$) and $\alpha^t = \iota$, the trivial flow. Then, $C^*(\Gamma\times\Gamma) \cong C^*(\Gamma)\otimes_\mathrm{max} C^*(\Gamma) \to C^*(\Gamma)\otimes_\mathrm{max} C^*(\Gamma)^\mathrm{op}$ by $(g_1,g_2) \mapsto g_1\otimes g_2 \mapsto (g_1\otimes g_2^{-1} \mapsto)\, g_1 \otimes (g_2^{-1})^\mathrm{op}$ with $g_1,g_2 \in \Gamma$. Hence any $(\iota,\beta)$-spherical representation $(\Pi : A^{(2)} \curvearrowright \mathcal{H}_\Pi,\xi)$ of $C^*(\Gamma)$ gives a unitary representation $\pi$ of $\Gamma\times\Gamma$, that is, $\pi(g_1,g_2) := \Pi(g_1\otimes(g_2^{-1})^\mathrm{op})$ for every $(g_1,g_2) \in \Gamma\times\Gamma$. Moreover, identity \eqref{Eq3.1} implies that $\pi(g^{-1},e)\xi = \pi(e,g)\xi$ and thus $\pi(g,g)\xi = \xi$ for all $g \in \Gamma$. Consequently, the unitarily equivalent classes of $(\iota,\beta)$-spherical representations of $C^*(\Gamma)$ coincide with the unitarily equivalent classes of spherical representations of the Gelfand pair $(\Gamma\times\Gamma, \Delta(\Gamma))$ with diagonal subgroup $\Delta(\Gamma) = \{(g,g); g \in \Gamma\}$. This special type of Gelfand pairs has been proved to be of great importance in the unitary representation theory for infinite-dimensional groups (see e.g.\  \cite{Olshanski:Proc90}). Actually, the present formalism was originally motivated when we realized that there was a certain similarity between Woronowicz's works on purification map and Olshanski's approach to unitary representations of infinite-dimensional groups.  

\medskip
The natural equivalence relation among $(\alpha^t,\beta)$-spherical representations is defined in the usual manner as follows. 

\begin{definition}\label{D3.2} Two $(\alpha^t,\beta)$-spherical representations $(\Pi_i : A^{(2)} \curvearrowright \mathcal{H}_{\Pi_i},\xi_i)$, $i=1,2$, are \emph{equivalent}, if there is a unitary transform $u : \mathcal{H}_{\Pi_1} \to \mathcal{H}_{\Pi_2}$ so that 
\[
u\Pi_1(x) = \Pi_2(x)u \quad (x \in A^{(2)}) \qquad \text{and} \qquad u\xi_1 = \xi_2. 
\]
We denote by $\mathrm{Rep}_\beta(\alpha^t)$ all the equivalent classes of $(\alpha^t,\beta)$-spherical representation of $A$. 
\end{definition}    

Here is a simple lemma. 

\begin{lemma}\label{L3.3} For an $(\alpha^t,\beta)$-spherical representation $(\Pi : A^{(2)} \curvearrowright \mathcal{H}_\Pi,\xi)$ of $A$, we define 
\begin{equation}\label{Eq3.3}
\varphi_{(\Pi,\xi)}(x) := (\Pi(x)\xi|\xi)_{\mathcal{H}_\Pi}, \qquad x \in A^{(2)}. 
\end{equation}
Then we have 
\[
(a\otimes1^\mathrm{op})\cdot\varphi_{(\Pi,\xi)} = (1\otimes(\alpha^{-i\beta/2}(a))^\mathrm{op})\cdot\varphi_{(\Pi,\xi)}, \quad 
\varphi_{(\Pi,\xi)}\cdot(a\otimes1^\mathrm{op}) = \varphi_{(\Pi,\xi)}\cdot(1\otimes(\alpha^{i\beta/2}(a))^\mathrm{op})
\]
for every $a\in A_{\alpha^t}^\infty$. 
\end{lemma} 
\begin{proof}
We have 
\begin{align*}
((a\otimes1^\mathrm{op})\cdot\varphi_{(\Pi,\xi)})(x) 
&=
(\Pi(x)\Pi(a\otimes1^\mathrm{op})\xi|\xi)_{\mathcal{H}_\Pi}  \\
&=
(\Pi(x)\Pi(1\otimes(\alpha^{-i\beta/2}(a))^\mathrm{op})\xi|\xi)_{\mathcal{H}_\Pi} \\
&= 
((1\otimes(\alpha^{-i\beta/2}(a))^\mathrm{op})\cdot\varphi_{(\Pi,\xi)})(x) 
\end{align*}
for every $x \in A^{(2)}$. Hence we get the first formula. The second formula is just obtained by taking the adjoint (as linear functionals) for the first formula.    
\end{proof}

Motivated by the above lemma we give the next definition. 

\begin{definition}\label{D3.4}
We call $\varphi_{(\Pi,\xi)}$ in \eqref{Eq3.3} \emph{the $(\alpha^t,\beta)$-spherical state associated with a given $(\alpha^t,\beta)$-spherical representation $(\Pi : A^{(2)} \curvearrowright \mathcal{H}_\Pi, \xi)$}. Moreover, we call a state $\varphi$ on $A^{(2)}$ an \emph{abstract $(\alpha^t,\beta)$-spherical state} of $A$, if 
\begin{equation}\label{Eq3.4}
\begin{gathered} 
(a\otimes1^\mathrm{op})\cdot\varphi = (1\otimes(\alpha^{-i\beta/2}(a))^\mathrm{op})\cdot\varphi \\ 
\Big(\text{or equivalently}, \varphi\cdot(a\otimes1^\mathrm{op}) = \varphi\cdot(1\otimes(\alpha^{i\beta/2}(a))^\mathrm{op}))\Big) 
\end{gathered}
\end{equation}
holds for every $a \in A_{\alpha^t}^\infty$. We denote by $S_\beta(\alpha^t)$ all the abstract $(\alpha^t,\beta)$-spherical states of $A$. 
\end{definition}

We clarify the relationship between $(\alpha^t,\beta)$-spherical representations and (abstract) $(\alpha^t,\beta)$-spherical states as follows.  

\begin{proposition}\label{P3.5} For an abstract $(\alpha^t,\beta)$-spherical state $\varphi$ of $A$, the GNS triple $(\Pi_\varphi : A^{(2)} \curvearrowright \mathcal{H}_\varphi,\xi_\varphi)$ associated with $\varphi$ becomes an $(\alpha^t,\beta)$-spherical representation of $A$ such that $\varphi_{(\Pi_\varphi,\xi_\varphi)} = \varphi$. Moreover, 
\[
[(\Pi : A^{(2)} \curvearrowright \mathcal{H}_\Pi,\xi)] \in \mathrm{Rep}_\beta(\alpha^t) \mapsto \varphi_{(\Pi,\xi)} \in S_\beta(\alpha^t)
\] 
is a well-defined bijection with inverse map $\varphi \mapsto [(\Pi_\varphi : A^{(2)} \curvearrowright \mathcal{H}_\varphi,\xi_\varphi)]$. 
\end{proposition}
\begin{proof} 
We have
\begin{equation*}
\begin{aligned}
(\Pi_\varphi(a\otimes1^\mathrm{op})\xi_\varphi|\Pi_\varphi(x)\xi_\varphi)_{\mathcal{H}_\varphi} 
&= 
\varphi(x^*(a\otimes1^\mathrm{op})) \\
&= 
((a\otimes1^\mathrm{op})\cdot\varphi)(x^*) \\
&= 
((1\otimes(\alpha^{-i\beta/2}(a))^\mathrm{op})\cdot\varphi)(x^*) \\
&=
\varphi(x^*(1\otimes(\alpha^{-i\beta/2}(a))^\mathrm{op})) \\
&=
(\Pi_\varphi(1\otimes(\alpha^{-i\beta/2}(a))^\mathrm{op}))\xi_\varphi|\Pi_\varphi(x)\xi_\varphi)_{\mathcal{H}_\varphi} 
\end{aligned}
\end{equation*}
for every $x \in A^{(2)}$. Hence we obtain that $\Pi_\varphi(a\otimes1^\mathrm{op})\xi_\varphi=\Pi_\varphi(1\otimes(\alpha^{-i\beta/2}(a))^\mathrm{op})\xi_\varphi$. The formula $\varphi_{(\Pi_\varphi,\xi_\varphi)}=\varphi$ is trivial by definition. 

By the definition of equivalence among $(\alpha^t,\beta)$-spherical representations, $[(\Pi,\xi)]\mapsto \varphi_{(\Pi,\xi)}$ is clearly well-defined and injective. The first part of the statement also shows the surjectivity. Moreover, by the uniqueness statement on GNS triples, we observe that the mapping $\varphi \mapsto [(\Pi_\varphi : A^{(2)} \curvearrowright \mathcal{H}_\varphi,\xi_\varphi)]$ is well defined. 
\end{proof} 

For an arbitrary $\omega \in K_\beta(\alpha^t)$, let $(\pi_\omega : A \curvearrowright \mathcal{H}_\omega,\xi_\omega)$ be the GNS triple associated with $\omega$. The mapping $\pi_\omega(a)\xi_\omega \in \pi_\omega(A_{\alpha^t}^\infty)\xi_\omega \mapsto \pi_\omega(\alpha^{i\beta/2}(a^*))\xi_\omega \in \mathcal{H}_\omega$ has a unique isometric conjugate-linear extension $J_\omega : \mathcal{H}_\omega \to \mathcal{H}_\omega$ with $J_\omega^2 = I$. Then we have 
\begin{equation}\label{Eq3.5}
\pi_\omega(a)J_\omega\pi_\omega(b^*)J_\omega \pi_\omega(c)\xi_\omega 
= 
\pi_\omega(a\,c\,\alpha^{i\beta/2}(b))\xi_\omega =J_\omega\pi_\omega(b^*)J_\omega \pi_\omega(a) \pi_\omega(c)\xi_\omega  
\end{equation}
for any $a,b,c \in A_{\alpha^t}^\infty$, and hence $J_\omega \pi_\omega(A) J_\omega \subseteq \pi_\omega(A)'$. By universality, there is a $*$-representation $\pi_\omega^{(2)} : A^{(2)} \curvearrowright \mathcal{H}_\omega$ sending $a\otimes b^\mathrm{op} \mapsto \pi_\omega(a) J_\omega \pi_\omega(b^*)J_\omega$ for any simple tensor $a\otimes b^\mathrm{op} \in A^{(2)}$. It is known that $\xi_\omega$ also is separating for $\pi_\omega(A)''$ (see e.g.\ \cite[Corollary 5.3.9]{BratteliRobinson:Book2}) and there is a unique faithful normal state $\bar{\omega}$ on $\pi_\omega(A)''$ such that $\bar{\omega}(\pi_\omega(a))=\omega(a)$ for every $a \in A$. Then, by Takesaki's characterization,  the mapping $\pi_\omega(a) \mapsto \pi_\omega(\alpha^{-\beta t}(a))$ gives a modular automorphism $\sigma_t^{\bar{\omega}}$ ($t\in\mathbb{R}$) of the faithful normal state $\bar{\omega}$, and by Tomita's theorem 
\begin{equation}\label{Eq3.6}
\pi_\omega(A)'=J_\omega\pi_\omega(A)'' J_\omega
\end{equation} 
holds. 

\medskip
Before giving the next lemma, we recall \emph{the purification map} $\omega \mapsto \omega^{(2)}$ from the state space $S(A)$ of $A$ into the state space $S(A^{(2)})$ introduced by Woronowicz \cite{Woronowicz:CMP72,Woronowicz:CMP73,Woronowicz:RMP74}. The purification $\omega^{(2)}$ of a given $\omega \in S(A)$ is a unique exact, $j$-positive state on $A^{(2)}$ with subject to 
$\omega^{(2)}(a\otimes1^\mathrm{op}) = \omega(a)$ for every $a \in A$ (see \cite[section 2]{Woronowicz:RMP74}). Here a $\varphi \in S(A^{(2)})$ is said to be \emph{exact}, if $\pi_\varphi(A\otimes\mathbb{C}1^\mathrm{op})' = \pi_{\varphi}(\mathbb{C}1\otimes A^\mathrm{op})''$ in $\mathcal{H}_\varphi$ holds, where $\pi_\varphi : A^{(2)} \curvearrowright \mathcal{H}_\varphi$ is the GNS representation associated with $\varphi$. Also, $\varphi$ is said to be \emph{$j$-positive}, if $\varphi(a\otimes(a^*)^\mathrm{op}) \geq 0$ for all $a \in A$. The $j$-positivity implies that $\varphi\circ j(x) = \overline{\varphi(x)}$ for every $x \in A^{(2)}$. In particular, any $j$-positive state $\varphi$ enjoys that $\varphi(a\otimes 1^\mathrm{op}) = \varphi(1\otimes a^\mathrm{op})$ for every $a \in A$.

\begin{lemma}\label{L3.6} The pair $(\pi_\omega^{(2)} : A^{(2)} \curvearrowright \mathcal{H}_{\omega},\xi_\omega)$ is an $(\alpha^t,\beta)$-spherical representation of $A$ and the associated $(\alpha^t,\beta)$-spherical state $\varphi_{(\pi_\omega^{(2)},\xi_\omega)}$ enjoys that 
\begin{equation}\label{Eq3.7}
\begin{aligned}
\varphi_{(\pi_\omega^{(2)},\xi_\omega)}(a\otimes (b^*)^\mathrm{op}) 
&= 
\omega(a\,\alpha^{i\beta/2}(b^*)) = \omega(\alpha^{i\beta/4}(b)^*\alpha^{i\beta/4}(a)) \\
&= 
(\pi_\omega(\alpha^{i\beta/4}(a))\xi_\omega|\pi_\omega(\alpha^{i\beta/4}(b))\xi_\omega)_{\mathcal{H}_\omega}
\end{aligned}
\end{equation}
 holds for any $a,b \in A_{\alpha^t}^\infty$. Moreover, $\omega \in K_\beta(\alpha^t) \mapsto \varphi_{(\pi_\omega^{(2)},\xi_\omega)} \in S_\beta(\alpha^t)$ is the restriction of the purification map $\omega \in S(A) \mapsto \omega^{(2)} \in S(A^{(2)})$ to $K_\beta(\alpha^t)$. 
\end{lemma}
\begin{proof}
For every $a \in A_{\alpha^t}^\infty$ we have
\begin{equation*}
\begin{aligned}
\pi_\omega^{(2)}(a\otimes1^\mathrm{op})\xi_\omega
&=
\pi_\omega(a)\xi_\omega 
=
\pi_\omega(\alpha^{i\beta/2}((\alpha^{-i\beta/2}(a)^*)^*))\xi_\omega \\
&=
J_\omega\pi_\omega(\alpha^{-i\beta/2}(a)^*)J_\omega\xi_\omega
=
\pi_\omega^{(2)}(1\otimes(\alpha^{-i\beta/2}(a))^\mathrm{op})\xi_\omega. 
\end{aligned}
\end{equation*}
Since $\xi_\omega$ is trivially cyclic for $\pi_\omega^{(2)}(A^{(2)})$, we conclude that $(\pi_\omega^{(2)} : A^{(2)} \curvearrowright \mathcal{H}_{\omega},\xi_\omega)$ is an $(\alpha^t,\beta)$-spherical representation of $A$. By the construction of $\varphi_{(\pi_\omega^{(2)},\xi_\omega)}$ and the uniqueness of GNS representations, the triple $(\pi_\omega^{(2)} : A^{(2)} \curvearrowright \mathcal{H}_{\omega},\xi_\omega)$ is indeed the GNS triple associated with $\varphi_{(\pi_\omega^{(2)},\xi_\omega)}$. 

For any $a,b \in A_{\alpha^t}^\infty$, by equation \eqref{Eq3.5} we also have
\begin{equation*}
\begin{aligned}
\varphi_{(\pi_\omega^{(2)},\xi_\omega)}(a\otimes (b^*)^\mathrm{op}) 
&=
(\pi_\omega(a)J_\omega\pi_\omega(b)J_\omega\xi_\omega|\xi_\omega)_{\mathcal{H}_\omega} \\
&= 
(\pi_\omega(a\,\alpha^{i\beta/2}(b^*))\xi_\omega|\xi_\omega)_{\mathcal{H}_\omega} 
= \omega(a\,\alpha^{i\beta/2}(b^*)), 
\end{aligned}
\end{equation*}
and hence the $(\alpha^t,\beta)$-KMS condition with $\omega$ gives the desired formula \eqref{Eq3.7}. 

The formula we have just proved shows that $\varphi_{(\pi_\omega^{(2)},\xi_\omega)}(a\otimes(a^*)^\mathrm{op}) = \Vert \pi_\omega(\alpha^{i\beta/4}(a))\xi_\omega\Vert_{\mathcal{H}_\omega}^2 \geq 0$ for all $a \in A_{\alpha^t}^\infty$. Hence, by the density of $A_{\alpha^t}^\infty$, we see that $\varphi_{(\pi_\omega^{(2)},\xi_\omega)}$ is $j$-positive. By identity \eqref{Eq3.6}, the $\varphi_{(\pi_\omega^{(2)},\xi_\omega)}$ is trivially exact. Moreover, $\varphi_{(\pi_\omega^{(2)},\xi_\omega)}(a\otimes1^\mathrm{op}) = (\pi_\omega^{(2)}(a\otimes1^\mathrm{op})\xi_\omega\,|\,\xi_\omega)_{\mathcal{H}_\omega} = (\pi_\omega(a)\xi_\omega\,|\,\xi_\omega)_{\mathcal{H}_\omega} = \omega(a)$ for all $a \in A$. Hence $\varphi_{(\pi_\omega^{(2)},\xi_\omega)}$ is nothing but the purification $\omega^{(2)}$. 
\end{proof}

Due to the last statement in the above lemma, we often denote by $\omega^{(2)}$ the $(\alpha^t,\beta)$-spherical state associated with $(\pi_\omega^{(2)} : A^{(2)} \curvearrowright \mathcal{H}_\omega, \xi_\omega)$ for simplicity. The next proposition tells us that the image of $K_\beta(\alpha^t)$ under the purification map is exactly $S_\beta(\alpha^t)$. 

\begin{proposition}\label{P3.7}
For a given $\varphi \in S_\beta(\alpha^t)$, the left partial state $\varphi_L \in S(A)$ defined to be $a \in A \mapsto \varphi(a\otimes1^\mathrm{op}) \in \mathbb{C}$ falls into $K_\beta(\alpha^t)$ and $\varphi_{(\pi_{\varphi_L}^{(2)},\xi_{\varphi_L})} = \varphi$ holds.  In particular,  the mapping $\omega \in K_\beta(\alpha^t) \mapsto \omega^{(2)} \in S_\beta(\alpha^t)$ is bijective with inverse map $\varphi \mapsto \varphi_L$. 
\end{proposition}
\begin{proof}
For every $a \in A_{\alpha^t}^\infty$ and every $b \in A$ we have 
\begin{equation*}
\begin{aligned} 
\varphi_L(ab) 
&= 
\varphi(ab\otimes1^\mathrm{op}) \\
&= 
(\varphi\cdot(a\otimes1^\mathrm{op}))(b\otimes1^\mathrm{op}) \\
&= 
(\varphi\cdot(1\otimes(\alpha^{i\beta/2}(a))^\mathrm{op})))(b\otimes1^\mathrm{op}) \quad \text{(by equation \eqref{Eq3.4})} \\
&= 
\varphi(b\otimes(\alpha^{i\beta/2}(a))^\mathrm{op}) \\ 
&= 
((1\otimes(\alpha^{i\beta/2}(a))^\mathrm{op})\cdot\varphi)(b\otimes1^\mathrm{op}) \\
&= 
((\alpha^{i\beta}(a)\otimes1^\mathrm{op})\cdot\varphi)(b\otimes1^\mathrm{op}) \quad \text{(by equation \eqref{Eq3.4} again)} \\
&=
\varphi(b\,\alpha^{i\beta}(a)\otimes1^\mathrm{op}) \\
&= 
\varphi_L(b\,\alpha^{i\beta}(a)), 
\end{aligned}
\end{equation*} 
showing that $\varphi_L$ falls into $K_\beta(\alpha^t)$. Moreover, by Lemma \ref{L3.6} we observe that 
\begin{equation*}
\begin{aligned} 
\varphi_{(\pi_{\varphi_L}^{(2)},\xi_{\varphi_L})}(a\otimes(b^*)^\mathrm{op}) 
&= 
\varphi_L(a\,\alpha^{i\beta/2}(b^*)) 
= 
\varphi(a\,\alpha^{i\beta/2}(b^*)\otimes1^\mathrm{op}) \\
&= 
((\alpha^{i\beta}(b^*)\otimes1^\mathrm{op})\cdot\varphi)(a\otimes1^\mathrm{op}) \\
&= 
((1\otimes(b^*)^\mathrm{op})\cdot\varphi)(a\otimes1^\mathrm{op}) \quad \text{(by equation \eqref{Eq3.4})} \\
&= 
\varphi(a\otimes(b^*)^\mathrm{op})
\end{aligned}
\end{equation*}
for all $a,b \in A_{\alpha^t}^\infty$. Hence we obtain that $\varphi_{(\pi_{\varphi_L}^{(2)},\xi_{\varphi_L})}=\varphi$ by the density of $A_{\alpha^t}^\infty$. The last assertion is trivial now, and hence we are done.  
\end{proof}

Summing up the discussion so far, we have the following fundamental theorem:  

\begin{theorem}\label{T3.8} There are one-to-one correspondences among $K_\beta(\alpha^t)$, $S_\beta(\alpha^t)$ and $\mathrm{Rep}_\beta(\alpha^t)$ as follows. 
\begin{equation*}
\begin{matrix}
K_\beta(\alpha^t) & \longleftrightarrow & S_\beta(\alpha^t) & \longleftrightarrow & \mathrm{Rep}_\beta(\alpha^t) \\
\omega & \longmapsto & \omega^{(2)} & \longmapsto & [(\pi_\omega^{(2)} : A^{(2)} \curvearrowright \mathcal{H}_\omega, \xi_\omega)] \\
\varphi_L & \longmapsfrom & \varphi := \varphi_{(\Pi,\xi)} & \longmapsfrom & [(\Pi : A^{(2)} \curvearrowright \mathcal{H}_\Pi, \xi)],  
\end{matrix}
\end{equation*}
where the map $K_\beta(\alpha^t) \to S_\beta(\alpha^t)$ is the restriction of the purification map to $K_\beta(\alpha^t)$. Moreover, the restriction of $\pi_{\omega^{(2)}} : A^{(2)} \curvearrowright \mathcal{H}_\omega$ to $A \cong A\otimes\mathbb{C}1^\mathrm{op}$ equipped with $\xi_\omega$ is nothing less than the GNS triple associated with $\omega$. 
\end{theorem} 

We would like to emphasize that the purification map $\omega \in S(A) \mapsto \omega^{(2)} \in S(A^{(2)})$ is not surjective, though the image of its restriction to $K_\beta(\alpha^t)$ is completely determined as all the abstract $(\alpha^t,\beta)$-spherical states of $A$, which is a much simpler notion than that of exact, $j$-positive states. In fact, it seems very difficult to confirm the exactness for an arbitrarily given state on $A^{(2)}$. Moreover, Woronowicz established in \cite[Theorem 1.2]{Woronowicz:CMP72}\footnote{The theorem does not require the separability assumption for $A$ nor the factoriality for states; see \cite[section 2]{Woronowicz:RMP74}.} that \emph{the purification map $\omega \mapsto \omega^{(2)}$ sends each quasi-equivalent class to an equivalent class, where the {\rm(}quasi-{\rm)}equivalence for states is by means of their GNS representations}. Hence we can discuss the quasi-equivalence for $(\alpha^t,\beta)$-KMS states through the correspondences established in Theorem \ref{T3.8}.   

\medskip
Here is a corollary. 

\begin{corollary} \label{C3.9} We have the following assertions{\rm:} 
\begin{itemize} 
\item[(1)] $\omega \in K_\beta(\alpha) \mapsto \omega^{(2)} \in S_\beta(\alpha)$ is bijective and affine homeomorphic with respect to weak$^*$ topology. 
\item[(2)] $\omega \in K_\beta(\alpha^t)$ is a factor state if and only if $\omega^{(2)} \in S_\beta(\alpha^t)$ is a pure state. In other words, any extreme points in $S_\beta(\alpha^t)$ are pure states. 
\item[(3)] Let $\Pi : A^{(2)} \curvearrowright \mathcal{H}_\Pi$ be a $*$-representation. If $\Pi : A^{(2)} \curvearrowright \mathcal{H}_\Pi$ is irreducible, then $\dim \mathcal{H}_\Pi^{(\alpha^t,\beta)} \leq 1$. Conversely, if $\dim \mathcal{H}_\Pi^{(\alpha^t,\beta)} = 1$, then the cyclic sub-representation of $\Pi  : A^{(2)}\curvearrowright \mathcal{H}_\Pi $ associated with any unit vector in $\mathcal{H}_\Pi^{(\alpha^t,\beta)}$ must be irreducible.  
\end{itemize} 
\end{corollary} 
\begin{proof} 
Item (1): Clearly $\varphi \in S_\beta(\alpha^t) \mapsto \varphi_L \in K_\beta(\alpha^t)$ is a weak$^*$ continuous, affine map, which is the inverse of the purification map $\omega \mapsto \omega^{(2)}$ by Theorem \ref{T3.8}. Hence it suffices to prove that $S_\beta(\alpha^t)$ is weak$^*$ compact. However, it is easy to see that identity \eqref{Eq3.4} is preserved after passing the weak$^*$ limit. Hence we are done. 

\medskip
Item (2): The only if part is a special case of Woronowicz's result \cite{Woronowicz:CMP72},\cite[section 2]{Woronowicz:RMP74}. However, it is easy, by the construction of $\omega^{(2)}$ with $\omega \in K_\beta(\alpha)$, to show the desired assertion as follows. Observe that $\pi^{(2)}(A^{(2)})' = \pi_\omega(A)' \cap (J_\omega \pi_\omega(A) J_\omega)' = \pi_\omega(A)' \cap \pi_\omega(A)''$ by identity \eqref{Eq3.6}. Hence $\pi_\omega(A)''$ is a factor, or equivalently $\omega$ is a factor state, if and only if $\pi_\omega^{(2)}(A^{(2)})'$ is $1$-dimensional, that is, $\pi_{\omega}^{(2)}$ is irreducible, or equivalently $\omega^{(2)}$ is a pure state.  

\medskip
The first part of item (3): On the contrary, suppose that $\dim\mathcal{H}_\Pi^{(\alpha^t,\beta)} \geq 2$. Then there is an orthogonal pair of unit vectors $\xi_1, \xi_2 \in \mathcal{H}_\Pi^{(\alpha^t,\beta)}$. Remark that any nonzero vector in $\mathcal{H}_\Pi$ is cyclic for $\Pi(A^{(2)})$ thanks to the irreducibility. Hence we have two $(\alpha^t,\beta)$-spherical representations $(\Pi : A^{(2)} \curvearrowright \mathcal{H}_\Pi,\xi_1)$ and $(\Pi : A^{(2)} \curvearrowright \mathcal{H}_\Pi,\xi_2)$. Since $\Pi(A^{(2)})'$ is $1$-dimensional by the irreducibility, these two $(\alpha^t,\beta)$-spherical representations are not equivalent. Therefore, Theorem \ref{T3.8}  tells us that the $(\alpha^t,\beta)$-spherical states $\varphi_1$ and $\varphi_2$ associated with them are not the same and hence $(\varphi_1)_L \neq (\varphi_2)_L$. For the ease of notation, we write $\omega_i := (\varphi_i)_L \in K_\beta(\alpha^t)$, $i=1,2$. Since the GNS representation of the $\omega_i^{(2)} = \varphi_i$ is equivalent to $\Pi : A^{(2)} \curvearrowright \mathcal{H}_\Pi$ for both $i=1,2$, we conclude, by \cite[Theorem 1.2]{Woronowicz:CMP72}, that the GNS representations $\pi_{\omega_i} : A \curvearrowright \mathcal{H}_{\omega_i}$, $i=1,2$, are quasi-equivalent. However, any factor states in $K_\beta(\alpha^t)$ are equal or disjoint (see \cite[Theorem 5.3.30(3),(4)]{BratteliRobinson:Book2}), which shows $\omega_1 = \omega_2$, a contradiction. 

The second part of item (3): By assumption we have $\mathcal{H}_\Pi^{(\alpha^t,\beta)} = \mathbb{C}\xi$ with a unit vector $\xi$. Let $\Pi^\xi : A^{(2)} \curvearrowright \mathcal{H}_{\Pi^\xi}$ with $\mathcal{H}_{\Pi^\xi} := [\Pi(A^{(2)})\xi]$ be the cyclic sub-representation associated with $\xi$. On the contrary, suppose that $\Pi^\xi$ is not irreducible. Then $\Pi^\xi(A^{(2)})' \neq \mathbb{C}1$, and thus there is an orthogonal pair $\xi_1,\xi_2 \in \mathcal{H}_{\Pi^\xi}^{(\alpha^t,\beta)}$ such that $\xi = \xi_1 + \xi_2$. By construction, $\mathcal{H}_{\Pi^\xi}^{(\alpha^t,\beta)} \subseteq \mathcal{H}_\Pi^{(\alpha^t,\beta)}$, a contradiction to the orthogonality of $\xi_1, \xi_2$. Hence $\Pi^\xi : A^{(2)} \curvearrowright \mathcal{H}_{\Pi^\xi}$ must be irreducible.         
\end{proof}  

A clearer treatment on item (3) above from a wider point of view is possible; see section 6. 

\section{Abstract investigations}

We keep the setting and the notation in section 3. 

Let $\Pi_\mathrm{univ} : A^{(2)} \curvearrowright \mathcal{H}_\mathrm{univ}$ be the universal representation of $A^{(2)}$, that is, 
\[
\Pi_\mathrm{univ} := \bigoplus_{\varphi \in S(A^{(2)})} \pi_\varphi : A^{(2)}\ \curvearrowright\ \mathcal{H}_\mathrm{univ} := \bigoplus_{\varphi \in S(A^{(2)})} \mathcal{H}_\varphi, 
\]
where $\pi_\varphi : A^{(2)} \curvearrowright \mathcal{H}_\varphi$ is the GNS representation associated with a state $\varphi \in S(A^{(2)})$. The existence of $(\alpha^t,\beta)$-spherical representations of $A$, or equivalently $(\alpha^t,\beta)$-KMS states on $A$, can be examined with calculating the closed subspace $\mathcal{H}_\mathrm{univ}^{(\alpha^t,\beta)}$.  In the subsequent subsection, we will examine this idea at the following two aspects; one is to give a reconstruction of Woronowicz's existence theorem for KMS states in the spherical representation point of view and the other is to investigate inductive flows. For the latter we will give a characterization for $(\alpha^t,\beta)$-spherical vectors, which enables us to deal with non-continuous flows in section 5. 

\subsection{Woronowicz's existence result on KMS states revisited} Following Woronowicz \cite{Woronowicz:LMP85} we let $\mathcal{L}^\beta(\alpha^t)$ denote the left ideal of $A^{(2)}$ generated by all the $a\otimes1^\mathrm{op} - 1\otimes(\alpha^{-i\beta/2}(a))^\mathrm{op}$ with $a \in A_{\alpha^t}^\infty$. 

\begin{lemma}\label{L4.1} 
The projection $e$ onto $\mathcal{H}_{\Pi_\mathrm{univ}}^{(\alpha^t,\beta)}$ falls into $\Pi_\mathrm{univ}(A^{(2)})''$ and the closure $\overline{\Pi_\mathrm{univ}(\mathcal{L}^\beta(\alpha^t))}$ of $\Pi_\mathrm{univ}(\mathcal{L}_\beta(\alpha^t))$ in the weak operator topology coincides with $\Pi_\mathrm{univ}(A^{(2)})''(1-e)$.  
\end{lemma} 
\begin{proof} 
By construction, $\mathcal{H}_{\Pi_\mathrm{univ}}^{(\alpha^t,\beta)}$ is invariant for $\Pi_\mathrm{univ}(A^{(2)})'$, and hence $e$ falls into $\Pi_u(A^{(2)})''$. Since $\overline{\Pi_\mathrm{univ}(\mathcal{L}^\beta(\alpha^t))}$ is clearly a left ideal of $\Pi_\mathrm{univ}(A^{(2)})''$, there exists a unique projection $f \in \Pi_\mathrm{univ}(A^{(2)})''$ such that $\overline{\Pi_\mathrm{univ}(\mathcal{L}^\beta(\alpha^t))} = \Pi_\mathrm{univ}(A^{(2)})'' f$. 

Observe that $\Pi_\mathrm{univ}(a\otimes1^\mathrm{op}-1\otimes(\alpha^{-i\beta/2}(a))^\mathrm{op})(1-f)=0$ for all $a \in A_{\alpha^t}^\infty$. This implies that $1-f \leq e$, that is $1-e \leq f$. 

For any projection $q \in \Pi_\mathrm{univ}(A^{(2)})''$, we observe $\overline{\Pi_\mathrm{univ}(\mathcal{L}^\beta(\alpha^t))}q = \{0\} \Leftrightarrow \Pi_\mathrm{univ}(\mathcal{L}^\beta(\alpha^t))q = \{0\}$, which is equivalent to $\Pi_\mathrm{univ}(a\otimes1^\mathrm{op}-1\otimes(\alpha^{-i\beta/2}(a))^\mathrm{op})q = 0$ for all $a \in A_{\alpha^t}^\infty$. Since $\Pi_\mathrm{univ}(a\otimes1^\mathrm{op}-1\otimes(\alpha^{-i\beta/2}(a))^\mathrm{op})e = 0$ for all $a \in A_{\alpha^t}^\infty$ we obtain, by the construction of $f$, that 
\[
1-f = \bigvee_{\substack{ q \in \Pi_\mathrm{univ}(A^{(2)})''; \\ 
\text{$q = q^* = q^2$ and $\overline{\Pi_\mathrm{univ}(\mathcal{L}^\beta(\alpha))}q = 0$}}} q \geq e, 
\]
implying that $f \leq 1-e$. Consequently, $f = 1-e$.        
\end{proof}  

Here is a reproduction of Woronowicz's result mentioned above in terms of $(\alpha^t,\beta)$-spherical representations.

\begin{corollary}\label{C4.2} \rm{(Woronowicz \cite{Woronowicz:LMP85} for (i) $\Leftrightarrow$ (iv))} The following are equivalent: 
\begin{itemize} 
\item[(i)] There exists an $(\alpha^t,\beta)$-KMS state on $A$. 
\item[(ii)] There exists an $(\alpha^t,\beta)$-spherical representation of $A$. 
\item[(iii)] $\dim\mathcal{H}_\mathrm{univ}^{(\alpha^t,\beta)} \geq 1$. 
\item[(iv)] $\mathcal{L}^\beta(\alpha^t) \neq A^{(2)}$. 
\end{itemize}
\end{corollary} 
\begin{proof} 
The equivalence among (i),(ii) and (iii) follows from Theorem \ref{T3.8} with the definition of $\mathcal{H}_{\Pi_\mathrm{univ}}^{(\alpha^t,\beta)}$. Hence it remains to show (iii) $\Leftrightarrow$ (iv). 

We will use the notations in Lemma \ref{L4.1} in what follows. It suffices to prove that $e=0$ if and only if $\mathcal{L}^\beta(\alpha^t) \neq A^{(2)}$. Observe that the if part is trivial. Assume $e=0$. By Lemma \ref{L4.1} we have $\Pi_\mathrm{univ}(\mathcal{L}^\beta(\alpha^t)) = \Pi_\mathrm{univ}(A^{(2)})\cap\overline{\Pi_\mathrm{univ}(\mathcal{L}^\beta(\alpha^t))} = \Pi_\mathrm{univ}(A^{(2)})\cap\Pi_\mathrm{univ}(A^{(2)})'' = \Pi_\mathrm{univ}(A^{(2)})$, implying that $\mathcal{L}^\beta(\alpha^t) = A^{(2)}$ since the universal representation $\Pi_\mathrm{univ}$ is faithful.    
\end{proof} 

The consideration in this subsection admits a certain interpretation in terms of representation theory; we will discuss it in section 6. 

\subsection{Inductive flows} We start with the next general characterization of $(\alpha^t,\beta)$-spherical vectors, which is analogous to the KMS condition. 

\begin{proposition}\label{P4.3} Let $\Pi : A^{(2)} \curvearrowright \mathcal{H}_\Pi$ be a $*$-representation. For a given vector $\xi \in \mathcal{H}_\Pi$ the following conditions are equivalent:   
\begin{itemize} 
\item[(i)] $\xi$ is an $(\alpha^t,\beta)$-spherical vector in $\Pi : A^{(2)} \curvearrowright \mathcal{H}_\Pi$. 
\item[(ii)] There exists a subset $\mathcal{A}$ of $A_{\alpha^t}^\infty$ such that $\mathcal{A}$ is norm dense in $A$ and that \eqref{Eq3.1} holds for every $a \in \mathcal{A}$.   
\item[(iii)] There exist a norm dense subset $\mathcal{A}$ of $A$ and a dense subset $\mathcal{D}$ of $\mathcal{H}_\Pi$ with the following property{\rm:} For each $a \in \mathcal{A}$ and each $\eta \in \mathcal{D}$ there is a bounded continuous function $F(z)$ on $0\wedge(\beta/2) \leq \mathrm{Im}\,z \leq 0\vee(\beta/2)$ such that $F(z)$ is holomorphic in its interior and 
\[
F(t) = (\Pi(\alpha^t(a)\otimes1^\mathrm{op})\xi|\eta)_{\mathcal{H}_\Pi}, \quad 
F(t+i\beta/2) = (\Pi(1\otimes(\alpha^t(a))^\mathrm{op})\xi|\eta)_{\mathcal{H}_\Pi}
\]
for all $t \in \mathbb{R}$. 
\item[(iv)] For each $a \in A$ and each $\eta \in \mathcal{H}_\Pi$ there is a bounded continuous function $F(z)$ on $0\wedge(\beta/2) \leq \mathrm{Im}\,z \leq 0\vee(\beta/2)$ such that $F(z)$ is holomorphic in its interior and 
\[
F(t) = (\Pi(\alpha^t(a)\otimes1^\mathrm{op})\xi|\eta)_{\mathcal{H}_\Pi}, \quad 
F(t+i\beta/2) = (\Pi(1\otimes(\alpha^t(a))^\mathrm{op})\xi|\eta)_{\mathcal{H}_\Pi}
\]
for all $t \in \mathbb{R}$. 
\end{itemize}
\end{proposition} 
\begin{proof}
(i) $\Rightarrow$ (ii): Trivial with $\mathcal{A} = A_{\alpha^t}^\infty$. 

(ii) $\Rightarrow$ (iii): Assume that (ii) holds true. There is a sequence $a_n \in A_{\alpha^t}^\infty$ so that $\Vert a_n - a\Vert \to 0$ as $n\to\infty$. Consider the entire function 
\[
F_n(z) := (\Pi(\alpha^z(a_n)\otimes1^\mathrm{op})\xi|\eta)_{\mathcal{H}_\Pi} = (\Pi(1\otimes(\alpha^{z-i\beta/2}(a_n))^\mathrm{op})\xi|\eta)_{\mathcal{H}_\Pi}. 
\]
(The second equality holds by \eqref{Eq3.1}.) Since $\Vert \alpha_{t+is}(a_n)\Vert = \Vert \alpha_{is}(a_n)\Vert$, $F_n(z)$ is bounded on  $0\wedge(\beta/2) \leq \mathrm{Im}\,z \leq 0\vee(\beta/2)$. Observe that  
\begin{align*}
F_n(t) 
&=  
(\Pi(\alpha^t(a_n)\otimes1^\mathrm{op})\xi|\eta)_{\mathcal{H}_\Pi}, \\
F_n(t+i\beta/2) 
&=(\Pi(1\otimes(\alpha^t(a_n))^\mathrm{op})\xi|\eta)_{\mathcal{H}_\Pi}. 
\end{align*} 
Since 
\begin{align*} 
\sup_{t \in \mathbb{R}}|F_n(t) - (\Pi(\alpha^t(a)\otimes1^\mathrm{op})\xi|\eta)_{\mathcal{H}_\Pi}| 
&\leq 
\Vert a_n-a\Vert \Vert\xi\Vert_{\mathcal{H}_\Pi} \Vert\eta \Vert_{\mathcal{H}_\Pi} \to 0, \\
\sup_{t \in \mathbb{R}}|F_n(t+i\beta/2) - (\Pi(1\otimes(\alpha^t(a))^\mathrm{op})\xi|\eta)_{\mathcal{H}_\Pi}| 
&\leq 
\Vert a_n-a\Vert \Vert\xi\Vert_{\mathcal{H}_\Pi} \Vert\eta \Vert_{\mathcal{H}_\Pi}\to 0, 
\end{align*} 
the limit of $F_n(z)$ as $n\to\infty$ exists on $0\wedge(\beta/2) \leq \mathrm{Im}z \leq 0 \vee(\beta/2)$ and gives the desired function in (iii) thanks to the Phragmen--Lindel\"{o}f method. 

(iii) $\Rightarrow$ (iv): Choose $a_n \in \mathcal{A}$ and $\eta_n \in \mathcal{D}$ in such a way that $\Vert a_n - a\Vert \to 0$ and $\Vert \eta_n - \eta\Vert_{\mathcal{H}_\Pi} \to 0$ as $n\to\infty$. Each Let $F_n(z)$ be the function in (iii) for $(a_n,\eta_n)$. Then   
\begin{align*} 
&\sup_{t \in \mathbb{R}}|F_n(t) - (\Pi(\alpha^t(a)\otimes1^\mathrm{op})\xi|\eta)_{\mathcal{H}_\Pi}| \\
&\qquad\leq 
\Vert a_n-a\Vert \Vert\xi\Vert_{\mathcal{H}_\Pi}\sup_n \Vert\eta_n \Vert_{\mathcal{H}_\Pi} + \Vert a\Vert \Vert\xi\Vert_{\mathcal{H}_\Pi} \Vert\eta_n - \eta \Vert_{\mathcal{H}_\Pi}\to 0, \\
&\sup_{t \in \mathbb{R}}|F_n(t+i\beta/2) - (\Pi(1\otimes(\alpha^t(a))^\mathrm{op})\xi|\eta)_{\mathcal{H}_\Pi}| \\
&\qquad\leq 
\Vert a_n-a\Vert \Vert\xi\Vert_{\mathcal{H}_\Pi}\sup_n \Vert\eta_n \Vert_{\mathcal{H}_\Pi} + \Vert a\Vert \Vert\xi\Vert_{\mathcal{H}_\Pi} \Vert\eta_n - \eta \Vert_{\mathcal{H}_\Pi} \to 0
\end{align*} 
as $n\to\infty$. Hence the Phragmen--Lindel\"{o}f method shows that the limit of $F_n(z)$ as $n\to\infty$ defines the desired function in (iv).  

(iv) $\Rightarrow$ (i): Choose an arbitrary $a \in A_{\alpha^t}^\infty$. For each $\eta \in \mathcal{H}_\Pi$,  we consider the entire function $G(z) := (\Pi(\alpha^{z-i\beta/2}(a)\otimes1^\mathrm{op})\xi|\eta)_{\mathcal{H}_\Pi}$. Let $F(z)$ be the function in (iv) for $(\alpha^{-i\beta/2}(a),\eta)$. Then, $H(z) := G(z)-F(z)$ is bounded continuous on $0\wedge(\beta/2) \leq \mathrm{Im}\,z \leq 0\vee(\beta/2)$ and holomorphic in its interior. By construction, $H(t) = 0$ for all $t \in \mathbb{R}$, and hence the Schwarz reflection principle enables us to show that $H(z)\equiv0$. Consequently, we have  
\[
(\Pi(a\otimes1^\mathrm{op})\xi|\eta)_{\mathcal{H}_\Pi} 
=
G(i\beta/2) = F(i\beta/2) 
=
(\Pi(1\otimes (\alpha^{-i\beta/2}(a))^\mathrm{op})\xi|\eta)_{\mathcal{H}_\Pi}. 
\] 
Since $\eta \in \mathcal{H}_\Pi$ is arbitrary, we finally obtain that $\Pi(a\otimes1^\mathrm{op})\xi = \Pi(1\otimes(\alpha^{-i\beta/2}(a))^\mathrm{op})\xi$. 
\end{proof} 

Here we will investigate the special case that $A=\varinjlim A_\lambda$, that is, $A$ is the $C^*$-completion of the union of a net of unital $C^*$-subalgebras $A_\lambda$, $\lambda \in \Lambda$, and moreover that each $A_\lambda$ is globally invariant under the action $\alpha^t$. In the case, we easily see that $(A_\lambda)_{\alpha^t}^\infty = A_\lambda \cap A_{\alpha^t}^\infty$ holds and is norm-dense in $A_\lambda$ for each $\lambda \in \Lambda$. At the moment, we do not have any actual application of the next fact, but we do give it for the future study. 

\begin{proposition}\label{P4.4} Let $\Pi : A^{(2)} \curvearrowright \mathcal{H}_\Pi$ be a $*$-representation. For each $\lambda \in \Lambda$ we define a closed subspace of $\mathcal{H}_\Pi$ as follows.
\[
\mathcal{H}_{\Pi\downarrow A_\lambda}^{(\alpha^t,\beta)} 
:= 
\{ \xi \in \mathcal{H}_\Pi\ ; \Pi(a\otimes1^\mathrm{op})\xi = \Pi(1\otimes(\alpha_\lambda^{-i\beta/2}(a))^\mathrm{op})\xi\ \text{for all $a \in (A_\lambda)_{\alpha^t}^\infty$} \}.
\]
Then the $\mathcal{H}_{\Pi\downarrow A_\lambda}^{(\alpha^t,\beta)}$, $\lambda\in\Lambda$, form a non-increasing net of closed subspaces and 
\[
\mathcal{H}_\Pi^{(\alpha^t,\beta)} = \bigcap_{\lambda \in \Lambda} \mathcal{H}_{\Pi\downarrow A_\lambda}^{(\alpha^t,\beta)}. 
\]
If we choose $\Pi : A^{(2)} \curvearrowright \mathcal{H}_\Pi$ to be $\Pi_\mathrm{univ} : A^{(2)} \curvearrowright \mathcal{H}_\mathrm{univ}$, then $K_\beta(\alpha^t) \neq \emptyset$ if and only if the intersection is not trivial.  
\end{proposition}
\begin{proof}
Remark that the $(A_\lambda)_{\alpha^t}^\infty = A_\lambda \cap A_{\alpha^t}^\infty$ are non-decreasing in $\lambda$ and sit in $A_{\alpha_\lambda^t}$. Hence $\mathcal{H}_{\Pi}^{(\alpha^t,\beta)} \subseteqq \bigcap_\lambda \mathcal{H}_{\Pi\downarrow A_\lambda}^{(\alpha^t,\beta)}$. On the other hand, if $\xi$ is in $\bigcap_\lambda \mathcal{H}_{\Pi \downarrow A_\lambda}^{(\alpha^t,\beta)}$, then 
\[
\Pi(a\otimes1^\mathrm{op})\xi = \Pi(1\otimes(\alpha^{-i\beta/2}(a))^\mathrm{op})\xi
\]
for all $a \in A_{\alpha^t}^\infty \cap \big(\bigcup_\lambda A_\lambda\big)$. To prove $\xi \in \mathcal{H}_\Pi^{(\alpha^t,\beta)}$, it suffices, by item (ii) of Proposition \ref{P4.3}, to prove that $A_{\alpha^t}^\infty\cap \big(\bigcup_\lambda A_\lambda\big)$ is norm dense in $A$. For any $a \in A$ there is a sequence $a_k$ in $\bigcup_\lambda A_\lambda$ such that $\Vert a_k - a \Vert \to 0$ as $k \to 0$. Each $a_k$ falls into $A_{\lambda(k)}$ for some $\lambda(k)$, and hence we can choose $b_k \in (A_{\lambda(k)})_{\alpha^t_{(\lambda(k))}}^\infty = A_{\alpha^t}^\infty \cap A_{\lambda(k)}$ so that $\Vert b_k - a_k\Vert < 1/k$. Consequently, every $b_k$ falls into $A_{\alpha^t}^\infty \cap \big(\bigcup_\lambda A_\lambda\big)$ and $\Vert b_k - a\Vert \leq 1/k + \Vert a_k -a \Vert \to 0$ as $k \to \infty$. Hence we are done.  
\end{proof}

\section{Non-continuous flows and Local normality}

We keep the setting and the notation in section 2, but we relax the continuity assumption of the flow $\alpha^t$ unlike sections 3-4 to discuss inductive limits of $W^*$-algebras. In what follows, we should employ a formulation of KMS conditions (in terms of boundary values of analytic functions) that makes sense without any continuity assumption on flows (see e.g.\ \cite[Proposition 5.3.7(2)]{BratteliRobinson:Book2}) as the definition of KMS condition.   

\subsection{Reformulation of $(\alpha^t,\beta)$-spherical representations} In the current setup, we cannot prove that the $*$-subalgebra $A_{\alpha^t}^\infty$ is norm-dense in $A$ (since we have not assumed any continuity assumption on $\alpha^t$) so that Definition \ref{D3.1} does not make sense. Hence we need to introduce a new definition of $(\alpha^t,\beta)$-spherical vectors. The definition looks not intuitive at all but is justified from Proposition \ref{P4.3}. 
 
\begin{definition}\label{D5.1} 
Let $\Pi : A^{(2)} \curvearrowright \mathcal{H}_\Pi$ be a $*$-representation. A vector $\xi \in \mathcal{H}_\Pi$ is an \emph{$(\alpha^t,\beta)$-spherical vector in $\Pi : A^{(2)} \curvearrowright \mathcal{H}_\Pi$}, if for each $a \in A$ and each $\eta \in \mathcal{H}_\Pi$ there is a bounded continuous function $F(z) = F^\xi_{a,\eta}(z)$ on $0 \wedge (\beta/2) \leq \mathrm{Im}z \leq 0 \vee(\beta/2)$ such that $F(z)$ is holomorphic in its interior and 
\[
F(t) = (\Pi(\alpha^t(a)\otimes1^\mathrm{op})\xi\,|\,\eta)_{\mathcal{H}_\Pi}, \qquad 
F(t+i\beta/2) = (\Pi(1\otimes(\alpha^t(a))^\mathrm{op})\xi\,|\,\eta)_{\mathcal{H}_\Pi} 
\]
for all $t \in \mathbb{R}$. 
\end{definition}

\begin{remark&definition}\label{D5.2} Thanks to Proposition \ref{P4.3} the above definition agrees with the previous one when the flow $\alpha^t$ is continuous. Moreover, Proposition \ref{P4.3} {\rm (iii) $\Leftrightarrow$(iv)} clearly holds even in the current setting with the same proof. Namely, it suffices to examine the existence of $F^\xi_{a,\eta}$ only for each $a \in \mathcal{A}$ and $\eta \in \mathcal{D}$ with norm-dense subsets $\mathcal{A} \subset A$ and $\mathcal{D} \subset \mathcal{H}_\Pi$. 

Hence, with $\mathcal{D} = \{\Pi(x)\xi\,;\,x \in A^{(2)}\}$ we obtain that $(\Pi : A^{(2)} \curvearrowright \mathcal{H}_\Pi, \xi)$ is an $(\alpha^t,\beta)$-spherical representation of $A$ if and only if $\varphi = \varphi_{(\Pi,\xi)}$ {\rm(}see formula \eqref{Eq3.3}{\rm)} enjoys the following property{\rm:} For each $a \in A$ and $x \in A^{(2)}$ there is a bounded continuous function $F(z) = F^\varphi_{a,x}(z)$ on $0 \wedge (\beta/2) \leq \mathrm{Im}z \leq 0 \vee(\beta/2)$ such that $F(z)$ is holomorphic in its interior and 
\[
F(t) = ((\alpha^t(a)\otimes1^\mathrm{op})\cdot\varphi)(x), \qquad 
F(t+i\beta/2) = ((1\otimes(\alpha^t(a))^\mathrm{op})\cdot\varphi)(x) 
\]
for all $t \in \mathbb{R}$. 

Remark that it is enough to confirm the above property for each $x \in A\otimes A^\mathrm{op}$, the algebraic tensor product. Finally, we define an abstract $(\alpha^t,\beta)$-spherical state as above in the present setup.  
\end{remark&definition} 

With these reformulations, Theorem \ref{T3.8} is still valid without any continuity assumption on the flow $\alpha^t$. On the other hand, we cannot prove that $S_\beta(\alpha^t)$ is weak$^*$ compact, which was crucially used in Corollary \ref{C3.9}(1), in the general case. We will give three lemmas that are necessary and sufficient to explain how Theorem \ref{T3.8} still holds without the continuity assumption. 

\begin{lemma}\label{L5.3} Let $\Pi : A^{(2)} \curvearrowright \mathcal{H}_\Pi$ be a $*$-representation and $\xi \in \mathcal{H}_\Pi$ be a unit $(\alpha^t,\beta)$-spherical vector in it. Then, the left partial state $\varphi_L$ of $\varphi := \varphi_{(\Pi,\xi)}$ defines an $(\alpha^t,\beta)$-KMS state on $A$. 
\end{lemma}
\begin{proof} For any $a, b \in A$ there are functions $F_1(z) = F^\varphi_{a,b\otimes1^\mathrm{op}}(z)$ and $F_2(z) = F^\varphi_{a^*,b^*\otimes1^\mathrm{op}}(z)$ as in Remark and Definition \ref{D5.2}. Then $F_3(z) := \overline{F_2(\overline{z-i\beta})}$ is bounded continuous on $(\beta/2)\wedge\beta \leq \mathrm{Im}z \leq (\beta/2)\vee\beta$ and holomorphic in its interior. Observe that 
\begin{align*}
F_1(t) &= \varphi((b\alpha^t(a))\otimes1^\mathrm{op})) = \varphi_L(b\alpha^t(a)), \\
F_1(t+i\beta/2) &= \overline{\varphi(b^*\otimes(\alpha^t(a^*))^\mathrm{op})} = F_3(t+i\beta/2), \\
F_3(t+i\beta) &= \overline{F_2(t)} = \overline{\varphi(b^*\alpha^t(a^*)\otimes1^\mathrm{op})} = \varphi_L(\alpha^t(a)b)
\end{align*}
for all $t \in \mathbb{R}$. The second identity enables us, by a standard method of complex analysis (see e.g.\ \cite[Lemma 9.2.11]{KadisonRingrose:Book2}), to prove that there is a common bounded continuous extension $F(z)$ of $F_1(z)$ and $F_3(z)$ to $0\wedge\beta \leq \mathrm{Im}z \leq 0\vee\beta$ such that it is holomorphic in its interior. Hence, the desired assertion follows thanks to the first and the third identities above.  
\end{proof} 

For an $\omega \in K_\beta(\alpha^t)$, the construction of its purification $\omega^{(2)}$ in section 3 is known to work without the continuity assumption of the flow $\alpha^t$ nor the use of $A^\infty_{\alpha^t}$. Let us explain this fact briefly for the reader's convenience. 

When $\beta=0$, the $\omega$ must be tracial and hence the construction of purification $\omega^{(2)}$ in section 3 trivially works without the presence of $\alpha^t$. Thus, we will assume $\beta \neq 0$ in what follows. 

Let $(\pi_\omega : A \curvearrowright \mathcal{H}_\omega, \xi_\omega)$ be the GNS triple associated with $\omega$. By the KMS condition that we have employed (see e.g.\ \cite[Proposition 5.3.7(2)]{BratteliRobinson:Book2}), the standard procedure of analytic continuation for periodic functions and Liouville's theorem shows that $\omega\circ\alpha^t = \alpha$ for all $t\in\mathbb{R}$. Thus, we can construct a pointwise strongly continuous flow $\bar{\alpha^t}$ on $\pi_\omega(A)''$ such that 
\[
\bar{\alpha}^t(\pi_\omega(a))=\pi_\omega(\alpha^t(a))
\]
for every $a\in A$ and $t\in\mathbb{R}$. The same pattern as in the proof of Proposition \ref{P4.3}(ii)$\Rightarrow$(iii) using the Phragmen--Lindel\"{o}f method (also see the proof of Lemma \ref{L5.4}(2) below for the method of approximation) enables us to confirm that the vector state of $\xi_\omega$ satisfies the $(\bar{\alpha}^t,\beta)$-KMS condition. Hence, by \cite[Corollary 5.3.9]{BratteliRobinson:Book2} (whose proof uses only \cite[Proposition 5.3.7(2)]{BratteliRobinson:Book2} as the KMS condition) we can apply Tomita's theorem to the pair $(\pi_\omega(A)'',\xi_\omega)$ and the operator $J_\omega$ in section 2 should be the modular conjugation operator associated with the pair. With this $J_\omega$ we can construct all the members that appeared around equations \eqref{Eq3.5} and \eqref{Eq3.6}. We remark that the modular automorphism group $\sigma^{\bar{\omega}}_t$ associated with $\bar{\omega}$ is given by $\sigma_t^{\bar{\omega}}=\bar{\alpha}^{-\beta t}$ for every $t\in\mathbb{R}$. This last formula trivially holds when $\beta = 0$, because $\omega$ and hence $\bar{\omega}$ are tracial in the case. 

We denote by $(\pi_\omega(A)'')_{\sigma_t^{\bar{\omega}}}^\infty$ the unital $*$-subalgebra of all $\sigma_t^{\bar{\omega}}$-analytic elements in $\pi_\omega(A)''$, which is known to be dense e.g.\ in the strong$^*$ (operator) topology on $B(\mathcal{H}_\omega)$ (see \cite[Theorem 2.4.7, Proposition 2.5.22]{BratteliRobinson:Book1}).  

\begin{lemma}\label{L5.4} We have the following assertions{\rm:} 
\begin{itemize} 
\item[(1)] $J_\omega x \xi_\omega = \sigma_{-i/2}^{\bar{\omega}}(x^*)\xi_\omega$ for all $x \in (\pi_\omega(A)'')_{\sigma_t^{\bar{\omega}}}^\infty$. 
\item[(2)] $(\pi_\omega^{(2)} : A^{(2)} \curvearrowright \mathcal{H}_\omega, \xi_\omega)$ is an $(\alpha^t,\beta)$-spherical representation of $A$. 
\item[(3)] If $\omega = \varphi_L$ with an $(\alpha^t,\beta)$-spherical state $\varphi$ of $A$, then $\omega^{(2)} = \varphi$ holds. 
\end{itemize} 
\end{lemma} 

Remark that item (3) holds in a more general context due to Woronowicz \cite{Woronowicz:RMP74} with a completely different proof from the one given below. 

\begin{proof} 
Item (1): This is well known; see e.g.\ \cite[Lemma VIII.3.18(ii)]{Takesaki:Book2}. 

\medskip
Item (2): Let $a \in A$ and $\eta \in \mathcal{H}_\omega$ be arbitrarily given. Choose an sequence $x_n \in (\pi_\omega(A)'')_{\sigma_t^{\bar{\omega}}}^\infty$ in such a way that $x_n\xi_\omega \to \pi_\omega(a)\xi_\omega$ and $x_n^*\xi_\omega \to \pi_\omega(a^*)\xi_\omega$ as $n\to\infty$. For each $n$, we consider the entire function $G_n(z) := (\sigma_z^{\bar{\omega}}(x_n)\xi_\omega\,|\,\eta)_{\mathcal{H}_\omega}$. Observe that $\Vert \sigma_z^{\bar{\omega}}(x_n)\Vert $ is bounded in $z$ on $-i/2 \leq \mathrm{Im}z \leq 0$ and so is $G_n(z)$. Since 
\[
G_n(t-i/2) = (\sigma_{-i/2}^{\bar{\omega}}(\sigma_t^{\bar{\omega}}(x_n))\xi_\omega\,|\,\eta)_{\mathcal{H}_\omega} = (J_\omega \sigma_t^{\bar{\omega}}(x_n)^*\xi_\omega\,|\,\eta)_{\mathcal{H}_\omega} 
\]
by item (1), we obtain that
\begin{align*} 
\sup_{t\in\mathbb{R}}|G_n(t) - (\sigma_t^{\bar{\omega}}(\pi_\omega(a))\xi_\omega\,|\,\eta))_{\mathcal{H}_\omega}| 
&\leq 
\Vert x_n\xi_\omega - \pi_\omega(a)\xi_\omega\Vert_{\mathcal{H}_\omega} \Vert \eta \Vert_{\mathcal{H}_\omega} \to 0, \\
\sup_{t\in\mathbb{R}}|G_n(t-i/2) - (J_\omega \sigma_t^{\bar{\omega}}(\pi_\omega(a^*))\xi_\omega\,|\,\eta)_{\mathcal{H}_\omega}| 
&\leq 
\Vert x_n^*\xi_\omega - \pi_\omega(a^*)\xi_\omega\Vert_{\mathcal{H}_\omega} \Vert\eta\Vert_{\mathcal{H}_\omega} \to 0 
\end{align*}
as $n \to \infty$, where we used that $\Vert\sigma_t^{\bar{\omega}}(x)\xi_\omega\Vert_{\mathcal{H}_\omega} = \Vert x\xi_\omega\Vert_{\mathcal{H}_\omega}$ for all $x \in \pi_\omega(A)''$. Hence the Phragmen--Lindel\"{o}f method shows that the limit of $G_n(z)$ as $n\to \infty$ defines a bounded continuous function $G(z)$ on $-1/2 \leq \mathrm{Im}z \leq 0$ such that $G(z)$ is holomorphic in its interior and 
\[
G(t) = (\sigma_t^{\bar{\omega}}(\pi_\omega(a))\xi_\omega\,|\,\eta))_{\mathcal{H}_\omega}, \qquad 
G(t-i/2) = (J_\omega \sigma_t^{\bar{\omega}}(\pi_\omega(a^*))\xi_\omega\,|\,\eta)_{\mathcal{H}_\omega}
\]
for all $t \in \mathbb{R}$. Then $F(z) := G(-z/\beta)$ is bounded continuous on $0\wedge(\beta/2) \leq \mathrm{Im}z \leq 0 \vee (\beta/2)$, holomorphic in its interior, and 
\begin{align*}
F(t) &= G(-t/\beta) = (\pi_\omega(\alpha^t(a))\xi_\omega\,|\,\eta)_{\mathcal{H}_\omega}, \\
F(t+i\beta/2) &= G(-t/\beta - i/2) = (J_\omega\pi_\omega(\alpha^t(a))^*\xi_\omega\,|\,\eta)_{\mathcal{H}_\omega}
\end{align*}
for all $t \in \mathbb{R}$. This shows the desired assertion by the definition of $\pi_\omega^{(2)} : A^{(2)} \curvearrowright \mathcal{H}_\omega$. 

\medskip
Item (3): Let $a,b \in A$ be arbitrarily given, and it suffices to prove $\omega^{(2)}(a\otimes b^\mathrm{op}) = \varphi(a\otimes b^\mathrm{op})$. By item (2), there are functions $F_1(z) = F^\varphi_{b,a\otimes1^\mathrm{op}}(z)$ and $F_2(z) = F^{\omega^{(2)}}_{b,a\otimes1^\mathrm{op}}(z)$ as in Definition \ref{D5.2}. Since $\varphi_L=\omega$, we have $F_1(t) = F_2(t)$ for all $t \in \mathbb{R}$. By the Schwarz reflection principle enables us to show $F_1(z) \equiv F_2(z)$, and in particular, $F_1(i\beta/2) = F_2(i\beta/2)$, which is nothing but the desired identity. 
\end{proof} 

The next lemma is a bit non-trivial in the present setting, and thus we give its proof for the sake of completeness.  

\begin{lemma}\label{L5.5} For any $*$-representation $\Pi : A^{(2)} \curvearrowright \mathcal{H}_\Pi$,  all the $(\alpha^t,\beta)$-spherical vectors in $\Pi : A^{(2)} \curvearrowright \mathcal{H}_\Pi$, denoted by $\mathcal{H}_\Pi^{(\alpha^t,\beta)}$ as before, form a closed subspace of $\mathcal{H}_\Pi$  that is invariant under $\Pi(A^{(2)})'$ even in the present setup. 
\end{lemma}
\begin{proof} 
It is not difficult to see that $\mathcal{H}_\Pi^{(\alpha^t,\beta)}$ is a subspace and also invariant under $\Pi(A^{(2)})'$. Hence we will only prove that it is closed. 

Assume that a sequence $\xi_n \in \mathcal{H}_\Pi^{(\alpha^t,\beta)}$ converges to a $\xi \in \mathcal{H}_\Pi$ as $n\to\infty$. Let $a \in A$ and $\eta \in \mathcal{H}_\Pi$ be arbitrarily given. For each $n$ we have a function $F_n(z) = F^{\xi_n}_{a,\eta}(z)$ as in Definition \ref{D5.1}. Then 
\begin{align*} 
\sup_{t\in\mathbb{R}} |F_n(t) - (\Pi(\alpha^t(a)\otimes1^\mathrm{op})\xi\,|\,\eta)_{\mathcal{H}_\Pi}| 
&\leq \Vert a \Vert\,\Vert\xi_n - \xi\Vert_{\mathcal{H}_\Pi}\,\Vert\eta\Vert_{\mathcal{H}_\Pi} \to 0, \\
\sup_{t\in\mathbb{R}} |F_n(t+i\beta/2) - (\Pi(1\otimes(\alpha^t(a))^\mathrm{op})\xi\,|\,\eta)_{\mathcal{H}_\Pi}| 
&\leq \Vert a \Vert\,\Vert\xi_n - \xi\Vert_{\mathcal{H}_\Pi}\,\Vert\eta\Vert_{\mathcal{H}_\Pi} \to 0. 
\end{align*}
Hence the Phragmen--Lindel\"{o}f method shows that the limit of $F_n(z)$ as $n\to\infty$ defines a function on $0\wedge(\beta/2) \leq \mathrm{Im}z \leq 0\vee(\beta/2)$ that shows that $\xi$ falls into $\mathcal{H}_\Pi^{(\alpha^t,\beta)}$. 
\end{proof}  

\subsection{Review on $W^*$-inductive limits} In the rest of this section, we will investigate $A=\varinjlim A_\lambda$, where every $A_\lambda$ is a $W^*$-algebra and any inclusion $A_\lambda \hookrightarrow A_{\lambda'}$ with $\lambda < \lambda'$ is normal, that is, $A_\lambda$ is a $W^*$-subalgebra of $A_{\lambda'}$. In this setting, it is natural \emph
{not to assume that the flow $\alpha^t$ is pointwise norm-continuous}. 

The consideration in this and the next sections is originally motivated from the recent study of inductive limits of compact quantum groups due to Ryosuke Sato \cite{Sato:preprint19} as well as the algebraic quantum field theory. In fact, this situation naturally appeared there. In both, \emph{locally normal states} are of main importance rather than arbitrary states. Here, an $\omega \in A^*$ is \emph{locally normal} if $\omega$ is normal on each $A_\lambda$. A \emph{locally normal state} means a state on $A$ that is locally normal. 

\medskip
To understand locally normal $(\alpha^t,\beta)$-KMS states, we have to recall some facts on locally normal states; see Takeda \cite{Takeda:PJA54},\cite{Takeda:TohokuMathJ55}\footnote{We would like to point out that it was Takeda \cite{Takeda:TohokuMathJ55} who first considered locally normal states to define $W^*$-inductive limits, but the proof of \cite[Lemma 2]{Takeda:TohokuMathJ55}, asserting that the locally normal states are full in the sense of \cite{Kadison:Topology65}, seems (to us) containing a gap. The gap we thought is not any trouble when a given inductive system contains a `sequential subnet'.}, Kadison \cite{Kadison:Topology65}, Haag--Kadison--Kastler \cite{HaagKadisonKastler:CMP70} and Takesaki \cite{Takesaki:PacificJMath70}, etc., mainly in the context of algebraic quantum field theory. For a technical reason\footnote{See the previous footnote.}, \emph{we will assume, in the rest of this subsection, that the $A_\lambda$ has a sequential subnet $A_{\lambda(n)}$, that is, $n \mapsto \lambda(n)$ is monotone and each $A_{\lambda'}$ is contained in $A_{\lambda(n)}$ for some $n$ {\rm(}depending on $\lambda'${\rm)}}. The assumption is trivially fulfilled in the usual setting of algebraic quantum field theory (see e.g.\ \cite[Lemma 7]{HugenholtzWierinca:CMP69}). 

Observe that any normal state $\omega$ on $A_\lambda$ can inductively be extended to $A_{\lambda(n)}$ as normal states for all sufficiently large $n$, and the sequence of normal states that we just obtained clearly defines a locally normal state on the whole $A$ by its norm-continuity. With this argument, we can easily see that any normal state on $A_\lambda$ is obtained as the restriction of a locally normal state on $A$ to $A_\lambda$,  and this fact enables us to see that the locally normal states separate points in $A$. The same thing has essentially been explained in the proof of \cite[Proposition 7]{HaagKadisonKastler:CMP70}. 

The second dual $A^{**}$ is well known to be a $W^*$-algebra whose unique predual is $A^*$. By \cite[Proposition 2]{Takesaki:PacificJMath70} there is a unique central projection $z$ of $A^{**}$ such that the unique predual $M_*$ of $M := zA^{**}$ is exactly all the locally normal linear functionals in $A^*$ via the dual pairing between $A^*$ and $A^{**}$. The bi-dual map $(\alpha^t)^{**}$ defines a flow on $A^{**}$, and we observe that if $\alpha^t(A_\lambda) = A_\lambda$ for every $t \in \mathbb{R}$ and $\lambda \in \Lambda$, then $(\alpha^t)^{**}(z) = z$ and the restriction of $(\alpha^t)^{**}$ to $M$ defines a flow $\hat{\alpha}^t$ on $M$. Since the restriction of $(\alpha^t)^{**}$ to $A$ via the embedding $A \hookrightarrow A^{**}$ is exactly the original $\alpha^t$, the unital $*$-homomorphism $A \hookrightarrow A^{**} \to zA^{**} = M$ intertwines $\alpha^t$ with $\hat{\alpha}^t$. 

Since the locally normal states on $A$ separate points in $A$ as remarked before, the unital $*$-homomorphism $A \to M$ must be injective. Thus we may think of $A$ as a $C^*$-subalgebra of $M$, and any locally normal linear functional on $A$ is obtained as just the restriction of a member of $M_*$ to $A$. Moreover, the relative topology on $A_\lambda$ induced from the $\sigma$-weak topology on $M$ is exactly the $\sigma$-weak topology on $A_\lambda$, since we have remarked that any normal state on $A_\lambda$ is obtained as the restriction of a normal state on $M$ to $A_\lambda$ in the setting here. In particular, each $A_\lambda$ is a $W^*$-subalgebra of $M$. It is also trivial that $A$ is $\sigma$-weakly dense in $M$. 

Here is a summary; we have constructed a $W^*$-algebra $M$ enlarging $A$, such that the following three assertions hold:  
\begin{itemize}
\item[(i)] $A$ is $\sigma$-weakly dense in $M$, 
\item[(ii)] each $A_\lambda$ is a $W^*$-subalgebra of $M$, 
\item[(iii)] the dual map $M^* \to A^*$ of the inclusion map $A \subset M$ sends $M_*$ onto the subspace consisting of all the locally normal linear functionals on $A$. 
\end{itemize}
Moreover, if $\alpha^t(A_\lambda) = A_\lambda$ holds for every $\lambda \in \Lambda$, then there is a unique flow $\hat{\alpha}^t$ on $M$ as an extension of $\alpha^t$. 

Remark that items (i) and (iii) show that $M_*$ is canonically identified with all the locally normal ones in $A^*$. Hence $M$ is uniquely determined. It should also be pointed out that $M$ is nothing less than the $W^*$-inductive limit of the $A_\lambda$ in the sense of Takeda  \cite{Takeda:TohokuMathJ55}. In what follows, we call this $M$ the \emph{locally normal $W^*$-envelope} of $A$. 

\subsection{Locally bi-normal $(\alpha^t,\beta)$-spherical representations} 
We keep the setting in the previous subsection. Following \cite{EffrosLance:AdvMath77} we will introduce the notion of \emph{local bi-normality} for linear functionals on $A$ as follows. A $\varphi \in (A^{(2)})^*$ is said to be \emph{locally bi-normal}, if $(a,b^\mathrm{op}) \mapsto \varphi(a\otimes b^\mathrm{op})$ is separately normal on each $A_\lambda \times A_\lambda^\mathrm{op}$. We denote by $S_\mathrm{lbin}(A^{(2)})$ all the locally bi-normal states on $A^{(2)}$. Similarly, we say that a $*$-representation $\Pi : A^{(2)} \curvearrowright \mathcal{H}_\Pi$ is \emph{locally bi-normal}, if $(a,b^\mathrm{op}) \mapsto \Pi(a\otimes b^\mathrm{op})$ is separately normal on each $A_\lambda\times A_\lambda^\mathrm{op}$. Here is a simple lemma. 

\begin{lemma}\label{L5.6} Let $M^{(2)} := M\otimes_\mathrm{bin}M^\mathrm{op}$ be the bi-normal tensor product of the locally normal $W^*$-envelope $M$ and its opposite $W^*$-algebra $M^\mathrm{op}$, and denote by $S_\mathrm{bin}(M^{(2)})$ all the bi-normal states on $M^{(2)}$, i.e., states on $M^{(2)}$ that are separately normal. {\rm(}See {\rm\cite{EffrosLance:AdvMath77}} for the notion.{\rm)} Then the following assertions hold{\rm:} 
\begin{itemize}
\item[(1)] There is a unique unital $*$-homomorphism $\rho : A^{(2)} \to M^{(2)}$ such that $\rho(x) = x$ on the algebraic tensor products $A\otimes A^\mathrm{op} \subset M\otimes M^\mathrm{op}$. 
\item[(2)] Each $\varphi \in S_\mathrm{lbin}(A^{(2)})$ has a unique $\hat{\varphi} \in S_\mathrm{bin}(M^{(2)})$ such that $\hat{\varphi}\circ\rho = \varphi$, and any member of $S_\mathrm{bin}(M^{(2)})$ arises in this way. 
\item[(3)] For any locally bi-normal $*$-representation $\Pi : A^{(2)} \curvearrowright \mathcal{H}_\Pi$, there exists a unique $*$-representation $\widehat{\Pi} : M^{(2)} \curvearrowright \mathcal{H}_\Pi$ on the same space such that 
\begin{itemize} 
\item[(i)] $\widehat{\Pi} : M^{(2)} \curvearrowright \mathcal{H}_\Pi$ is bi-normal, that is, $(a,b^\mathrm{op}) \in M \times M^\mathrm{op} \mapsto \widehat{\Pi}(a\otimes b^\mathrm{op}) \in B(\mathcal{H}_\Pi)$ is separately normal,  
\item[(ii)] $\Pi : A^{(2)} \curvearrowright \mathcal{H}_\Pi$ is factor through $M^{(2)}$ by $\rho$ and $\widehat{\Pi}$, that is, the diagram 
\[
\xymatrix{ 
A^{(2)}  \ar[rr]^{\Pi} \ar[dr]_\rho & & B(\mathcal{H}_\Pi) \\
  & M^{(2)} \ar[ur]_{\widehat{\Pi}}
}  
\]
commutes.
\end{itemize}
\end{itemize}
\end{lemma}
\begin{proof}
Item (1): Recall the canonical inclusion $A\otimes A^\mathrm{op} \subset M\otimes M^\mathrm{op}$ of algebraic tensor products, and the $A^{(2)}$ and $M^{(2)}$ are obtained as their completions with respect to appropriate norms. Since we have considered the maximal $C^*$-norm on $A\otimes A^\mathrm{op}$ to obtain $A^{(2)}$, the inclusion map $A \otimes A^\mathrm{op} \hookrightarrow M\otimes M^\mathrm{op}$ extends to $A^{(2)}$ as a unital $*$-homomorphism $\rho : A^{(2)} \to M^{(2)}$.  

\medskip
Item (2): For any $\varphi \in S_\mathrm{lbin}(A^{(2)})$, the bi-linear form $(a,b^\mathrm{op}) \in A\times A^\mathrm{op} \mapsto s_\varphi(a,b^\mathrm{op}) := \varphi(a\otimes b^\mathrm{op}) \in \mathbb{C}$ is separately, $\sigma$-weakly continuous, where $A$ and $A^\mathrm{op}$ are regarded as $\sigma$-weakly dense $C^*$-algebras of $M$ and $M^\mathrm{op}$, respectively. By \cite[Theorem 2.3]{JohnsonKadisonRingrose:BSMF72} the form $s_\varphi$ extends to a unique bi-linear form $\tilde{s}_\varphi : M\times M^\mathrm{op} \to \mathbb{C}$ that is separately normal. Thus, we have a linear map $\tilde{\varphi} : M\otimes M^\mathrm{op} \to \mathbb{C}$ defined in such a way that $\tilde{\varphi}(x\otimes y^\mathrm{op}) := \tilde{s}_\varphi(x,y^\mathrm{op})$ for any simple tensor $x\otimes y^\mathrm{op} \in M\otimes M^\mathrm{op}$. By construction, $\tilde{\varphi}$ is separately normal and $\tilde{\varphi}(1\otimes 1^\mathrm{op}) = 1$. It remains to confirm the positivity. 

Let $x_i \otimes b_i^\mathrm{op} \in M \otimes A^\mathrm{op}$, $i=1,\dots,n$, be given. By \cite[Theorem 2.4.7]{BratteliRobinson:Book1} one can choose nets $a_{i,\lambda} \in A$ such that $\Vert a_{i,\lambda} \Vert \leq \Vert x_i\Vert$ and $a_{i,\lambda} \to x_i$ in the strong$^*$ (operator) topology, where we regard $M$ as a concrete von Neumann algebra on a Hilbert space. Observe that $a_{i,\lambda}^* a_{j,\lambda} \to x_i^* x_j$ strongly, and hence $\sum_{i,j=1}^n \tilde{\varphi}(x_i^* x_j \otimes b_i^\mathrm{op}{}^* b_j^\mathrm{op}) = \lim_\lambda \sum_{i,j=1}^n \varphi(a_{i,\lambda}^* a_{j,\lambda} \otimes b_i^\mathrm{op}{}^* b_j^\mathrm{op}) \geq 0$ by the separate normality. Thus, $\tilde{\varphi}$ is positive on $M\otimes A^\mathrm{op}$. Repeating the same approximation in the second tensor component, we conclude that $\tilde{\varphi}$ is positive on the whole $M\otimes M^\mathrm{op}$. Hence, it is a bi-normal state on $M\otimes M^\mathrm{op}$. By means of bi-normal tensor products (see \cite[section 2]{EffrosLance:AdvMath77}), the state $\tilde{\varphi}$ has a unique extension $\hat{\varphi} \in S_\mathrm{bin}(M^{(2)})$ with $\hat{\varphi}\circ\rho = \varphi$. 

On the other hand, for any $\psi \in S_\mathrm{bin}(M^{(2)})$ we have a state $\psi\circ\rho$ on $A^{(2)}$. Since $\rho(x) = x$ on $A\otimes A^\mathrm{op} \subset M\otimes M^\mathrm{op}$, we easily see that $\psi\circ\rho$ falls into $S_\mathrm{lbin}(A^{(2)})$.   

\medskip
Item (3): Decomposing $\Pi : A^{(2)} \curvearrowright \mathcal{H}_\Pi$ into a direct sum of cyclic representations, we may and do assume that there is a unit cyclic vector $\xi \in \mathcal{H}_\Pi$. 

Consider the state $\varphi$ on $A^{(2)}$ defined by $\varphi(x) := (\Pi(x)\xi\,|\,\xi)_{\mathcal{H}_\Pi}$ for every $x \in A^{(2)}$. This is clearly locally bi-normal. By item (2) there exists a (unique) bi-normal state $\hat{\varphi}$ on $M^{(2)}$ with $\hat{\varphi}\circ\rho = \varphi$. 

Let $(\Pi_{\hat{\varphi}} : M^{(2)} \curvearrowright \mathcal{H}_{\hat{\varphi}},\xi_{\hat{\varphi}})$ be the GNS triple associated with $\hat{\varphi}$. Since $\hat{\varphi}$ is bi-normal, it is easy to see that $(a,b^\mathrm{op}) \in M\times M^\mathrm{op} \mapsto \Pi_{\hat{\varphi}}(a\otimes b^\mathrm{op}) \in B(\mathcal{H}_{\hat{\varphi}})$ is separately normal. Hence, $[\Pi_{\hat{\varphi}}(A\otimes A^\mathrm{op})\xi] = \mathcal{H}_{\hat{\varphi}}$, and the uniqueness of GNS representations together with $\rho(x) = x$ for every $x \in A\otimes A^\mathrm{op}$ enables us to assume that $\mathcal{H}_\Pi = \mathcal{H}_{\hat{\varphi}}$, $\xi=\xi_{\hat{\varphi}}$ and $\Pi = \Pi_{\hat{\varphi}}\circ\rho$. Consequently, $\widehat{\Pi} : = \Pi_{\hat{\varphi}} : M^{(2)} \curvearrowright \mathcal{H}_\Pi$ works well. 

The uniqueness of $\widehat{\Pi} : M^{(2)} \curvearrowright \mathcal{H}_\Pi$ immediately follows from its bi-normality and $\rho(x) = x$ for all $x \in A\otimes A^\mathrm{op}$.    
\end{proof} 

We introduce a few notations. Set 
\[
K_\beta^\mathrm{ln}(\alpha^t) := K_\beta(\alpha^t) \cap M_*, \quad S_\beta^\mathrm{lbin}(\alpha^t) := S_\beta(\alpha^t) \cap S_\mathrm{lbin}(A^{(2)}), 
\]
the locally normal $(\alpha^t,\beta)$-KMS states on $A$ and the locally bi-normal $(\alpha^t,\beta)$-spherical states of $A$, respectively. 
We also denote by $\mathrm{Rep}_\beta^\mathrm{lbin}(\alpha^t)$ all the equivalent classes of locally bi-normal $(\alpha^t,\beta)$-spherical representations of $A$. 

\begin{theorem}\label{T5.7} 
The following assertions hold true{\rm:} 
\begin{itemize} 
\item[(1)] The correspondences in Theorem \ref{T3.8} induce the bijections among $K_\beta^\mathrm{ln}(\alpha^t)$, $S_\beta^\mathrm{lbin}(\alpha^t)$ and $\mathrm{Rep}_\beta^\mathrm{lbin}(\alpha^t)$ as follows. 
\begin{equation*}
\begin{matrix} 
K_\beta^\mathrm{ln}(\alpha^t) & \longleftrightarrow & S^\mathrm{lbin}_\beta(\alpha^t) & \longleftrightarrow & \mathrm{Rep}^\mathrm{lbin}_\beta(\alpha^t) \\
\omega  & \longmapsto & \omega^{(2)} & \longmapsto & [(\pi_\omega^{(2)}=\pi_{\omega^{(2)}} : A^{(2)} \curvearrowright \mathcal{H}_\omega,\xi_\omega)], \\
(\varphi_{(\Pi,\xi)})_L = (\varphi_{(\Pi,\xi)})_R & \longmapsfrom & \varphi_{(\Pi,\xi)} & \longmapsfrom & [(\Pi : A^{(2)} \curvearrowright \mathcal{H}_\Pi,\xi)],   
\end{matrix}
\end{equation*}
and the bijection between $K_\beta^\mathrm{ln}(\alpha^t)$ and $S_\beta^\mathrm{lbin}(\alpha^t)$ is affine. 
\item[(2)] Assume that the flow $\alpha^t$ admits a {\rm(}unique{\rm)} extension $\hat{\alpha}^t$ to $M$ {\rm(}This is the case when $\alpha^t(A_\lambda) = A_\lambda$ holds for every $t\in\mathbb{R}$ and $\lambda\in\Lambda$.{\rm)} Set $K^\mathrm{n}_\beta(\hat{\alpha}^t) := K_\beta(\hat{\alpha}^t) \cap M_*$, the normal $(\hat{\alpha}^t,\beta)$-KMS states on $M$, and $S_\beta^\mathrm{bin}(\hat{\alpha}^t) := S_\beta(\hat{\alpha}^t)\cap S_\mathrm{bin}(M^{(2)})$, the bi-normal $(\hat{\alpha}^t,\beta)$-spherical states of $M$. Let $\mathrm{Rep}_\beta^\mathrm{bin}(\hat{\alpha}^t)$ be all the equivalent classes of bi-normal $(\hat{\alpha}^t,\beta)$-spherical representations of $M$. Here, we employ Definitions \ref{D5.1}-\ref{D5.2} but replace $A^{(2)}$ with $M^{(2)}$. Then, the mappings  
\[
\begin{matrix} 
K^\mathrm{n}_\beta(\hat{\alpha}^t) & \longrightarrow & K_\beta^\mathrm{ln}(\alpha^t), & \qquad & S^\mathrm{n}_\beta(\hat{\alpha}^t) & \longrightarrow & S_\beta^\mathrm{lbin}(\alpha^t) \\
\hat{\omega} & \longmapsto & \hat{\omega}\!\upharpoonright_A & & \hat{\varphi} & \longmapsto & \hat{\varphi}\circ\rho 
\end{matrix}
\]
are bijective and affine. The mapping 
\[
\begin{matrix}
\mathrm{Rep}_\beta^\mathrm{bin}(\hat{\alpha}^t) & \longrightarrow & \mathrm{Rep}_\beta^\mathrm{lbin}(\alpha^t) \\
[(\widehat{\Pi} : M^{(2)} \curvearrowright \mathcal{H}_{\widehat{\Pi}},\xi)] & \longmapsto & [\Pi := \widehat{\Pi}\circ\rho : A^{(2)} \curvearrowright \mathcal{H}_{\widehat{\Pi}},\xi)] 
\end{matrix}
\]
is also bijective. 
\end{itemize} 
\end{theorem}

Before giving the proof, we point out that item (2) above explains how the present approach includes Ryosuke Sato's work \cite{Sato:preprint19} as a special case; also see the end of subsection 9.2. Actually, he treated only $K^\mathrm{n}_\beta(\hat{\alpha}^t)$, $S^\mathrm{n}_\beta(\hat{\alpha}^t)$ and $\mathrm{Rep}_\beta^\mathrm{bin}(\hat{\alpha}^t)$ (with $\beta=-1$) instead. The present approach emphasizing the notion of local (bi-)normality is more natural, as we will see below, to discuss the spectral decomposition of KMS states (quantized characters) as well as spherical representations. Therefore, we will not touch $K^\mathrm{n}_\beta(\hat{\alpha}^t)$, $S^\mathrm{n}_\beta(\hat{\alpha}^t)$ and $\mathrm{Rep}_\beta^\mathrm{bin}(\hat{\alpha}^t)$ in the subsequent sections.   

\begin{proof} 
Item (1): As explained in subsection 5.1, Theorem \ref{T3.8} holds without the continuity of the flow $\alpha^t$. Hence it suffices to confirm all the correspondences there preserve the local (bi-)normality. 

Assume that $\omega \in K_\beta(\alpha^t)$ is locally normal. Since the $A_\lambda$, $\lambda \in \Lambda$, form a net with respect to the set-inclusion and since $\bigcup_\lambda A_\lambda$ is norm-dense in $A$, it is easy to see that the restriction of the GNS representation $\pi_\omega$ to each $A_\lambda$ is normal, i.e., continuous on the unit ball in the weak operator topologies. By their constructions, the purification $\omega^{(2)}$ and its GNS representations $\pi_\omega^{(2)} : A^{(2)} \curvearrowright \mathcal{H}_\omega$ must be locally bi-normal. The opposite direction in the correspondences is easy to treat. 

\medskip
Item (2); $K^\mathrm{n}_\beta(\hat{\alpha}^t) \to K_\beta^\mathrm{ln}(\alpha^t)$: Let $\hat{\omega} \in K^\mathrm{n}_\beta(\hat{\alpha}^t)$ be arbitrarily chosen. By property (ii) of $M$, the restriction of $\hat{\omega}$ to each $A_\lambda$ must be normal and hence $\omega := \hat{\omega}\!\upharpoonright_A$ is locally normal. This together with $\hat{\alpha}^t\!\upharpoonright_A = \alpha^t$ shows that $\omega$ falls in $K_\beta^\mathrm{ln}(\alpha^t)$.  

By property (i) of $M$ we see that $\hat{\omega}_1\!\upharpoonright_A = \hat{\omega}_2\!\upharpoonright_A$ implies $\hat{\omega}_1 = \hat{\omega}_2$ for any pair $\hat{\omega}_1, \hat{\omega}_2 \in K^\mathrm{n}_\beta(\hat{\alpha}^t)$.  Hence the map $K^\mathrm{n}_\beta(\hat{\alpha}^t) \to K_\beta^\mathrm{ln}(\alpha^t)$ is injective, and also it is trivially affine. 

Let $\omega \in K_\beta^\mathrm{ln}(\alpha^t)$ be arbitrarily chosen. By property (iii) of $M$, there is an $\hat{\omega} \in M_*$ such that $\hat{\omega}\!\upharpoonright_A = \omega$. We may and do assume that $\beta\neq0$. As remarked in the discussion above Lemma \ref{L5.4}, $\hat{\omega}\circ\hat{\alpha}^t(a) = \omega\circ\alpha^t(a) = \omega(a) = \hat{\omega}(a)$ holds for every $a \in A$, and hence $\hat{\omega}\circ\hat{\alpha}^t = \hat{\omega}$ holds by normality together with property (i) of $M$. A standard argument like the proof of Proposition \ref{P4.3} (iii) $\Rightarrow$ (iv) with the same approximation procedure as in the proof of Lemma \ref{L5.4}(2) enables us to see that $\hat{\omega}$ falls in $K^\mathrm{n}_\beta(\hat{\alpha}^t)$. (See also the proof of \cite[Proposition 5.3.7(1)$\Rightarrow$(2)]{BratteliRobinson:Book2}.)  Hence the map $K^\mathrm{n}_\beta(\hat{\alpha}^t) \to K_\beta^\mathrm{ln}(\alpha^t)$ is surjective.  

\medskip
Item (2); $S^\mathrm{n}_\beta(\hat{\alpha}^t) \to S_\beta^\mathrm{lbin}(\alpha^t)$: Let $\hat{\varphi} \in S^\mathrm{n}_\beta(\hat{\alpha}^t)$ be arbitrarily chosen. By property (ii) of $M$, the mapping $(a,b^\mathrm{op}) \in A \times A^\mathrm{op} \mapsto \hat{\varphi}\circ\rho(a\otimes b^\mathrm{op}) = \hat{\varphi}(a\otimes b^\mathrm{op}) \in \mathbb{C}$ is separately normal on each $A_\lambda \times A_\lambda^\mathrm{op}$. Hence $\hat{\varphi}\circ\rho$ is locally bi-nomal. By Remark \ref{D5.2} we also observe that $\hat{\varphi}\circ\rho$ is an $(\alpha^t,\beta)$-spherical state. Therefore, we conclude that $\hat{\varphi}\circ\rho \in S_\beta^\mathrm{lbin}(\alpha^t)$.  

Let $\hat{\varphi}_1, \hat{\varphi}_2 \in S^\mathrm{n}_\beta(\hat{\alpha}^t)$ be arbitrarily chosen. Assume that $\hat{\varphi}_1\circ\rho = \hat{\varphi}_2\circ\rho$. Then $\hat{\varphi}_1(x) = \hat{\varphi}_2(x)$ for all $x \in A\otimes A^\mathrm{op}$. By the bi-normality, $\hat{\varphi}_1$ and $\hat{\varphi}_2$ agree on $M\otimes M^\mathrm{op}$, and so do they on $M^{(2)}$ too by the norm-density of $M\otimes M^\mathrm{op}$ in $M^{(2)}$. Hence the map $S^\mathrm{n}_\beta(\hat{\alpha}^t) \to S_\beta^\mathrm{lbin}(\alpha^t)$ is injective and trivially affine.  
 
Let $\varphi \in S_\beta^\mathrm{lbin}(\alpha^t)$ be arbitrarily chosen. By item (1) there is an $\omega \in K_\beta^\mathrm{ln}(\alpha^t)$ so that $\varphi = \omega^{(2)}$. As above, there is a unique $\hat{\omega} \in K^\mathrm{n}_\beta(\hat{\alpha}^t)$ such that $\hat{\omega}\!\upharpoonright_A = \omega$ holds. Let $(\pi_{\hat{\omega}} : M \curvearrowright \mathcal{H}_{\hat{\omega}}, \xi_{\hat{\omega}})$ be the GNS triple of $\hat{\omega}$ and $J_{\hat{\omega}} : \mathcal{H}_{\hat{\omega}} \to \mathcal{H}_{\hat{\omega}}$ be the modular conjugation operator constructed from $\xi_{\hat{\omega}}$. By the `bi-normal universality' of $M^{(2)}$ we have a unique $*$-representation $\pi_{\hat{\omega}}^{(2)} : M^{(2)} \curvearrowright \mathcal{H}_{\hat{\omega}}$ defined by 
\[
\pi_{\hat{\omega}}^{(2)}(a\otimes b^\mathrm{op}) := \pi_{\hat{\omega}}(a) J_{\hat{\omega}}\pi_{\hat{\omega}}(b^*)J_{\hat{\omega}}, \qquad 
a\otimes b^\mathrm{op} \in M^{(2)}.  
\]
This is clearly separately normal, and hence uniquely determined by its restriction to $A\otimes A^\mathrm{op}$ ($\subset M\otimes M^\mathrm{op} \subset M^{(2)}$). In the exactly same way as in the proof of Lemma \ref{L5.4}(2) we see that $\hat{\varphi} = \varphi_{(\pi_{\hat{\omega}}^{(2)},\xi_{\hat{\omega}})}$ falls in $S_\beta(\hat{\alpha}^t)$. Thus, thanks to Remark \ref{D5.2}, we can confirm that $\hat{\varphi}\circ\rho$ falls in $S_\beta^\mathrm{lbin}(\alpha^t)$, because $\rho(x) = x$ on $A\otimes A^\mathrm{op}$  ($\subset M\otimes M^\mathrm{op} \subset M^{(2)}$). Since 
\[
\hat{\varphi}\circ\rho(a\otimes1^\mathrm{op}) = \hat{\varphi}(a\otimes1^\mathrm{op}) =  \hat{\omega}(a) = \omega(a), \qquad a \in A,
\]  
we have $\hat{\varphi}\circ\rho = \omega^{(2)} = \varphi$ by Lemma \ref{L5.4}(3). Hence the map $S^\mathrm{n}_\beta(\hat{\alpha}^t) \to S_\beta^\mathrm{lbin}(\alpha^t)$ is surjective. 

\medskip
Item (2); $\mathrm{Rep}_\beta^\mathrm{bin}(\hat{\alpha}^t) \to \mathrm{Rep}_\beta^\mathrm{lbin}(\alpha^t)$: Let $(\widehat{\Pi} : M^{(2)} \curvearrowright \mathcal{H}_{\widehat{\Pi}},\xi)$ be an arbitrary bi-normal $(\hat{\alpha}^t,\beta)$-spherical representation of $M$. It is clear that $(a,b^\mathrm{op}) \mapsto \Pi(a\otimes b^\mathrm{op}) := \widehat{\Pi}(\rho(a\otimes b^\mathrm{op})) = \widehat{\Pi}(a\otimes b^\mathrm{op})$ is separately normal on each $A_\lambda\times A_\lambda^\mathrm{op}$. Moreover, it is trivial that $\xi$ is an $(\alpha^t,\beta)$-spherical vector in $\Pi : A^{(2)} \curvearrowright \mathcal{H}_{\widehat{\Pi}}$ in the sense of Definition \ref{D5.1}. Since $\Pi(A\otimes A^\mathrm{op})\xi = \widehat{\Pi}(A\otimes A^\mathrm{op})\xi$ is dense in $\mathcal{H}_{\widehat{\Pi}}$ by the bi-normality of $\widehat{\Pi}$, we conclude that $(\Pi:=\widehat{\Pi}\circ\rho : A^{(2)} \curvearrowright \mathcal{H}_{\widehat{\Pi}},\xi)$ is a locally bi-normal $(\alpha^t,\beta)$-spherical representation of $A$. 

 Let $(\widehat{\Pi}_i : M^{(2)} \curvearrowright \mathcal{H}_{\widehat{\Pi}_i},\xi_i)$, $i=1,2$, be bi-normal $(\hat{\alpha}^t,\beta)$-spherical representations of $M$. If those are equivalent, then so are the corresponding $(\widehat{\Pi}_i=\widehat{\Pi}_i\circ\rho : M^{(2)} \curvearrowright \mathcal{H}_{\widehat{\Pi}_i},\xi_i)$. Conversely, we assume that $(\widehat{\Pi}_i\circ\rho : A^{(2)} \curvearrowright \mathcal{H}_{\widehat{\Pi}_i},\xi_i)$ are equivalent with unitary transform $u$. Then we have $u\widehat{\Pi}(a\otimes b^\mathrm{op}) = \widehat{\Pi}(a\otimes b^\mathrm{op})u$ for every simple tensor $a\otimes b^\mathrm{op} \in A\otimes A^\mathrm{op}$ ($\subset M^{(2)}$). By the bi-normality, we have $u\widehat{\Pi}(a\otimes b^\mathrm{op}) = \widehat{\Pi}(a\otimes b^\mathrm{op})u$ for every simple tensor $a\otimes b^\mathrm{op} \in M\otimes M^\mathrm{op}$ ($\subset M^{(2)}$). By the norm-continuity, we conclude that $u\widehat{\Pi}(x)=\widehat{\Pi}(x)u$ holds for every $x \in M^{(2)}$. Consequently, the mapping $[(\widehat{\Pi} : M^{(2)} \curvearrowright \mathcal{H}_{\widehat{\Pi}},\xi)] \mapsto [\Pi := \widehat{\Pi}\circ\rho : A^{(2)} \curvearrowright \mathcal{H}_{\widehat{\Pi}},\xi)]$ is well defined and injective. Therefore, it suffices to prove only the surjectivity. 
 
 Let $(\Pi : A^{(2)} \curvearrowright \mathcal{H}_\Pi,\xi)$ be a locally bi-normal $(\alpha^t,\beta)$-spherical representation of $A$. By item (1), we have $\varphi := \varphi_{(\Pi,\xi)} \in S_\beta^\mathrm{lbin}(\alpha^t)$. Then, what we have proved above shows that there is a unique $\hat{\varphi} \in S_\beta^\mathrm{bin}(\hat{\alpha}^t)$ with $\hat{\varphi}\circ\rho = \varphi$. Then, the GNS triple $(\Pi_{\hat{\varphi}} : M^{(2)} \curvearrowright \mathcal{H}_{\hat{\varphi}},\xi_{\hat{\varphi}})$ must be a bi-normal $(\hat{\alpha}^t,\beta)$-spherical representation of $M$. As in the proof of Lemma \ref{L5.6}(3), the uniqueness of GNS representations implies that $[(\Pi_{\hat{\varphi}}\circ\rho : A^{(2)} \curvearrowright \mathcal{H}_{\hat{\varphi}},\xi_{\hat{\varphi}})] = [(\Pi : A^{(2)} \curvearrowright \mathcal{H}_\Pi,\xi)]$. Hence we are done. 
\end{proof} 

We are closing this section with a fact on the geometry of $K_\beta^\mathrm{ln}(\alpha^t)$. 

\begin{proposition} \label{P5.8} 
The locally normal $(\alpha^t,\beta)$-KMS states $K_\beta^\mathrm{ln}(\alpha^t)$ form a face of all the $(\alpha^t,\beta)$-KMS states $K_\beta(\alpha^t)$. In particular, a locally normal $(\alpha^t,\beta)$-KMS state $\omega$ is an extreme point in $K_\beta^\mathrm{ln}(\alpha^t)$ if and only if $\omega$ is a factor state. 
\end{proposition} 
\begin{proof} Assume that a convex combination $\omega = \sum_{i=1}^n \lambda_i \omega_i$ in $K_\beta(\alpha^t)$ falls in $K_\beta^\mathrm{ln}(\alpha^t)$. Then, for each $1 \leq i \leq n$, $\lambda_i \omega_i \leq \omega$ and hence there is a positive $h'_i \in \pi_\omega(A)'$ so that $\lambda_i \omega_i(a) = (\pi_\omega(a)\xi_\omega\,|\,h'_i \xi_\omega)_{\mathcal{H}_\omega}$ for all $a \in A$, where $(\pi_\omega : A \curvearrowright \mathcal{H}_\omega,\xi_\omega)$ is the GNS triple associated with $\omega$. See \cite[Theorem 2.3.19]{BratteliRobinson:Book1}. It is easy to see that the restriction of $\pi_\omega$ to each $A_\lambda$ is normal. (See the proof of Theorem \ref{T5.7}(1).) Hence $\lambda_i \omega_i$ and thus $\omega_i$ are locally normal. Hence we have proved the first part. 

The second part immediately follows from \cite[Theorem 5.3.30(3)]{BratteliRobinson:Book2}, which does not need the continuity assumption of the flow $\alpha^t$.        
\end{proof} 

\section{Viewpoint of Gelfand pairs} 

Olshanski formulated the notion of Gelfand pairs for general topological groups as follows. Let $G > H$ be a pair of topological group and its closed subgroup. He calls $(G,H)$ a \emph{Gelfand pair}, if for any unitary representation $\pi : G \curvearrowright \mathcal{H}_\pi$, the projection $e$ onto $\mathcal{H}_\pi^H := \{\xi \in \mathcal{H}\,; \pi(h)\xi = \xi\ \text{for all $h \in H$}\}$ satisfies that   
\[
e\pi(g_1)e\pi(g_2)e = e\pi(g_2)e\pi(g_1)e 
\]
for all $g_1, g_2 \in G$. It is known, see \cite{Olshanski:AdvSovietMath91}, that for any topological group $G$, $(G\times G, \Delta(G))$ with diagonal subgroup $\Delta(G)$ becomes a Gelfand pair. 

\medskip
In this section, we will establish the same property for $C^*$-flows in any  representations. In what follows, we do not assume any continuity of the flow $\alpha^t$ in question as in section 5 so that we employ Definition \ref{D5.1} here.  

\begin{theorem}\label{T6.1} Let $\Pi : A^{(2)} \curvearrowright \mathcal{H}_\Pi$ be a $*$-representation, and $e$ be the projection onto all the $(\alpha^t,\beta)$-spherical vectors 
$\mathcal{H}_\Pi^{(\alpha^t,\beta)}$ in $\Pi : A^{(2)} \curvearrowright \mathcal{H}_\Pi$. Then,   
\[
e\Pi(x)e\Pi(y)e = e\Pi(y)e\Pi(x)e
\]
holds for every pair $x,y \in A^{(2)}$. 

Moreover, if $\Pi : A^{(2)} \curvearrowright \mathcal{H}_\Pi$ is an $(\alpha^t,\beta)$-spherical representation with $(\alpha^t,\beta)$-spherical vector $\xi$, then the projection $e$ is exactly the Jones projection associated with the faithful normal conditional expectation $E_{\mathcal{Z}(L(\Pi))}$ from $L(\Pi) : = \Pi(A\otimes\mathbb{C}1^\mathrm{op})''$ onto the center $\mathcal{Z}(L(\Pi))$ with respect to the vector $\xi$, that is, $e = [\mathcal{Z}(L(\Pi))\xi]$, so that $e x\xi = E_{\mathcal{Z}(L(\Pi))}(x)\xi$ holds for every $x \in L(\Pi)$. In particular, $e\Pi(A^{(2)})''e = \mathcal{Z}(L(\Pi))e$ holds.  
\end{theorem}    

It should be noticed that the latter assertion is related to \cite[Theorem 4.2.4]{BratteliRobinson:Book1}. 

\medskip
We remark that the theorem gives a more conceptual explanation of Corollary \ref{C3.9}(3). Namely, $\mathcal{H}_\Pi^{(\alpha^t,\beta)} = [\mathcal{Z}(L(\Pi))\xi]$ immediately explains everything there. Moreover, $e\Pi(A^{(2)})e$ is nothing but (almost) the center of $\mathcal{Z}(M)$. This is quite natural. In fact, when $G$ is a finite group, any function $f$ on $\Delta(G)\backslash G\times G /\Delta(G)$ gives and arises from a central function on $G$, and the central functions on $G$ with convolution product form the center of the group algebra of $G$. Thus, $e\Pi_\mathrm{univ}(A^{(2)})e$ should be understood as an analogue of $\Delta(G)\backslash G\times G /\Delta(G)$. By means of Lemma \ref{L4.1}, the `dual' of embedding $\Delta(G) \hookrightarrow G\times G$ has two non-commutative analogues: $A^{(2)} \ni a\otimes1^\mathrm{op}-1\otimes(\alpha^{-i\beta/2}(a))^\mathrm{op} \mapsfrom a \in A_{\alpha^t}^\infty$ from the right and $a \in A_{\alpha^t}^\infty \mapsto \alpha^{-i\beta/2}(a)\otimes1^\mathrm{op}-1\otimes a^\mathrm{op} \in A^{(2)}$ from the left. (Remark that these two analogues are the same under no presence of flow.) In this point of view, Woronowicz's theorem, which was re-produced in subsection 4.1 of this paper, can be interpreted as an analogue of the non-triviality of $(G\times G)/\Delta(G)$ `in the universal representation'. This is quite natural because $G\times G \curvearrowright (G\times G)/\Delta(G)$ is naturally identified with the `two-sided action' of $G\times G$.  

\medskip
The proof needs several facts that seem to be of independent interest. In the next lemma and proposition, we assume that $(L,\mathcal{H},J,\mathcal{P})$ is a standard form of a $\sigma$-finite von Neumann algebra $L$, and $\xi \in \mathcal{P}$ be a cyclic (and automatically separating) vector. Denote by $\omega_\xi$ the vector state on $L$ of $\xi$ and by $\sigma_t^{\omega_\xi}$ its modular automorphism. 

\begin{lemma}\label{L6.2} For a $\zeta \in \mathcal{H}$ the following are equivalent{\rm:} 
\begin{itemize} 
\item[(i)] For any $x \in L$ and $\eta \in \mathcal{H}$, there exists a bounded continuous function $F(z) = F_{x,\eta}(z)$ on $-1/2 \leq \mathrm{Im}z \leq 0$ such that $F(z)$ is holomorphic in its interior and 
\[
F(t) = (\sigma_t^{\omega_\xi}(x)\zeta\,|\,\eta)_\mathcal{H}, \qquad F(t-i/2) = (J\sigma_t^{\omega_\xi}(x)^*J\zeta\,|\,\eta)_\mathcal{H}
\]
for all $t \in \mathbb{R}$. 
\item[(ii)] the above {\rm(i)} holds when $L$ is replaced with a certain $\sigma$-weakly dense subspace $L_0$ of $L$. 
\item[(iii)] For any $\sigma_t^{\omega_\xi}$-analytic element $x$ of $L$,  
\[
x\zeta = J\sigma_{i/2}^{\omega_\xi}(x)^*J\zeta
\]
holds. 
\end{itemize}    
\end{lemma} 
\begin{proof}  
(iii) $\Rightarrow$ (ii) is easy to confirm with choosing $L_0$ to be the $*$-subalgebra consisting of all $\sigma_t^{\omega_\xi}$-analytic elements of $L$. (i) $\Rightarrow$ (iii) is shown in the same way as in the proof of Proposition \ref{P4.3} (iv) $\Rightarrow$ (i), and the detail is left to the reader. Thus, we need to prove only (ii) $\Rightarrow$ (i). 

The following discussion is based on the same idea as in the proof of Lemma \ref{L5.4}(2), but some care of approximation is necessary because we cannot directly confirm that $\Vert\sigma_t^{\omega_\xi}(x)\zeta\Vert_\mathcal{H} = \Vert x\zeta\Vert_\mathcal{H}$ for $x \in L$, etc. 

By \cite[Theorem 2.4.7]{BratteliRobinson:Book1}, $L_0$ is dense in $L$ in the strong$^*$ (operator) topology. For any $x \in L$ one can choose $x_n \in L_0$ so that $x_n\xi \to x\xi$ and $x_n^*\xi \to x^*\xi$ as $n\to\infty$. By assumption, for each $\sigma^{\omega_\xi}_t$-analystic $y \in L$ there is a bounded continuous function $F_n(z)$ on $-i/2 \leq \mathrm{Im}z \leq 0$ such that $F_n(z)$ is holomorphic in its interior and 
\[
F_n(t) = (\sigma_t^{\omega_\xi}(x_n)\zeta\,|\,y\xi)_\mathcal{H}, \quad F_n(t-i/2) = (J\sigma_t^{\omega_\xi}(x_n)^*J\zeta\,|\,y\xi)_\mathcal{H}
\] 
for all $t \in \mathbb{R}$. Then, we have 
\begin{align*} 
|F_n(t) -  (\sigma_t^{\omega_\xi}(x)\zeta\,|\,y\xi)_\mathcal{H}| 
&\leq 
\Vert \zeta\Vert_\mathcal{H}\,\Vert J\sigma_{i/2}^{\omega_\xi}(y)^*J\sigma_t^{\omega_\xi}(x_n^*-x^*)\xi\Vert_\mathcal{H} \quad \text{(see Lemma \ref{L5.4}(1))} \\
&\leq 
\Vert \zeta\Vert_\mathcal{H}\,\Vert\sigma_{i/2}^{\omega_\xi}(y)\Vert_\mathcal{H}\,\Vert x_n^*\xi - x^*\xi\Vert_\mathcal{H} \to 0, \\
|F_n(t-i/2) -  (J\sigma_t^{\omega_\xi}(x)^*J\zeta\,|\,y\xi)_\mathcal{H}| 
&\leq 
\Vert \zeta\Vert_\mathcal{H}\,\Vert y\Vert_\mathcal{H}\,\Vert x_n\xi - x\xi\Vert_\mathcal{H} \to 0, 
\end{align*} 
and thus we can conclude that (i) holds for only $\eta$ of the special form $\eta= y\xi$ with $\sigma^{\omega_\xi}_t$-analytic $y \in L$ as in the proof of Proposition \ref{P4.3} using the Phragmen--Lindel\"{o}f method. 

Since the elements $y\xi$ with $\sigma^{\omega_\xi}_t$-analytic $y \in L$ form a  dense subspace of $\mathcal{H}$, it is easy to relax the above restriction of $\eta$ by an approximation procedure based on the Phragmen--Lindel\"{o}f method again.  
\end{proof} 

The next proposition shows how the Tomita--Takesaki theory is powerful. 

\begin{proposition} \label{P6.3} Assume that $\zeta \in \mathcal{H}$ enjoys the equivalent conditions of Lemma \ref{L6.2}. Then, the $\zeta$ belongs to $[\mathcal{Z}(L)\xi]$. 
\end{proposition}
\begin{proof}
Consider the vector states $\omega'_\xi$, $\omega'_\zeta$ on $L' = JLJ$ of $\xi$, $\zeta$, respectively. Then, it is easy to see that the modular automorphism $\sigma_t^{\omega'_\xi}$ of $\omega'_\xi$ enjoys 
\[
\sigma_t^{\omega'_\xi}(x') = J \sigma_t^{\omega_\xi}(Jx'J)J, \quad x' \in L'. 
\]

By Lemma \ref{L6.2}(iii) the support projection $s(\omega'_\zeta)$ of $\omega'_\zeta$ becomes $[L\zeta] = [L'\zeta]$, falling into the center $\mathcal{Z}(L) = \mathcal{Z}(L')$. We write $s := s(\omega'_\zeta)$ for simplicity. For any $x' \in L'$ and any $\sigma_t^\xi$-analytic $y \in L$, letting $y' := Jy^* J \in L'$, we have 
\begin{align*}
\omega'_\zeta(x' y') 
&= 
(x' Jy^* J\zeta\,|\,\zeta)_\mathcal{H} \\
&= 
(x' \sigma_{-i/2}^{\omega_\xi}(y)\zeta\,|\,\zeta)_\mathcal{H} \quad \text{(by Lemma \ref{L6.2}(iii))}\\
&= 
(x'\zeta\,|\,\sigma_{i/2}^{\omega_\xi}(y^*)\zeta)_\mathcal{H} \\
&= 
(x'\zeta\,|\,J\sigma_{-i}^{\omega_\xi}(y)J\zeta)_\mathcal{H} \quad \text{(by Lemma \ref{L6.2}(iii) again)} \\
&= 
(J\sigma_{-i}^{\omega_\xi}(y)^*J x'\zeta\,|\,\zeta)_\mathcal{H} \\
&= 
(J\sigma_{i}^{\omega_\xi}(Jy'J)J x'\zeta\,|\,\zeta)_\mathcal{H}.
\end{align*}   
It follows, by Takesaki's characterization, that $\sigma_t^{\omega'_\zeta} = \sigma_t^{\omega'_\xi}$ on $sL'$. 

Consider $\omega'_\zeta \leq \omega'_\zeta + \omega'_\xi$ on $sL'$ ($= (sL)'$ with center $s\mathcal{Z}(L)$). Observe that the modular automorphism of $\omega'_\zeta + \omega'_\xi$ on $sL'$ is also given by the restriction of $\sigma_t^{\omega'_\xi}$ to $sL'$. Hence, by \cite[Theorem 5.4]{PedersenTakesaki:ActaMath73} (actually the first part of its proof), we obtain a contractive, positive operator $h \in s\mathcal{Z}(L)$ such that $\omega'_\zeta(x') = \omega'_\zeta(hx') + \omega'_\xi(hx')$ holds for every $x' \in L'$ ({\it n.b.}, $hs = h$ and $s = s(\omega'_\zeta)$) and $0 \leq s - h \leq s$. 

Let $\zeta = u|\zeta|$ be the polar decomposition in the standard form. By construction, $|\zeta|$ is the unique implementing vector in $\mathcal{P}$ of $\omega'_\zeta$ and $u \in L$ with $u^*u=s$. See e.g.\ \cite[Exercises IX.1]{Takesaki:Book2}. Hence we obtain that 
\[
(x'(s-h)^{1/2}|\zeta|\,|\,(s-h)^{1/2}|\zeta|)_\mathcal{H} 
= (x'h^{1/2}\xi\,|\,h^{1/2}\xi)_\mathcal{H}
\]
for every $x' \in L'$. Since $s-h$ and $h$ are positive elements of $\mathcal{Z}(L)$, we see that $(s-h)^{1/2}|\zeta| = h^{1/2}\xi$ in $\mathcal{P}$ by the uniqueness of representing vectors (see e.g.\ \cite[Theorem 2.5.31]{BratteliRobinson:Book1}).  

Denote $e := \mathbf{1}_{\{1\}}(h) \in \mathcal{Z}(L)$. Since $\Vert e\xi\Vert_\mathcal{H}^2 = (e(s-h)\zeta\,|\,\zeta)_\mathcal{H} = 0$,  we have $e=0$ and hence $\mathbf{1}_{(0,1]}(s-h) = s$. Consequently, we obtain that  
\[
|\zeta| = s|\zeta| = \lim_{\varepsilon\searrow0} ((s-h)+\varepsilon1)^{-1/2}(s-h)^{1/2}|\zeta| = \lim_{\varepsilon\searrow0} ((s-h)+\varepsilon1)^{-1/2}h^{1/2}\xi \in [\mathcal{Z}(L)\xi]. 
\]

It is easy to see, by the use of $|\zeta| \in [\mathcal{Z}(L)\xi]$, that $x|\zeta| = J\sigma_{i/2}^{\omega_\xi}(x)^* J|\zeta|$ for any $\sigma_t^\xi$-analytic $x \in M$. For any $\sigma_t^{\omega_\xi}$-analytic $x \in L$, we have
\[
xu |\zeta| = x\zeta = J\sigma_{i/2}^{\omega_\xi}(x)^* J\zeta = u J\sigma_{i/2}^{\omega_\xi}(x)^*J|\zeta| = ux|\zeta|
\]  
and hence $xu = ux u^*u = uu^*u x = ux$, since $u^* u = s \in \mathcal{Z}(L)$. This means that $u$ falls into $\mathcal{Z}(L)$, and hence $\zeta = u|\zeta|$ does into $[\mathcal{Z}(L)\xi]$. 
\end{proof}

We are ready to prove Theorem \ref{T6.1}. 

\begin{proof} (Theorem \ref{T6.1})  
Since $\mathcal{H}_\Pi^{(\alpha^t,\beta)}$ is invariant under $\Pi(A^{(2)})'$, the standard exhaustion argument enables us to find a family of unit vectors $\xi_i$, $i \in I$, in $\mathcal{H}_\Pi^{(\alpha^t,\beta)}$ so that 
\[ 
[\Pi(A^{(2)})\mathcal{H}_\Pi^{(\alpha^t,\beta)}] = \bigoplus_{i\in I} [\Pi(A^{(2)})\xi_i]. 
\]
Notice that the projection onto each $[\Pi(A^{(2)})\xi_i]$ falls into $\Pi(A^{(2)})'$. Thus, to prove the first part of the statement, we may and do assume that $\Pi : A^{(2)} \curvearrowright \mathcal{H}_\Pi$ is an $(\alpha^t,\beta)$-spherical representation with a vector $\xi$. Namely, it suffices to prove only the latter half of the statement. 

As explained and claimed in subsection 5.1 we can still apply Theorem \ref{T3.8} to $(\Pi : A^{(2)} \curvearrowright \mathcal{H}_\Pi,\xi)$ without the continuity of the flow $\alpha^t$. Namely, we may and do assume that $(\Pi : A^{(2)} \curvearrowright \mathcal{H}_\Pi,\xi) = (\pi_\omega^{(2)} : A^{(2)} \curvearrowright \mathcal{H}_\omega,\xi_\omega)$ with $\omega \in K_\beta(\alpha^t)$. Here and below, the notations around equation \eqref{Eq3.5} (with the notations appeared just above Lemma \ref{L5.4}) are and will be used. 

Notice that $L(\Pi) = \pi_\omega(A)''$, and $(L(\Pi),\mathcal{H}_\Pi,J,\mathcal{P})$ with $\mathcal{P} = \overline{\{ x J x \xi\,; x \in L(\Pi)\}}$ becomes a standard form (see e.g.\ \cite[subsection 2.5.4]{BratteliRobinson:Book1}, where $J:= J_\omega$ denotes the modular conjugation; see subsection 5.1. Recall that the canonical `extension' $\bar{\omega}$ of $\omega$ to $M$ is the vector state $\omega_\xi$ and that 
\[
\Pi(a\otimes b^\mathrm{op}) = \pi_\omega(a)J\pi_\omega(b)^* J, \qquad 
\sigma_t^{\omega_\xi}(\pi_\omega(a)) = \pi_\omega(\alpha^{-\beta t}(a))
\]
hold for any $a, b \in A$. Moreover, we observe that 
\[
\Pi(A^{(2)})'' = L(\Pi) \vee JL(\Pi)J = L(\Pi)\vee L(\Pi)', \qquad \mathcal{H}_\Pi = [L(\Pi)\xi].  
\]

Let $\zeta \in \mathcal{H}_\Pi^{(\alpha^t,\beta)}$ be arbitrarily given. We can confirm condition (ii) of Lemma \ref{L6.2} from Definition \ref{D5.1} with $\zeta$. Thus, Proposition \ref{P6.3} shows that $\zeta \in [\mathcal{Z}(L(\Pi))\xi]$. Since $\mathcal{Z}(L(\Pi)) = (L(\Pi)\vee L(\Pi)')' = \Pi(A^{(2)})'$, it is easy to see $[\mathcal{Z}(L(\Pi))\xi] \subseteq \mathcal{H}_\Pi^{(\alpha^t,\beta)}$. Consequently, $\mathcal{H}_\Pi^{(\alpha^t,\beta)} = [\mathcal{Z}(L(\Pi))\xi]$, and by construction  the projection $e$ must coincide with the Jones projection associated with the faithful normal conditional expectation $E_{\mathcal{Z}(L(\Pi))}$ from $L(\Pi)$ onto $\mathcal{Z}(L(\Pi))$ with respect to $\bar{\omega}=\omega_\xi$ (see e.g.\ \cite[page 130]{Kosaki:JFA86}). 

Since Lemma \ref{L6.2}(iii) and $exe = E_{\mathcal{Z}(L(\Pi))}(x)e$ (see e.g.\ \cite[Lemma 3.2]{Kosaki:JFA86}), we have 
\begin{equation}\label{Eq6.1}
e x Jy^* Je 
= 
ex \sigma_{-i/2}^{\omega_\xi}(y)e 
= 
E_{\mathcal{Z}(M)}(x\sigma_{-i/2}^{\omega_\xi}(y))e \in \mathcal{Z}(L(\Pi))e
\end{equation}
if $x \in M$ is arbitrary and $y$ is a $\sigma_t^{\omega_\xi}$-analytic element of $L(\Pi)$. Therefore, we conclude that $e\Pi(A^{(2)})''e = \mathcal{Z}(L(\Pi))e$. 
\end{proof} 

We remark that no general construction of $\alpha^t$-analytic elements \emph{inside $A$} is available without the continuity assumption on the flow $\alpha^t$. However, in the setting of Theorem \ref{T5.7}(2), $L(\Pi) = \widehat{\Pi}(M\otimes\mathbb{C}1^\mathrm{op})$ holds and the extension $\hat{\alpha}^t$ possesses a $\sigma$-weakly dense unital $*$-subalgebra consisting of $\hat{\alpha}^t$-analytic elements of $M$. Moreover, formula \eqref{Eq6.1} is nothing less than 
\begin{equation}\label{Eq6.2}
e\widehat{\Pi}(x\otimes y^\mathrm{op})e = E_{\mathcal{Z}(L(\Pi))}(\widehat{\Pi}(x\hat{\alpha}^{i\beta/2}(y)\otimes1^\mathrm{op}))e 
= 
E_{\mathcal{Z}(L(\Pi))}(\widehat{\Pi}(1\otimes(\hat{\alpha}^{-i\beta/2}(x)y)^\mathrm{op}))e
\end{equation}
for all $\hat{\alpha}^t$-analytic $x, y \in M$. Probably, this formula is more natural than formula \eqref{Eq6.1} in the setting. 

\medskip
Here is a definition. 

\begin{definition}\label{D6.4} A $*$-representation $\Pi : A^{(2)} \curvearrowright \mathcal{H}_\Pi$ is said to be of $(\alpha^t,\beta)$-spherical type, if a unit cyclic vector in $\mathcal{H}_\Pi^{(\alpha^t,\beta)}$ exists. 
\end{definition} 

Assume that two $(\alpha^t,\beta)$-spherical type representations $\Pi_i : A^{(2)} \curvearrowright \mathcal{H}_{\Pi_i}$, $i=1,2$, are unitarily equivalent. Namely, we have a unitary transform $u : \mathcal{H}_{\Pi_1} \to \mathcal{H}_{\Pi_2}$ so that $u\Pi_1(x) = \Pi_2(x)u$ holds for every $x \in A^{(2)}$. Since those $*$-representations are of $(\alpha^t,\beta)$-spherical type,  they have cyclic unit vectors $\xi_i \in \mathcal{H}_{\Pi_i}^{(\alpha^t,\beta)}$. Then, the $(\alpha^t,\beta)$-spherical representations $(\Pi_1 : A^{(2)} \curvearrowright \mathcal{H}_{\Pi_1},\xi_1)$ and $(\Pi_2 : A^{(2)} \curvearrowright \mathcal{H}_{\Pi_2},u\xi_1)$ are equivalent in the sense of Definition \ref{D3.2}, but $u\xi_1 \neq \xi_2$ is possible. On the other hand, Theorem \ref{T6.1} and its proof show that the $u$ sends $\mathcal{H}_{\Pi_1}^{(\alpha^t,\beta)} = [\mathcal{Z}(L(\Pi_1))\xi_1]$ onto $\mathcal{H}_{\Pi_2}^{(\alpha^t,\beta)} = [\mathcal{Z}(L(\Pi_2))\xi_2]$. We have the $(\alpha^t,\beta)$-spherical state identity $\varphi_{(\Pi_1,\xi_1)} = \varphi_{(\Pi_2,u\xi_1)}$ and hence the $(\alpha^t,\beta)$-KMS state identity $(\varphi_{(\Pi_1,\xi_1)})_L = (\varphi_{(\Pi_2,u\xi_1)})_L$ with $u\xi_1 \in [\mathcal{Z}(L(\Pi_2))\xi_2]$ cyclic for $L(\Pi_2)$. A further investigation will be carried out under the local bi-normality of $*$-representations plus a certain assumption on $A$ in section 8.  

\section{Ergodic method} 

The aim of this section is to upgrade several existing, important general facts on extreme characters to locally normal KMS states in the most general setup. The key fact among those is \emph{the approximation theorem} for extreme characters of inductive limit groups and its connection to the theory of Martin boundaries. The method dates back to Vershik \cite{Vershik:Doklady74} in essence. It was Vershik--Kerov \cite{VershikKerov:FunctAnalAppl81} who first established the theorem in the analysis of extreme characters and coined \emph{Ergodic method} for it; see \cite[chapeter 1, section 1]{Kerov:Book}. Then, the idea of approximation theorem was further developed by many hands; for example, Olshanski \cite[section 22]{Olshanski:Proc90} using convex analysis and in his spherical representation theory, Okounkov--Olshanski \cite[section 6]{OkounkovOlshanski:IMRN98} giving a rather abstract formulation with convex analysis proof, Olshanski \cite[section 10]{Olshanski:JFA03} giving a detailed account of Vershik--Kerov's work in the case of $\mathrm{U}(\infty) = \varinjlim \mathrm{U}(N)$, Ryosuke Sato \cite[subsection 2.5]{Sato:JFA19} in the quantum group setup, and these names are only few. We also point out that the study of this section clarifies an overlap between those works and  \cite{Kishimoto:RepMathPhys00} due to Kishimoto. 

\medskip
Throughout this section, we will study locally normal KMS states in a special setup. Namely, $A = \varinjlim A_n$ is the inductive limit $C^*$-algebra of a sequence of atomic $W^*$-algebras with separable predual, where each inclusion map $A_n \hookrightarrow A_{n+1}$ is normal and $A_0 = \mathbb{C}1$. In what follows, we denote by $\mathfrak{Z}_n$ all the minimal projections in the center $\mathcal{Z}(A_n)$, which is, by definition, an at most  countably infinite set.  Note that $\mathfrak{Z}_0 = \{1\}$, a singleton. Accordingly, we have the central decomposition 
\[
A_n = \bigoplus_{z \in \mathfrak{Z}_n} zA_n, \quad zA_n  = B(\mathcal{H}_z).  
\]
with some separable Hilbert spaces $\mathcal{H}_z$, $z \in \mathfrak{Z}_n$. 

We do not assume that the flow $\alpha^t$ in question is not  pointwise norm-continuous here, but assume that it is an inductive flow, that is, $\alpha^t(A_n) = A_n$ holds for every $t\in\mathbb{R}$ and $n \in \mathbb{N}$. Hence we have the local flow $\alpha_n^t$, the restriction of $\alpha^t$ to $A_n$, for each $n$. 

\medskip
We begin by proving the next lemma for the flow $\alpha^t$. 

\begin{lemma}\label{L7.1} The following are equivalent{\rm:} 
\begin{itemize}
\item[(i)] For each $a \in A$, $t \mapsto \alpha^t(a)$ is continuous with respect to the weak topology on $A$ induced from the locally normal linear functionals $M_*$.  
\item[(ii)] For each $n$, $t \mapsto \alpha_n^t$ is continuous in the $u$-topology on $A_n$, that is, $\Vert \omega\circ\alpha_n^t - \omega\Vert \to 0$ as $t\to 0$ for all $\omega \in A_n{}_*$.  
\end{itemize}
Under these equivalent conditions, the local flow $\alpha_n^t$ must trivially act on the center $\mathcal{Z}(A_n)$ of each $A_n$.  
\end{lemma}
\begin{proof} 
(i) $\Rightarrow$ (ii): Let $n$ be arbitrarily chosen. We observe that, for every $a \in A_n$, $t \mapsto \alpha_n^t(a)$ is continuous with respect to the weak topology induced from the predual $(A_n)_*$. By \cite[Corollary 2.5.23]{BratteliRobinson:Book1} we conclude that $t \mapsto \alpha_n^t$ is continuous in the $u$-topology on $A_n$.  

(ii) $\Rightarrow$ (i): Let $a \in A$ be arbitrarily chosen. For each $\varepsilon > 0$ there are an $n_0$ and an $a_0 \in A_{n_0}$ such that $\Vert a - a_0\Vert < \varepsilon$. For every $\omega \in M_*$ we denote by $\omega_{n_0}$ the restriction of $\omega$ to $A_{n_0}$, and obtain that 
\begin{align*} 
|\omega(\alpha^t(a))-\omega(a)| 
&\leqq 
2\Vert\omega\Vert\,\Vert a -a_0\Vert + |\omega_{n_0}(\alpha_{n_0}^t(a_0))-\omega_{n_0}(a_0)| \\
&\leqq 
2\Vert\omega\Vert\varepsilon + \Vert\omega_{n_0}\circ\alpha_{n_0}^t-\omega_{n_0}\Vert(\Vert a\Vert + \varepsilon) 
\end{align*}
and hence 
\[
\limsup_{t\to0}|\omega_{n_0}(\alpha_{n_0}^t(a))-\omega_{n_0}(a)| \leqq 2\Vert\omega\Vert\varepsilon. 
\]
Since $\varepsilon>0$ is arbitrary, $|\omega_{n_0}(\alpha_{n_0}^t(a))-\omega_{n_0}(a)| \to 0$ as $t\to0$.   

\medskip
Choose an arbitrary $z \in \mathfrak{Z}_n$. Then $\alpha_n^t(z)$ falls into $\mathfrak{Z}_n$ for all $t \in \mathbb{R}$. There is a normal linear functional $\omega_n$ on $A_n$ such that $\omega_n(z) \neq 0$. Observe that $\alpha_n^t(z)z$ must be either $0$ or $\alpha_n^t(z)=z$ for each $t$. Since $\omega_n$ is normal, $f(t) := \omega_n(\alpha_n^t(z)z)$ is a continuous complex-valued function whose range sits in $\{0,\omega_n(z)\}$. Since $f(0) = \omega_n(z)$, we conclude that $f(t) \equiv \omega_n(z)$ and thus $\alpha_n^t(z)=z$ holds for every $t \in \mathbb{R}$. Consequently, the flow $\alpha^t$ trivially acts on $\mathfrak{Z}_n$ and hence on $\mathcal{Z}(A_n)$ too. 
\end{proof} 

\emph{In what follows, we will further assume that the flow $\alpha^t$ satisfies the equivalent continuity conditions in Lemma \ref{L7.1}, and hence it fixies any elements in $\mathcal{Z}(A_n)$ for every $n$.} 

\medskip
For each $z \in \mathfrak{Z}_n$ we will denote by $\alpha_z^t$ the local flow defined to be the restriction of $\alpha^t$ to $zA_n = B(\mathcal{H}_z)$. In this way, we have a countable family of flows $\alpha_z^t$ on $B(\mathcal{H}_z)$, $z \in \mathfrak{Z}_n$. By the continuity assumption of $\alpha^t_z$, there is a strongly continuous one-parameter unitary group $u_{z,t}$ in $B(\mathcal{H}_z)$ such that $\alpha^t_z = \mathrm{Ad}u_z^t$. (See \cite[Proposition 2.5]{Kallman:JFA71}.) This unitary group $u_z^t$ is unique up to a (continuous) character of the additive group $\mathbb{R}$. 

By Stone's theorem there is a unique positive self-adjoint operator $H_z$ on $\mathcal{H}_z$ so that $u_z^t = \exp(itH_z)$. Moreover, it is well known (or an easy exercise; see e.g.\ \cite[Example 5.5]{Pillet:LNM06}) that the set $K_\beta^\mathrm{n}(\alpha_z^t)$ of all the normal $(\alpha_z^t,\beta)$-KMS states is a singleton if $\exp(-\beta H_z)$ is of trace class; otherwise the empty set. A more precise fact is available; the density operator of the unique normal $(\alpha_z^t,\beta)$-KMS state $\tau^{(\alpha_z^t,\beta)}$ (if exists) must be  
\begin{equation}\label{Eq7.1}
\rho(\alpha_z^t,\beta) := \frac{1}{\mathrm{Tr}_z(\exp(-\beta H_z))}\exp(-\beta H_z) \in B(\mathcal{H}_z) = zA_n
\end{equation}
with the canonical trace $\mathrm{Tr}_z$ on $B(\mathcal{H}_z)$, and $\alpha_z^t = \mathrm{Ad}\rho(\alpha^t_z,\beta)^{-it/\beta}$ holds for every $t$ when $\beta\neq0$. (Note that $\rho(\alpha_z^t,\beta)^{-it/\beta}$ agrees with $u_{z,t}$ only when $\mathrm{Tr}_z(\exp(-\beta H_z))=1$.) We denote by $\mathfrak{Z}_n^\times(\alpha^t,\beta)$ all the $z \in \mathfrak{Z}_n$ of $K_\beta^\mathrm{n}(\alpha_z^t) = \emptyset$, and define $z_{n}^\times(\alpha^t,\beta) := \sum_{z \in \mathfrak{Z}_n^\times(\alpha^t,\beta)} z$ if $\mathfrak{Z}_n^\times(\alpha^t,\beta) \neq \emptyset$, otherwise $z_{n}^\times(\alpha^t,\beta) := 0$. We write $\mathfrak{Z}_n(\alpha^t,\beta) := \mathfrak{Z}_n\setminus\mathfrak{Z}_n^\times(\alpha^t,\beta)$ for simplicity. 

We have a faithful normal conditional expectation $E_n^{(\alpha^t,\beta)} : z_n(\alpha^t,\beta)A_n \to z_n(\alpha^t,\beta)\mathcal{Z}(A_n)$ with $z_n(\alpha^t,\beta):=1-z_n^\times(\alpha^t,\beta)$ defined by
\begin{equation}\label{Eq7.2}
E_n^{(\alpha^t,\beta)}(a) := \sum_{z\in\mathfrak{Z}_n(\alpha^t,\beta)}\tau^{(\alpha_z^t,\beta)}(za)z, \qquad a \in z_n(\alpha^t,\beta)A_n. 
\end{equation}
For the ease of description, we regard $E_n^{(\alpha^t,\beta)}$ as a map from the whole $A_n$ into $\mathcal{Z}(A_n)$ in the obvious manner. 

\medskip
We start with some properties on local KMS states $\tau^{(\alpha_z^t,\beta)}$. Item (1) below is the most important throughout this section.  

\begin{lemma}\label{L7.2} We have the following assertions{\rm:}
\begin{itemize}
\item[(1)] For each pair $(z,z') \in \mathfrak{Z}_{n}\times\mathfrak{Z}_{m}$ with $n>m\geq0$, if $zz' \neq 0$, then we have the non-unital injective normal $*$-homomorphism $\iota_{zz'} : B(\mathcal{H}_{z'}) = z'A_{m} \to zA_{n} = B(\mathcal{H}_z)$ sending each $x \in z'A_{m}$ to $zx \in zz' A_{n} zz'$, which intertwines $\alpha_{z'}^t$ with $\alpha_z^t$, that is, $\iota_{zz'}\circ\alpha_{z'}^t = \alpha_z^t\circ\iota_{zz'}$. Moreover, $z \in \mathfrak{Z}_n(\alpha^t,\beta)$, i.e., $K_\beta^\mathrm{n}(\alpha_z^t)\neq\emptyset$ implies that $z' \in \mathfrak{Z}_m(\alpha^t,\beta)$, i.e., $K_\beta^\mathrm{n}(\alpha_{z'}^t)\neq\emptyset$, and 
\begin{equation}\label{Eq7.3}
\tau^{(\alpha_z^t,\beta)}\circ\iota_{zz'} = \tau^{(\alpha_z^t,\beta)}(zz')\tau^{(\alpha_{z'}^t,\beta)}
\end{equation}
holds. 
\item[(2)] For every $n>m\geq 0$, we have $z_{m}^\times(\alpha^t,\beta) \leq z_{n}^\times(\alpha^t,\beta)$ and thus $z_{m}(\alpha^t,\beta) \geq z_{n}(\alpha^t,\beta)$. 
\item[(3)] For each $n>m\geq0$, $E_n^{(\alpha^t,\beta)}\circ E_m^{(\alpha^t,\beta)} = E_n^{(\alpha^t,\beta)}$ holds on $A_m$. 
\end{itemize}  
\end{lemma}
\begin{proof}
Item (1): Since $A_{m} \subseteq A_{n}$ and $z \in \mathcal{Z}(A_{n})$, the mapping $x \mapsto zx$ defines a normal $*$-homomorphism from $A_{m}$ into $zA_{n}$, and its restriction to $z'A_{m}$ is nothing less than $\iota_{zz'}$. Since $z'A_{n} = B(\mathcal{H}_{z'})$ is a factor, the normal $*$-homomorphism $\iota_{zz'}$ is either injective or the zero map. By assumption, $\iota_{zz'}(z') = zz' \neq 0$ and hence $\iota_{zz'}$ is injective. Moreover, since $\alpha^t$ fixies $z$, we obtain that $\iota_{zz'}\circ\alpha_{z'}^t(z'x) = \iota_{zz'}(\alpha^t(z'x)) = z\alpha^t(z'x) = \alpha^t(zz'x) = \alpha_z^t(zz'x) = \alpha_z^t\circ\iota_{zz'}(z'x)$ for every $x \in A_{m}$. 

Assume that $K_\beta^\mathrm{n}(\alpha^t_z) \neq \emptyset$. Then, we have the unique normal $(\alpha_z^t,\beta)$-KMS state $\tau^{(\alpha_z^t,\beta)}$ on $B(\mathcal{H}_z) = zA_n$. Since $\iota_{zz'}\circ\alpha_{z'}^t = \alpha_z^t\circ\iota_{zz'}$, it is easy to see that $\tau^{(\alpha_z^t,\beta)}(zz')^{-1}\tau^{(\alpha^t_z,\beta)}\circ\iota_{zz'}$ falls into $K_\beta^\mathrm{n}(\alpha_{z'}^t)$, and it must agree with $\tau^{(\alpha^t_{z'},\beta)}$ since $K_\beta^\mathrm{n}(\alpha_{z'}^t)$ is a singleton or the empty set. Hence we have obtained the latter assertion including formula \eqref{Eq7.3}.  

\medskip
Item (2): Choose any $z' \in \mathfrak{Z}_{m}^\times(\alpha^t,\beta)$. By item (1), all the $z \in \mathfrak{Z}_n$ with $zz' \neq 0$ must fall into $\mathfrak{Z}_n^\times(\alpha^t,\beta)$. Thus, $z' = 1z' = \sum_{z \in \mathfrak{Z}_n} zz' = \sum_{z \in \mathfrak{Z}_n: zz' \neq 0} zz'$. Thus $z' \leqq \sum_{z \in \mathfrak{Z}_n^\times(\alpha^t,\beta)} z = z_{n}^\times(\alpha^t,\beta)$, and hence $z_{m}^\times(\alpha^t,\beta) \leq z_{n}^\times(\alpha^t,\beta)$. This implies that $z_{m}(\alpha^t,\beta) \geq z_{n}(\alpha^t,\beta)$.

\medskip
Item (3): For every $a \in A_m$ we have 
\begin{align*}
E_n^{(\alpha^t,\beta)}(E_m^{(\alpha^t,\beta)}(a)) 
&= 
\sum_{z \in \mathfrak{Z}_n(\alpha^t,\beta)} \tau^{(\alpha_z^t,\beta)}(zE_m^{(\alpha^t,\beta)}(a))z \\
&= 
\sum_{z \in \mathfrak{Z}_n(\alpha^t,\beta)}\sum_{z' \in \mathfrak{Z}_m(\alpha^t,\beta)} \tau^{(\alpha_{z'}^t,\beta)}(z'a)\tau^{(\alpha_z^t,\beta)}(zz')z  \\
&= 
\sum_{z \in \mathfrak{Z}_n(\alpha^t,\beta)}\sum_{z' \in \mathfrak{Z}_m(\alpha^t,\beta) : zz' \neq 0} \tau^{(\alpha_z^t,\beta)}(zz'a)z \quad \text{(by equation \eqref{Eq7.3})} \\
&= 
\sum_{z \in \mathfrak{Z}_n(\alpha^t,\beta)}\tau^{(\alpha_z^t,\beta)}(z z_m(\alpha^t,\beta)a)z \\
&= \sum_{z \in \mathfrak{Z}_n(\alpha^t,\beta)}\tau^{(\alpha_z^t,\beta)}(za)z \qquad \text{(by item (2))} \\
&= E_n^{(\alpha^t,\beta)}(a).
\end{align*}
Hence we are done. 
\end{proof} 

An explicit, local description of any locally normal $(\alpha^t,\beta)$-KMS states is available. 

\begin{lemma}\label{L7.3} Any $\omega \in K_\beta^\mathrm{ln}(\alpha^t)$ vanishes at every $z \in \mathfrak{Z}_n^\times(\alpha^t,\beta)$ and its restriction to each $A_n$ must be
\begin{equation*}
\omega\circ E_n^{(\alpha^t,\beta)} = \sum_{z\in\mathfrak{Z}_n(\alpha^t,\beta)}\omega(z)\,\tau^{(\alpha_z^t,\beta)}.
\end{equation*}
Moreover, any state on $A$ factors through $\mathcal{Z}(A_n)$ with $E_n^{(\alpha^t,\beta)}$ as above on each $A_n$ falls into $K_\beta^\mathrm{ln}(\alpha^t)$. 
\end{lemma}
\begin{proof}
It is easy to see that the restriction of $\omega$ to each $zA_n$ with $z\in\mathfrak{Z}_n$ satisfies the $(\alpha_z^t,\beta)$-KMS condition. Thus, the $\omega$ must agree with $\omega(z)\,\tau^{(\alpha_z^t,\beta)}$ on $zA_n$ if $z\in\mathfrak{Z}_n(\alpha^t,\beta)$, since $K_\beta^\mathrm{n}(\alpha_z^t)$ is a singleton in the case; otherwise $\omega(z)=0$. Hence, by normality, the $\omega$ must admit the desired formula on $A_n$. Hence we have shown the first part. 

If a state $\omega$ on $A$ factors through $E_n^{(\alpha^t,\beta)}$ on each $A_n$, then it is normal and satisfies the $(\alpha_n^t,\beta)$-KMS condition on each $A_n$ because the restriction of $\omega$ to $A_n$ is a (possibly infinite) weighted sum of $\tau^{(\alpha_z^t,\beta)}$. Thus, the $\omega$ must be locally normal, and moreover, it satisfies the $(\alpha^t,\beta)$-KMS condition on the norm-dense subalgebra $\bigcup_n A_n$. By the standard Phragmen--Lindel\"of method as in the proof of Proposition \ref{P4.3}, we conclude that the $\omega$ satisfies the $(\alpha^t,\beta)$-KMS condition on the whole $A$ and hence it falls into $K_\beta^\mathrm{ln}(\alpha^t)$.   
\end{proof}

The above lemma asserts that any $\omega \in K_\beta^\mathrm{ln}(\alpha^t)$ is completely determined by its restrictions to all the $\mathcal{Z}(A_n)$. This fact suggests us to consider the following subalgebras of $A$: Consider the commutative $W^*$-subalgebra $GZ_n:=\mathcal{Z}(A_n)\vee\mathcal{Z}(A_{n-1})\vee\cdots\vee\mathcal{Z}(A_1)$ of each $A_n$, and put $GZ := \varinjlim GZ_n$, i.e., the commutative $C^*$-subalgebra of $A$ generated by the $GZ_n$. These are nothing less than so-called Gelfand-Tsetlin algebras (see e.g.\ \cite{Vershik:ProgressMath06}). The next proposition says that the commutative $C^*$-algebra $GZ$ is enough to capture the full information of $K_\beta^\mathrm{ln}(\alpha^t)$. 

\begin{proposition}\label{P7.4} The restriction map $r_{GZ} : \omega \in K_\beta^\mathrm{ln}(\alpha^t) \mapsto r_{GZ}(\omega) := \omega\!\upharpoonright_{GZ} \in S(GZ)$ is affine, injective and continuous with respect to the weak$^*$ topologies. The map gives a homeomorphism between $K_\beta^\mathrm{ln}(\alpha^t)$ and $r_{GZ}(K_\beta^\mathrm{ln}(\alpha^t))$.  Moreover, The range $r_{GZ}(K_\beta^\mathrm{ln}(\alpha^t))$ is exactly the set of all locally normal states on $GZ$ which factors through $\mathcal{Z}(A_n)$ with $E_n^{(\alpha^t,\beta)}$ on each $GZ_n$.  
\end{proposition}
\begin{proof} 
It is trivial that the map $r_{GZ}$ is affine and continuous with respect to the weak$^*$ topologies. 

If $r_{GZ}(\omega_1) = r_{GZ}(\omega_1)$, then Lemma \ref{L7.3} shows that $\omega_1 = r_{GZ}(\omega_1)\circ E_n^{\alpha^t,\beta)} = r_{GZ}(\omega_2)\circ E_n^{\alpha^t,\beta)} = \omega_2$ on each $A_n$, since $\mathcal{Z}(A_n) \subset GZ_n \subset GZ$. By continuity we conclude that $\omega_1 = \omega_2$. Hence $r_{GZ}$ is injective. 

Assume that $r_{GZ}(\omega_i) \to r_{GZ}(\omega)$ in the weak$^*$ topology with $\omega_i, \omega \in K_\beta^\mathrm{ln}(\alpha^t)$. Then, $\omega_i(a) = r_{GZ}(\omega_i)(E_n^{(\alpha^t,\beta)}(a)) \to r_{GZ}(\omega)(E_n^{(\alpha^t,\beta)}(a)) = \omega(a)$ for all $a \in A_n$ and for all $n$. Hence, $\omega_i \to \omega$ pointwisely on $\bigcup_n A_n$. Since $\bigcup_n A_n$ is norm-dense in $A$, we conclude that $\omega_i \to \omega$ in the weak$^*$ topology.   

Any $r_{GZ}(\omega)$ with $\omega \in K_\beta^\mathrm{ln}(\alpha^t)$ is trivially locally normal on $GZ$ and factors through $\mathcal{Z}(A_n)$ with $E_n^{(\alpha^t,\beta)}$ on each $GZ_n$ (since $GZ_n \subset A_n$). 

Assume that a $\chi \in S(GZ)$ is locally normal and factors through $\mathcal{Z}(A_n)$ with $E_n^{(\alpha^t,\beta)}$ on each $GZ_n$. For each $n$, we define $\omega_n := \chi\circ E_n^{(\alpha^t,\beta)}$ on $A_n$, that is, 
\[
\omega_n(a) = \sum_{z \in \mathfrak{Z}_n(\alpha^t,\beta)} \chi(z)\,\tau^{(\alpha_z^t,\beta)}
\]
holds and hence it is trivially a normal $(\alpha_n^t,\beta)$-KMS state on $A_n$. Moreover, for every $a \in A_n$, we have
\begin{align*}
\omega_{n+1}(a) 
&= \chi\circ E_{n+1}^{(\alpha^t,\beta)}(a) \\
&= \chi\circ (E_{n+1}^{(\alpha^t,\beta)}\circ E_n^{(\alpha^t,\beta)}(a)) \quad \text{(by Lemma \ref{L7.2}(3))} \\
&= (\chi\circ E_{n+1}^{(\alpha^t,\beta)})\circ E_n^{(\alpha^t,\beta)}(a) \\
&= \chi\circ E_n^{(\alpha^t,\beta)}(a) = \omega_n(a). 
\end{align*}
It follows that a linear functional $\omega$ on $\bigcup_n A_n$ can be defined by $\omega = \omega_n$ on each $A_n$. Since $|\omega(a)| = |\omega_n(a)| \leq \Vert a \Vert$ for any $a \in A_n$, the $\omega$ is contractive on $\bigcup_n A_n$, and thus (uniquely) extends to a state on the whole $A$ thanks to \cite[Proposition 2.3.11]{BratteliRobinson:Book1}. Since $\omega\circ E_n^{(\alpha^t,\beta)} = \omega_n \circ E_n^{(\alpha^t,\beta)} = \omega_n = \omega$ on each $A_n$, we conclude that the $\omega$ falls into $K_\beta^\mathrm{ln}(\alpha^t)$ by Lemma \ref{L7.3} (its latter part). Moreover, we observe, by construction, that $\omega = \omega_n = \chi$ on each $GZ_n$. Therefore, $\chi = r_{GZ}(\omega)$ by continuity, and hence we have determined the range $r_{GZ}(K_\beta^\mathrm{ln}(\alpha^t))$ as desired. 
\end{proof} 

Since $GZ$ is commutative, it is natural to use measure theory to investigate $K_\beta^\mathrm{ln}(\alpha^t)$ through the restriction map $r_{GZ}$. The next lemma is a preparation for the use of measure theory.  

\begin{lemma}\label{L7.5} The following assertions hold{\rm:}
\begin{itemize}
\item[(1)] For any $z_k \in \mathfrak{Z}_{n_k}$ with $n_1 < n_2 < \cdots < n_m$, we have $z_m\cdots z_1 \neq 0$ if and only if $z_{k+1}z_k \neq 0$ for all $1 \leq k \leq m-1$. 
\item[(2)] Each $GZ_n$ is a commutative $W^*$-algebra with the complete family of minimal projections $z_n\cdots z_1 \neq 0$ with $z_k \in \mathfrak{Z}_k$, $1\leq k\leq n$.
\item[(3)] The restriction $r_{GZ}(\omega)$ of any $\omega \in K_\beta^\mathrm{ln}(\alpha^t)$ is calculated by 
\[
r_{GZ}(\omega)(z_n\cdots z_1) = 
\begin{cases}
\omega(z_n)\,\tau^{(\alpha_{z_n}^t,\beta)}(z_n\cdots z_1) & (z_n \in \mathfrak{Z}_n(\alpha^t,\beta)), \\
0 & (z_n \in \mathfrak{Z}_n^\times(\alpha^t,\beta))
\end{cases}
\]
for every minimal projection $z_n\cdots z_1$ of $GZ_n$ with $z_k \in \mathfrak{Z}_k$, $1 \leq k \leq n$. . 
\end{itemize}
\end{lemma}
\begin{proof}
Item (1): The only if part is trivial. 

The if part can be proved inductively. Assume that $z_k \cdots z_1 \neq 0$ has been confirmed under the assumption. By $z_{k+1}z_k\neq0$, Lemma \ref{L7.5}(1) ensures that $z_{k+1}\cdots z_1 = \iota_{z_{k+1}z_k}(z_k \cdots z_1)$ must be non-zero. 

\medskip
Item (2): Let $z_n\cdots z_1$ be as in the statement. Observe that all the linear combinations of elements of the form $x_n\cdots x_1$ with $x_k \in \mathcal{Z}(A_k)$ is $\sigma$-weakly dense in $GZ_n$ and also that for $z_k x_k = \lambda(x_k)z_k$ for some scalar $\lambda(x_k)$. Hence $(z_n\cdots z_1)(x_n\cdots x_1) = (z_n x_n)\cdots(z_1 x_1) = \lambda(x_n)\cdots\lambda(x_1)z_n\cdots z_1 \in \mathbb{C}z_n\cdots z_1$. By continuity we conclude that $z_n\cdots z_1 GZ_n = \mathbb{C}z_n\cdots z_1$, implying that $z_n\cdots z_1$ is a minimal projection of $GZ_n$. 

Let $z'_n\cdots z'_1$ be another one as in the statement. Then, $(z_n\cdots z_1) (z'_n\cdots z'_1) = (z_n z'_n)\cdots(z_1 z'_1) = \delta_{z_n,z'_n}\cdots\delta_{z_1,z'_1}z_n\cdots z_1$. Moreover, the (possibly infinite) sum of all such $z_n\cdots z_1$ trivially becomes $1$. Hence, all such $z_n\cdots z_1$ form a complete family of minimal projections of $GZ_n$. 

\medskip
Item (3) immediately follows from Lemma \ref{L7.3}. 
\end{proof}

Let $\omega \in K_\beta^\mathrm{ln}(\alpha^t)$ be arbitrarily given, and write $\chi := r_{GZ}(\omega)$ for simplicity. We will describe its GNS triple $(\pi_\chi : GZ \curvearrowright \mathcal{H}_\chi,\xi_\chi)$ in terms of measure theory as follows. Let $\Omega$ be the space of all infinite paths $(z(k)) \in \prod_k \mathfrak{Z}_k$ with $z_{n+1}z_n \neq 0$, and we consider the smallest $\sigma$-algebra over $\Omega$ containing all the cylinder sets $[z_m,\dots,z_m]^m_n = \{ (z(k)) \in \Omega\,; z(m)=z_m,\dots,z(n)=z_n\}$ with $m < n$. Then, we can define the probability measure $\gamma_\chi$ on $\Omega$ by 
\[
\gamma_\chi([z_n,\dots,z_m]_n^m) := 
\chi(z_n\cdots z_m), 
\]  
because, for example,  
\[
\chi(z_n\cdots z_m) = \sum_{z \in \mathfrak{Z}_{n+1}} \chi(z z_n\cdots z_m)
\]
thanks to the local normality of $\chi$. The next lemma is easily proved, and hence we leave its proof to the reader. 

\begin{lemma}\label{L7.6} There is a unitary transform from $\mathcal{H}_\chi$ onto $L^2(\Omega,\gamma_\chi)$ sending each $\pi_\chi(z_n\cdots z_1)\xi_\chi$ with minimal projection $z_n\cdots z_1$ of $GZ_n$ to the indicator function $\mathbf{1}_{[z_n,\dots,z_1]_n^1}$ of the corresponding cylinder $[z_n,\dots,z_1]^1_n$. The unitary transform intertwines the associated von Neumann algebra $\pi_\chi(GZ)''$ with $L^\infty(\Omega,\gamma_\chi)$ {\rm(}acting on $L^2(\Omega,\gamma_\chi)$ as multiplication operators{\rm)}.  
\end{lemma}  

Let $\mathscr{B}_m$ be the $\sigma$-subalgebra generated by all the cylinders $[z_n,\dots,z_m]_n^m$ with $n\geq m$. Set $\mathscr{B}_\infty := \bigcap_m \mathscr{B}_m$, the tail $\sigma$-algebra.  

\begin{lemma}\label{L7.7} If $\chi = r_{GZ}(\omega)$ is an extreme point in $r_{GZ}(K_\beta^\mathrm{ln}(\alpha^t))$ {\rm(}i,e., $\omega$ is an extreme point in $K_\beta^\mathrm{ln}(\alpha^t)${\rm)}, then the image of $\mathscr{B}_\infty$ under $\gamma_\chi$ is $\{0,1\}$. 
\end{lemma} 
\begin{proof} 
By Lemma \ref{L7.5}(3), we have, for any $\ell > m > n \geq 1$, 
\[ 
\gamma_\chi([z_\ell,\dots,z_m]_\ell^m \cap [z_n,\dots,z_1]_n^1) 
= 
\chi(z_\ell \cdots z_m z_n \cdots z_1) = 
\chi(z_\ell)\,\tau^{(\alpha_{z_\ell}^t,\beta)}(z_\ell \cdots z_m z_n \cdots z_1) 
\]
if $z_\ell \in \mathfrak{Z}_\ell(\alpha^t,\beta)$; otherwise $0$. Then, the iterative use of formula \eqref{Eq7.3} shows that 
\[
\tau^{(\alpha_{z_\ell}^t,\beta)}(z_\ell \cdots z_m z_n \cdots z_1) = 
\tau^{(\alpha_{z_\ell}^t,\beta)}(z_\ell \cdots z_m z_n)\,\tau^{(\alpha_{z_n}^t,\beta)}(z_n \cdots z_1)
\]
if $z_n \in \mathfrak{Z}_\ell(\alpha^t,\beta)$; otherwise $0$. Since 
\[
\chi(z_\ell)\,\tau^{(\alpha_{z_\ell}^t,\beta)}(z_\ell \cdots z_m z_n) 
=\chi(z_\ell\cdots z_m z_n) 
=\gamma_\chi([z_\ell,\dots,z_m]_\ell^m \cap [z_n]_n^n),
\] 
we have obtained that 
\[
\gamma_\chi([z_\ell,\dots,z_m]_\ell^m \cap [z_n,\dots,z_1]_n^1) 
=  \gamma_\chi([z_\ell,\dots,z_m]_\ell^m \cap [z_n]_n^n)\,\tau^{(\alpha_{z_n}^t,\beta)}(z_n \cdots z_1) 
\]
if $z_n \in \mathfrak{Z}_n(\alpha^t,\beta)$; otherwise $0$. It follows, by the $\pi$-$\lambda$ principle, that 
\begin{equation}\label{Eq7.4}
\gamma_\chi(\Lambda \cap [z_n,\dots,z_1]_n^1) = \gamma_\chi(\Lambda \cap [z_n]_n^n)\,\mathbf{1}_{\mathfrak{Z}_n(\alpha^t,\beta)}(z_n)\,\tau^{(\alpha_{z_n}^t,\beta)}(z_n \cdots z_1) 
\end{equation}
for any $\Lambda \in \mathcal{B}_m$ and $m > n$. 

On the contrary, suppose that $0 < \gamma_\chi(\Lambda) <1$ for some $\Lambda \in \mathscr{B}_\infty$. Let $p \in \pi_\chi(GZ)''$ be the non-trivial projection corresponding to $\mathbf{1}_\Lambda \in L^\infty(\Omega,\gamma_\chi)$ by Lemma \ref{L7.6}. Define $\chi_1 := \gamma_\chi(\Lambda)^{-1}(p\pi_\chi(\,\cdot\,)\xi_\chi\,|\,\xi_\chi)_{\mathcal{H}_\chi}$ and $\chi_2 := \gamma_\chi(\Lambda^c)^{-1}((1-p)\pi_\chi(\,\cdot\,)\xi_\chi\,|\,\xi_\chi)_{\mathcal{H}_\chi}$, both of which are trivially locally normal states on $GZ$. What we have proved in the above paragraph shows that 
\begin{align*}
\chi_1(z_n\cdots z_1)
&= 
\gamma_\chi(\Lambda \cap [z_n,\dots,z_1]_n^1) \\
&= \gamma_\chi(\Lambda \cap [z_n]_n^n)\,\mathbf{1}_{\mathfrak{Z}_n(\alpha^t,\beta)}(z_n)\,\tau^{(\alpha_{z_n}^t,\beta)}(z_n \cdots z_1) \\
&=
\chi_1(z_n)\,\mathbf{1}_{\mathfrak{Z}_n(\alpha^t,\beta)}(z_n)\,\tau^{(\alpha_{z_n}^t,\beta)}(z_n \cdots z_1) \\
&= 
\chi_1\circ E_n^{(\alpha^t,\beta)}(z_n\cdots z_1).  
\end{align*}
By Lemma \ref{L7.5}(2), we conclude that $\chi_1 = \chi_1\circ E_n^{(\alpha^t,\beta)}$ on each $GZ_n$. Similarly, we have $\chi_2 = \chi_2\circ E_n^{(\alpha^t,\beta)}$ on each $GZ_n$. Consequently, both $\chi_1$ and $\chi_2$ fall into $r_{GZ}(K_\beta^\mathrm{ln}(\alpha^t))$. Since $p$ is non-trivial, we see that $\chi_1 \neq \chi_2$. Since $\chi = \gamma_\chi(\Lambda)\,\chi_1 + \gamma_\chi(\Lambda^c)\,\chi_2$ by construction, the $\chi$ is not an extreme point in $r_{GZ}(K_\beta^\mathrm{ln}(\alpha^t))$, a contradiction.  
\end{proof}

We are ready to give Vershik--Kerov's approximation theorem in the present setup. The method here is the same as in \cite{Olshanski:JFA03}, which supplies the details to Vershik--Kerov's work \cite{VershikKerov:SovMathDokl82}. 

\begin{theorem}\label{T7.8} For any extreme $\omega \in K_\beta^\mathrm{ln}(\alpha^t)$, there is a sequence $z(m) \in \mathfrak{Z}_m(\alpha^t,\beta)$ with $z(m+1)z(m) \neq 0$ for all $m \geq 1$ so that
\[
\lim_{m\to\infty}\tau^{(\alpha_{z(m)}^t,\beta)}(z(m)a) = \omega(a), \qquad a \in \bigcup_n A_n.
\]
\end{theorem}
\begin{proof}
Write $\chi := r_{GZ}(\omega)$. Let $z \in \mathfrak{Z}_n(\alpha^t,\beta)$ be arbitrarily given. For $m > n$, we consider the random variable 
\[
X = \sum_{z'\in\mathfrak{Z}_m(\alpha^t,\beta)} \tau^{(\alpha_{z'}^t,\beta)}(z' z)\mathbf{1}_{[z']_m^m}
\]
on the probability space $(\Omega,\gamma_\chi)$. By the same calculation using Lemma \ref{L7.5}(3) together with the $\pi$-$\lambda$ principle, we can see that 
\[
\mathbb{E}[\mathbf{1}_\Lambda X] = \gamma_\chi(\Lambda \cap [z]_n^n) = \mathbb{E}[\mathbf{1}_\Lambda \mathbf{1}_{[z]_n^n}], \quad \Lambda \in \mathscr{B}_m,
\]
which implies that 
\[
\mathbb{E}[\mathbf{1}_{[z]_n^n}\mid\mathscr{B}_m] = X = \sum_{z' \in \mathfrak{Z}_m(\alpha^t,\beta)} \tau^{(\alpha_{z'}^t,\beta)}(z'z)\mathbf{1}_{[z']_m^m}. 
\] 
By Levy's downward martingale theorem (see e.g.\ \cite[section 14.4]{Williams:Book}) together with Lemma \ref{L7.7}, we have 
\[
\lim_{m\to\infty} \sum_{z' \in \mathfrak{Z}_m(\alpha^t,\beta)} \tau^{(\alpha_{z'}^t,\beta)}(z'z)\,\mathbf{1}_{[z']_m^m} = \lim_{m\to\infty} \mathbb{E}[\mathbf{1}_{[z]_n^n}\mid\mathscr{B}_m] = \mathbb{E}[\mathbf{1}_{[z]_n^n}] = \gamma_\chi([z]_n^n) = \chi(z)
\]
$\gamma_\chi$-almost surely. 

Since $\gamma_\chi([z']_m^m) = \omega(z) = 0$ for all $z \in \mathfrak{Z}_m^\times(\alpha^t,\beta)$, the union $\mathcal{N} := \bigcup_m \bigcup_{z' \in \mathfrak{Z}_m^\times(\alpha^t,\beta)} [z']_m^m$ is $\gamma_\chi$-null.  Observe that any path $(z(m))$ in $\Omega\setminus\mathcal{N}$ satisfies that $z(m) \in \mathfrak{Z}_m(\alpha^t,\beta)$ for all $m$. Since $\bigcup_n \mathfrak{Z}_n$ is countable, we can choose a sequence $(z(m))$ from $\Omega\setminus\mathcal{N}$ so that 
\[
\omega(z) = \lim_{m\to\infty} \sum_{z' \in \mathfrak{Z}_m(\alpha^t,\beta)} \tau^{(\alpha_{z'}^t,\beta)}(z'z)\mathbf{1}_{[z']_m^m}((z(m))) = 
\lim_{m\to\infty} \tau^{(\alpha_{z(m)}^t,\beta)}(z(m)z)
\]
for all $z \in \bigcup_n \mathfrak{Z}_n$. 

Let $a \in A_n$ be arbitrarily given. Choose an increasing sequence of finite subsets $\mathfrak{F}_k$ of $\mathfrak{Z}_n(\alpha^t,\beta)$. Then we have 
\begin{align*}
&|\tau^{(\alpha_{z(m)}^t,\beta)}(z(m)a) - \omega(a)| \\
&\leq
\sum_{z \in \mathfrak{Z}_n(\alpha^t,\beta)}
|\tau^{(\alpha_{z(m)}^t,\beta)}(z(m)z)-\omega(z)|\,|\tau^{(\alpha_z^t,\beta)}(za)| \quad \text{(by Lemmas \ref{L7.2}(1), \ref{L7.3})} \\
&\leq 
\Vert a\Vert\,\sum_{z \in \mathfrak{Z}_n(\alpha^t,\beta)}
|\tau^{(\alpha_{z(m)}^t,\beta)}(z(m)z)-\omega(z)| \\
&\leq
2\Vert a\Vert\,\sum_{z \in \mathfrak{F}_k}
|\tau^{(\alpha_{z(m)}^t,\beta)}(z(m)z)-\omega(z)| 
+
2\Vert a\Vert\,\left(1-\sum_{z \in\mathfrak{F}_k}
\omega(z)\right),
\end{align*}
since $\omega(z)=0$ for all $z\in\mathfrak{Z}_n^\times(\alpha^t,\beta)$ and since
\[
\sum_{z \in \mathfrak{Z}_n(\alpha^t,\beta)}
\omega(z) = \sum_{z \in \mathfrak{Z}_n(\alpha^t,\beta)}
\tau^{(\alpha_{z(m)}^t,\beta)}(z(m)z)
=1.
\] 
Thus, 
\[
\limsup_{m\to\infty} |\tau^{(\alpha_{z(m)}^t,\beta)}(z(m)a) - \omega(a)| \leq 
2\Vert a\Vert\,\left(1-\sum_{z \in \mathfrak{F}_k}\omega(z)\right) \to 0 \quad \text{as $k\to\infty$}
\]
by the normality of $\omega$ on $A_n$. Hence we are done. 
\end{proof}

So far, we have successfully established the approximation theorem in the present setup. Toward the classification problem, we should develop another framework of analysis of $K_\beta^\mathrm{ln}(\alpha^t)$. 

\medskip 
We introduce the Markov kernel
\begin{equation}\label{Eq7.5}
\kappa_{(\alpha^t,\beta)}(z,z') := 
\begin{cases} 
0 & (z \in \mathfrak{Z}_{n}^\times(\alpha^t,\beta)), \\
\tau^{(\alpha_z^t,\beta)}(zz') & (z \in \mathfrak{Z}_{n}(\alpha^t,\beta))
\end{cases} 
\end{equation}
for $(z,z') \in \mathfrak{Z}_n\times\mathfrak{Z}_{m}$ with $n > m \geq 0$. 

\begin{lemma}\label{L7.9} We have 
\begin{itemize}
\item[(1)] $\displaystyle \kappa_{(\alpha^t,\beta)}(z,1) = 1$ and  
$\displaystyle \kappa_{(\alpha^t,\beta)}(z,z') \geq 0$ for every $z \in \mathfrak{Z}_n$ and $z' \in \mathfrak{Z}_m$ with $n > m \geq 0$; 
\item[(2)] $\displaystyle \sum_{z' \in \mathfrak{Z}_m}\kappa_{(\alpha^t,\beta)}(z,z') = 1$ for every $z \in \mathfrak{Z}_n(\alpha^t,\beta)$; 
\item[(3)] $\displaystyle \kappa_{(\alpha^t,\beta)}(z,z'') = \sum_{z' \in \mathfrak{Z}_{m}} \kappa_{(\alpha^t,\beta)}(z,z')\kappa_{(\alpha^t,\beta)}(z',z'')$ 
for every pair $(z,z'') \in \mathfrak{Z}_{n}\times\mathfrak{Z}_{\ell}$ and $n > m > \ell \geq 0$.
\end{itemize}
\end{lemma}
\begin{proof} 
Item (1): The first normalization and the second non-negativity properties follow from that $\tau^{(\alpha_z^t,\beta)}$ is a state. 

\medskip
Item (2): The identity follows from $\sum_{z' \in \mathfrak{Z}_m} z' = 1$, the normality of $\tau^{(\alpha_z^t,\beta)}$ and $n>m\geq0$. 

\medskip
Item (3): If $z$ falls into $\mathfrak{Z}_{\beta,n}^\times(\alpha^t,\beta)$, then $\kappa_{(\alpha^t,\beta)}(z,z'')$ and $\kappa_{(\alpha^t,\beta)}(z,z')$ must always be zero, and hence the desired equation trivially holds. 

Assume that $z$ falls into $\mathfrak{Z}_{n}(\alpha^t,\beta)$. We have 
\begin{align*}
\kappa_{(\alpha^t,\beta)}(z,z'') 
&= \tau^{(\alpha_z^t,\beta)}(zz'') \\
&= \sum_{z' \in \mathfrak{Z}_m} \tau^{(\alpha_z^t,\beta)}(zz'z'') \qquad\qquad (\text{by the normality of $\tau^{(\alpha_z^t,\beta)}$}) \\
&= \sum_{z' \in \mathfrak{Z}_m:zz' \neq 0} \tau^{(\alpha_z^t,\beta)}\circ\iota_{zz'}(z'z'') \\
&= \sum_{z' \in \mathfrak{Z}_m:zz' \neq 0} \tau^{(\alpha_z^t,\beta)}(zz')\,\tau^{(\alpha_{z'}^t,\beta)}(z'z'') \quad (\text{by Lemma \ref{L7.2}(1)}) \\
&= \sum_{z' \in \mathfrak{Z}_m:zz' \neq 0} \kappa_{(\alpha^t,\beta)}(z,z')\,\kappa_{(\alpha^t,\beta)}(z',z'') \\
&= \sum_{z' \in \mathfrak{Z}_m} \kappa_{(\alpha^t,\beta)}(z,z')\,\kappa_{(\alpha^t,\beta)}(z',z'')
\end{align*}
because $zz'=0$ implies $\kappa_{(\alpha^t,\beta)}(z,z') = 0$ by definition. 
\end{proof}

Let $\mathfrak{Z} = \bigsqcup_{n\geq0}\mathfrak{Z}_n$ be the disjoint union. A function $\nu : \mathfrak{Z} \to \mathbb{C}$ is a \emph{$\kappa_{(\alpha^t,\beta)}$-harmonic} if identity
\begin{equation}\label{Eq7.6} 
\nu(z') = \sum_{z \in \mathfrak{Z}_{n+1}} \nu(z)\kappa_{(\alpha^t,\beta)}(z,z'), \quad z' \in \mathfrak{Z}_n
\end{equation}    
holds for every $n \geq 0$. A $\kappa_{(\alpha^t,\beta)}$-harmonic function $\nu$ is \emph{positive} if $\nu(z) \geq 0$ for all $z \in \mathfrak{Z}$, and \emph{normalized} if $\nu(1)=1$, where one must recall that $\mathfrak{Z}_0 = \{1\}$. By the normalization property in Lemma \ref{L7.9}(1) we observe that 
\[
1 = \nu(1) = \sum_{z\in\mathfrak{Z}_{n+1}}\nu(z)\kappa_{(\alpha^t,\beta)}(z,1) = \sum_{z\in\mathfrak{Z}_n}\nu(z)
\]
holds for every $n\geq0$. This means that a normalized, positive $\kappa_{(\alpha^t,\beta)}$-harmonic function gives or arises from a sequence of discrete measures on $\mathfrak{Z}_n$, $n\geq0$, called a \emph{coherent system}, see e.g.\ \cite[chapter 1, section 1]{Kerov:Book}. We denote by $H^+_1(\kappa_{(\alpha^t,\beta)})$ all the normalized, positive $\kappa_{(\alpha^t,\beta)}$-harmonic functions on $\mathfrak{Z}$, which is nothing less than the projective limit of the \emph{projective chain}
\[
\mathscr{M}(\mathfrak{Z}_{n-1}) \overset{\kappa_{(\alpha^t,\beta)}}{\longleftarrow} \mathscr{M}(\mathfrak{Z}_{n}), 
\]
where $\mathscr{M}(\mathfrak{Z}_n)$ is the set of all discrete probability measures on $\mathfrak{Z}_n$ and the map $\overset{\kappa_{(\alpha^t,\beta)}}{\longleftarrow}$ is given by matrix product from the right like the right-hand side of identity \eqref{Eq7.6}. 

Following Borodin--Olshanski \cite[section 7]{BorodinOlshanski:Book}, we will call such a Markov kernel a \emph{link} among $\mathfrak{Z}_n$ and denote it by 
\[
\mathfrak{Z}_{n-1} \overset{\kappa_{(\alpha^t,\beta)}}{\dashleftarrow} \mathfrak{Z}_n. 
\]

The next proposition identifies $K^\mathrm{ln}_\beta(\alpha^t)$ with $H^+_1(\kappa_{(\alpha^t,\beta)})$. This type of identification has played an essential role of harmonic analysis of inductive limits of groups since Vershik-Kerov's work. See e.g.\ \cite[chapter 7]{BorodinOlshanski:Book}. The proposition below, dealing with KMS states rather than tracial states, has also essentially been known in the AF algebra setting (i.e., all the $A_n$ are finite dimensional) by Kishimoto \cite[Proposition 4.1]{Kishimoto:RepMathPhys00}.  

\begin{proposition}\label{P7.10} For any $\omega \in K^\mathrm{ln}_\beta(\alpha^t)$, we define the bounded function $\nu[\omega]$ on $\mathfrak{Z} \to \mathbb{C}$ by $\nu[\omega](z) := \omega(z)$, $z \in \mathfrak{Z}$. Then, the correspondence $\omega \mapsto \nu[\omega]$ defines a bijective affine homeomorphism from $K_\beta^\mathrm{ln}(\alpha^t)$ with weak$^*$ topology onto $H^+_1(\kappa_{(\alpha^t,\beta)})$ with the topology of pointwise convergence. 
\end{proposition}
\begin{proof}
By Lemma \ref{L7.3} any $\omega \in K_\beta^\mathrm{ln}(\alpha^t)$ is uniquely determined by the function $\nu[\omega]$, that is, $\nu[\omega_1] = \nu[\omega_2]$ implies $\omega_1 = \omega_2$. 

By Lemma \ref{L7.3} and the definition of $\kappa_{(\alpha^t,\beta)}(z,z')$, we have
\[
\nu[\omega](z') = \omega(z') = \sum_{z \in \mathfrak{Z}_{n+1}(\alpha^t,\beta)} \omega(z)\,\tau^{(\alpha_z^t,\beta)}(zz') 
= \sum_{z\in\mathfrak{Z}_{n+1}} \nu[\omega](z)\,\kappa_{(\alpha^t,\beta)}(z,z')
\]
for any $z' \in \mathfrak{Z}_n$. This shows that $\nu[\omega]$ is positive $\kappa_{(\alpha^t,\beta)}$-harmonic, and $\nu[\omega](1)=\omega(1)=1$. Moreover, that the correspondence $\omega \mapsto \nu[\omega]$ defines an injective map from $K_\beta^\mathrm{ln}(\alpha^t)$ into $H_1^+(\kappa_{(\alpha^t,\beta)})$. The map is clearly affine. 

\medskip
Let $\nu \in H_1^+(\kappa_{(\alpha^t,\beta)})$ be arbitrarily given. Observe that, if $z' \in \mathfrak{Z}_n^\times(\alpha^t,\beta)$, then $zz'=0$ implies $\kappa_{(\alpha^t,\beta)}(z,z') =0$ and since $z' \in \mathfrak{Z}_n^\times(\alpha^t,\beta)$ and $zz'\neq0$ imply $z \in \mathfrak{Z}_{n+1}^\times(\alpha^t,\beta)$ by Lemma \ref{L7.2}(1), and hence $\kappa_{(\alpha^t,\beta)}(z,z') = 0$. Therefore, 
\[
\nu(z') = \sum_{z \in \mathfrak{Z}_{n+1}} \nu(z)\,\kappa_{(\alpha^t,\beta)}(z,z') = \sum_{z \in \mathfrak{Z}_{n+1}:zz'\neq0} \nu(z)\,\kappa_{(\alpha^t,\beta)}(z,z') = 0. 
\] 
Thus, $\nu$ is supported in $\bigsqcup_n \mathfrak{Z}_n(\alpha^t,\beta)$. Then, we can define a sequence of normal states $\omega[\nu]_n$ on $A_n$ by 
\[
\omega[\nu]_n := \sum_{z \in \mathfrak{Z}_n(\alpha^t,\beta)} \nu(z)\,\tau^{(\alpha_z^t,\beta)}.
\] 
Since all the $\tau^{(\alpha_z^t,\beta)}$ are normal $(\alpha_z^t,\beta)$-KMS states, we see that $\omega[\nu]_n$ is also a normal $(\alpha_n^t,\beta)$-KMS state.  

For any $a \in A_n$ we have 
\begin{align*} 
\omega[\nu]_{n+1}(a) 
&= 
\sum_{z \in \mathfrak{Z}_{n+1}(\alpha^t,\beta)} \nu(z)\,\tau^{(\alpha_z^t,\beta)}(za) \\
&=
\sum_{z \in \mathfrak{Z}_{n+1}(\alpha^t,\beta)} \nu(z)\,\sum_{z'\in\mathfrak{Z}_n} \tau^{(\alpha_z^t,\beta)}(zz'a) \\
&= 
\sum_{z \in \mathfrak{Z}_{n+1}(\alpha^t,\beta)} \nu(z)\,\sum_{z'\in\mathfrak{Z}_n:zz'\neq0} \tau^{(\alpha_z^t,\beta)}\iota_{zz'}(z'a) \\
&= 
\sum_{z \in \mathfrak{Z}_{n+1}(\alpha^t,\beta)} \nu(z)\sum_{z'\in \mathfrak{Z}_n(\alpha^t,\beta)} \tau^{(\alpha_z^t,\beta)}(zz')\,\tau^{(\alpha_{z'}^t,\beta)}(z'a) \quad \text{(by Lemma \ref{L7.2}(1))} \\
&= 
\sum_{z \in \mathfrak{Z}_{n+1}(\alpha^t,\beta)} \nu(z)\sum_{z'\in\mathfrak{Z}_n(\alpha^t,\beta)} \kappa_{(\alpha^t,\beta)}(zz')\,\tau^{(\alpha_{z'}^t,\beta)}(z'a) \\
&= 
\sum_{z'\in\mathfrak{Z}_n(\alpha^t,\beta)}\left(\sum_{z \in \mathfrak{Z}_{n+1}} \nu(z)\,\kappa_{(\alpha^t,\beta)}(zz')\right)\tau^{(\alpha_{z'}^t,\beta)}(z'a) \\
&= 
\sum_{z'\in\mathfrak{Z}_n(\alpha^t,\beta)}\nu(z')\,\tau^{(\alpha_{z'}^t,\beta)}(z'a) \qquad \text{(by the $\kappa_{(\alpha^t,\beta)}$-harmonicity of $\nu$)} \\
&= \omega[\nu]_n(a). 
\end{align*} 
It follows, as in the proof of Proposition \ref{P7.4}, that an $\omega[\nu] \in K_\beta^\mathrm{ln}(\alpha^t)$ can be defined by $\omega[\nu] = \omega[\nu]_n$ on each $A_n$. Since $\nu[\omega[\nu]](z) = \omega[\nu](z) = \omega[\nu]_n(z) = \nu(z)$ for any $z \in \mathfrak{Z}_n$, the mapping $\omega \in K_\beta^\mathrm{ln}(\alpha^t) \mapsto \nu[\omega] \in H_1^+(\kappa_{(\alpha_t,\beta)})$ is surjective. 

\medskip
If $\omega_i \to \omega$ in $K_\beta^\mathrm{ln}(\alpha^t)$, then $\nu[\omega_i](z) = \omega_i(z) \to \omega(z) = \nu[\omega](z)$ for all $z \in \mathfrak{Z}$. Hence $\omega \mapsto \nu[\omega]$ is continuous. 

On the other hand, assume that $\nu_i \to \nu$ in $H_1^+(\kappa_{(\alpha^t,\beta)})$ pointwisely. 
Choose an increasing sequence of finite subsets $\mathfrak{F}_{k}$ of $\mathfrak{Z}_n$, approaching to the whole $\mathfrak{Z}_n$. 
For any $a \in A_n$ we have, by the essentially same argument as in Theorem \ref{T7.8} (the last paragraph of its proof),  
\begin{align*}
|\omega[\nu_i](a)-\omega[\nu](a)| 
&\leq
2\Vert a\Vert\,\sum_{z \in \mathfrak{F}_{k}}|\nu_i(z)-\nu(z)| + \Vert a\Vert\,\left(1-\sum_{z\in\mathfrak{F}_{k}}\nu(z)\right),
\end{align*}
implying that $\lim_i \omega[\nu_i](a) = \omega[\nu](a)$. Since $\bigcup_n A_n$ is norm dense in $A$, it follows that $\omega[\nu_i] \to \omega[\nu]$ in the weak$^*$ topology. Since $\nu \mapsto \omega[\nu]$ is the inverse map of $\omega \mapsto \nu[\omega]$ as seen above, we conclude that $\omega \mapsto \nu[\omega]$ is a homeomorphism. 
\end{proof}

Thanks to the above proposition, we can directly apply Okounkov--Olshanski's abstract approximation theorem \cite[Theorem 6.1]{OkounkovOlshanski:IMRN98} to the analysis of $K_\beta^\mathrm{ln}(\alpha^t)$, and obtain Theorem \ref{T7.8} without requiring that $z(m+1)z(m) \neq 0$ for all $m$. 

\medskip
In closing of this section, we should point out that our treatment of ergodic method has not touched the notion of \emph{central measures} (coined by Vershik--Kerov \cite[chapter 1, section 1]{Kerov:Book}) or \emph{Gibbs measures} (by Borodin--Olshanski \cite[section 7.4]{BorodinOlshanski:Book}, but their formulation is more restricted than the former). However, the measure $\gamma_\chi$ with $\chi:=r_{GZ}(\omega)$ (that appeared around Lemmas \ref{L7.6}-\ref{L7.7}) is nothing but such a measure. In fact, Lemma \ref{L7.5}(3) shows that 
\[
\gamma_{r_{GZ}(\omega)}([z_n,\dots,z_1]_n^1) = \nu[\omega](z_n)\,\kappa_{(\alpha^t,\beta)}(z_n,z_{n-1})\cdots\kappa_{(\alpha^t,\beta)}(z_2,z_1)
\]
holds true, which is a key in the proof of Lemma \ref{L7.7}. Actually, this equation gives a one-to-one correspondence between the central measures and $H_1^+(\kappa_{(\alpha^t,\beta)})$.   

\section{Spectral decomposition} 

We keep the notations in the previous section and all the assumptions there (including that the flow $\alpha^t$ satisfies the equivalent continuity conditions in Lemma \ref{L7.1}). The goal is the irreducible decomposition for any locally bi-normal $(\alpha^t,\beta)$-spherical (type) representation of $A$. The decomposition is unique because $\Pi(A^{(2)})'$ is always commutative for such a representation $\Pi : A^{(2)} \curvearrowright \mathcal{H}_\Pi$. 

In what follows, all the statements below claim nothing when $K_\beta^\mathrm{ln}(\alpha^t) = \emptyset$. The non-triviality of $K_\beta^\mathrm{ln}(\alpha^t)$ should be studied in concrete examples. 

\medskip
Regarding each $A_n$ as a von Neumann algebra on $\mathcal{H}_n := \bigoplus_{z\in\mathfrak{Z}_n} \mathcal{H}_z$, we write $\mathfrak{A}_n := A_n \cap K(\mathcal{H}_z) \vartriangleleft A_n$, the compact operators in $A_n$.  It is clear that $\mathfrak{A}_n$ is a separable norm-closed ideal of $A_n$ and its $\sigma$-weak closure is $A_n$. Let $\mathfrak{A}$ be the $C^*$-algebra generated by the $\mathfrak{A}_n$, which is separable and unital since $\mathfrak{A}_0 = A_0 = \mathbb{C}1$. In particular, the state space $S(\mathfrak{A})$ is a compact metrizable convex subset of $\mathfrak{A}^*$ endowed with  the weak$^*$ topology. The $C^*$-algebra $\mathfrak{A}$ should be called  \emph{Stratila-Voiculescu $C^*$-algebra} associated with the inductive system $A_n \hookrightarrow A_{n+1}$. It plays  only a technical role below. Here is a lemma. 

\begin{lemma} \label{L8.1} The following hold{\rm:} 
\begin{itemize} 
\item[(1)] A state $\omega \in S(A)$ is locally normal if and only if its restriction to each $\mathfrak{A}_n$ is of norm $1$. Hence, the $\mathfrak{A}_n$ guarantee that Ruelle's separability condition \rm{S} for $A = \varinjlim A_n$. 
\item[(2)] The topology of pointwise convergence on $\mathfrak{A}$ ($\subset A$), the weak$^*$ topology, and the locally uniform topology {\rm(}see \cite[pages 132--133]{BratteliRobinson:Book1}{\rm)} coincide on the locally normal states $S_\mathrm{ln}(A)$.  
\end{itemize}
\end{lemma}   
 \begin{proof} By the structure of $\mathfrak{A}_n \vartriangleleft A_n$, one can choose a sequence of finite rank (in $\mathcal{H}_n$) projections $p_{n,i}$ in $\mathfrak{A}_n$ converging to $1$ $\sigma$-weakly in $A_n$ as $i\to\infty$.  

\medskip
Item (1): The only if part is clear, since $\omega(p_{n,i}) \to \omega(1) = 1$ as $i \to \infty$ by normality. 

The if part is as follows. Let $n$ be arbitrarily fixed. By assumption, the Hahn-Banach extension theorem (with \cite[Proposition 2.3.11]{BratteliRobinson:Book1}) gives a state $\tilde{\omega}$ on $K(\mathcal{H}_n)$ that is an extension of the restriction of $\omega$ to $\mathfrak{A}_n$. By \cite[Proposition 2.6.13]{BratteliRobinson:Book1} $\tilde{\omega}$ extends to a normal state on $B(\mathcal{H}_n)$. By \cite[Lemma 4.1.33]{BratteliRobinson:Book1} the original state $\omega$ must coincide with $\tilde{\omega}$ on $A_n$ and hence be normal on $A_n$. Since $n$ is arbitrary, $\omega$ is locally normal. 

\medskip
Item (2):  For any pair $\omega_1,\omega_2 \in S(A)$ we have 
\begin{align*}
&|(\omega_1 - \omega_2)(a)| 
\leq 
|(\omega_1 - \omega_2)(p_{n,i} a p_{n,i})| 
+
|(\omega_1 - \omega_2)((1-p_{n,i}) a p_{n,i})| + |(\omega_1 - \omega_2)(a (1-p_{n,i}))| \\
&\leq 
\sup\{ |(\omega_1 - \omega_2)(x)|\,;\, x \in p_{n,i} A_n p_{n,i},\ \Vert x \Vert \leq 1 \} \\
&\qquad\qquad\qquad+ 2\omega_1(1-p_{n,i})^{1/2}+2\omega_2(1-p_{n,i})^{1/2} \qquad \text{(by the Cauchy--Schwarz inequality)} \\
&=
\sup\{ |(\omega_1 - \omega_2)(x)|\,;\, x \in p_{n,i} A_n p_{n,i},\ \Vert x \Vert \leq 1 \} \\
&\qquad\qquad\qquad+ 2|(\omega_2-\omega_1)(p_{n,i})+\omega_2(1-p_{n,i})|^{1/2} +2\omega_2(1-p_{n,i})^{1/2} \\
&\leq 
\sup\{ |(\omega_1 - \omega_2)(x)|\,;\, x \in p_{n,i} A_n p_{n,i},\ \Vert x \Vert \leq 1 \} + 2|(\omega_1-\omega_2)(p_{n,i})|^{1/2} + 4\omega_2(1-p_{n,i})^{1/2}
\end{align*}
for all $a \in A_n$ with $\Vert a \Vert \leq 1$. 

For each subspace $B \subset A$, we write 
\[
\Vert \omega_1 - \omega_2\Vert_B := 
\sup\{ |(\omega_1 - \omega_2)(x)|\,;\, x \in B,\ \Vert x \Vert \leq 1 \}
\]
for simplicity. If $\omega_2$ is normal on $A_n$, then the inequality we provided above shows that, for each $n$ and $\varepsilon>0$ there exists an $i$ (depending only on $\omega_2$) so that 
\[
\Vert \omega_1 - \omega_2\Vert_{A_n} \leq 
\Vert \omega_1 - \omega_2\Vert_{p_{n,i} A_n p_{n,i}} + 2\Vert \omega_1 - \omega_2\Vert_{p_{n,i} A_n p_{n,i}}^{1/2} + \varepsilon. 
\]

Assume that $\omega_j \to \omega$ on $\mathfrak{A}$ pointwisely as $j\to\infty$ with $\omega_j, \omega \in S_\mathrm{ln}(A)$. For each $n$ and $\varepsilon > 0$ let $i$ be chosen as in the above inequality so that $i$ depends on $\omega$ and is independent of $j$. Since $\omega_j \to \omega$ pointwisely as $j\to\infty$ on the finite dimensional subalgebra $p_{n,i} A_n p_{n,i} = p_{n,i} \mathfrak{A}_n p_{n,i} \subset \mathfrak{A}_n$, we have $\lim_j \Vert \omega_j - \omega\Vert_{p_{n,i} A_n p_{n,i}} = 0$. Therefore, $\limsup_j \Vert \omega_j - \omega \Vert_{A_n} \leq \varepsilon$ for all $\varepsilon > 0$, and hence $\lim_j \Vert \omega_j - \omega\Vert_{A_n} = 0$. This shows that $\omega_j \to \omega$ in the locally uniform topology on $A$, which almost trivially implies $\omega_j \to \omega$ as $j \to \infty$ in the weak$^*$ topology.  
\end{proof} 

We have known that every $\alpha_z^t$ has an implementing strongly continuous one-parameter unitary group, and so does every $\alpha_n^t$. It follows that the restriction $\underline{\alpha}^t$ of the flow $\alpha^t$ to $\mathfrak{A}$ is pointwise norm-continuous. Therefore, the $(\underline{\alpha}^t,\beta)$-KMS states $K_\beta(\underline{\alpha}^t)$ form a weak$^*$ compact convex subset of $S(\mathfrak{A})$. See \cite[subsection 5.3.2]{BratteliRobinson:Book2}. 

We denote by $S_\mathrm{ln}(\mathfrak{A})$ all the $\omega \in S(\mathfrak{A})$ whose restrictions to each $J_n$ are of norm $1$, and also set $K_\beta^\mathrm{ln}(\underline{\alpha}^t) := K_\beta(\underline{\alpha}^t) \cap S_\mathrm{ln}(\mathfrak{A})$. The restriction map $r_\mathfrak{A} : S(A) \to S(\mathfrak{A})$ sending each state $\omega \in S(A)$ to its restriction $\underline{\omega}$ to $\mathfrak{A}$ is affine, surjective (thanks to the Hahn--Banach theorem) and continuous in the weak$^*$ topologies. The restriction of $r_\mathfrak{A}$ to $K_\beta(\alpha^t)$ clearly gives an affine map from $K_\beta(\alpha^t)$ \emph{to} $K_\beta(\underline{\alpha}^t)$. In what follows, the set of extreme points in a convex set $K$ is denoted by $\mathrm{ex}(K)$. The next lemma is motivated from \cite[Proposition 4.1.34]{BratteliRobinson:Book1} based on Ruelle's separability condition \rm{S}. 

\begin{lemma} \label{L8.2} 
The following hold{\rm:} 
\begin{itemize} 
\item[(1)] $K_\beta^\mathrm{ln}(\underline{\alpha}^t)$ is a stable face of $K_\beta(\underline{\alpha}^t)$. 
\item[(2)] $K_\beta^\mathrm{ln}(\underline{\alpha}^t)$ is a $G_\delta$ set in $K_\beta(\underline{\alpha}^t)$. 
\item[(3)] $\mathrm{ex}(K_\beta^\mathrm{ln}(\underline{\alpha}^t))$ is a Baire set in $K_\beta(\underline{\alpha}^t)$ and there is a  convex continuous function $f$ on $K_\beta(\underline{\alpha}^t)$ such that 
\[
\mathrm{ex}(K_\beta^\mathrm{ln}(\underline{\alpha}^t)) = \partial_f(K_\beta(\underline{\alpha}^t)) \cap K_\beta^\mathrm{ln}(\underline{\alpha}^t), 
\]
where $\partial_f(K_\beta(\underline{\alpha}^t))$ is the boundary set associated with $f$ {\rm(}see \cite[Definition 4.1.5]{BratteliRobinson:Book1}{\rm)}. 
\item[(4)] The restriction map $r_\mathfrak{A}$ induces an affine homeomorphism from $K_\beta^\mathrm{ln}(\alpha^t)$ onto $K_\beta^\mathrm{ln}(\underline{\alpha}^t)$ equipped with the weak$^*$ topologies. Moreover, for any $\omega \in S_\mathrm{ln}(A)$, $r_\mathfrak{A}(\omega) \in K_\beta(\underline{\alpha}^t)$ implies $\omega \in K_\beta(\alpha^t)$. 
\item[(5)] The restriction of the GNS triple $(\pi_\omega : A \curvearrowright \mathcal{H}_\omega,\xi_\omega)$ associated with any $\omega \in K_\beta^\mathrm{ln}(\alpha^t)$ to $\mathfrak{A}$ is gives that associated with $r_\mathfrak{A}(\omega)$, and $\pi_\omega(A)'' = \pi_\omega(\mathfrak{A})''$ holds. In particular, the GNS Hilbert space associated with any $\omega \in K_\beta^\mathrm{ln}(\alpha^t)$ must be separable. 
\end{itemize}
\end{lemma}    
\begin{proof} 
Items (1)--(3): The proof of \cite[Proposition 4.1.34]{BratteliRobinson:Book1} perfectly works to prove items (1),(3) in the present setup without any changes. Item (2) is a little improvement of the corresponding assertion in \cite[Proposition 4.1.34]{BratteliRobinson:Book1}. Let $\{a_{n,m}\}_m$ be a norm-dense sequence of the unit ball of $\mathfrak{A}_n$ for each $n$. The assertion follows from 
\[
K_\beta^\mathrm{ln}(\underline{\alpha}^t) = \bigcap_n \bigcap_\ell \bigcup_m \big\{ \omega \in K_\beta(\underline{\alpha}^t)\,;\, |\omega(a_{n,m})| > 1 - \ell^{-1} \big\}. 
\]   

\medskip
Item (4): By (the easier direction of) Lemma \ref{L8.1}(1) we have $r_\mathfrak{A}(K_\beta^\mathrm{ln}(\alpha^t)) \subset K_\beta^\mathrm{ln}(\underline{\alpha}^t)$. By local normality, the restriction map $r_\mathfrak{A}$ is injective on $S_\mathrm{ln}(A)$. Hence we have confirmed that the restriction map $r_\mathfrak{A}$ induces an injective affine map from $K_\beta^\mathrm{ln}(\alpha^t)$ to  $K_\beta^\mathrm{ln}(\underline{\alpha}^t)$ that is continuous in the weak$^*$ topologies. 

We then prove that $r_\mathfrak{A}(K_\beta^\mathrm{ln}(\alpha^t)) = K_\beta^\mathrm{ln}(\underline{\alpha}^t)$. Choose an arbitrary $\omega \in K_\beta^\mathrm{ln}(\underline{\alpha}^t)$. Its Hahn--Banach extension $\bar{\omega}$ to $A$ falls into $S_\mathrm{ln}(A)$ with $r_\mathfrak{A}(\bar{\omega}) = \omega$ by Lemma \ref{L8.1}(1). We have to prove that $\bar{\omega} \in K_\beta(\alpha^t)$. 

Let $(\pi_{\bar{\omega}} : A \curvearrowright \mathcal{H}_{\bar{\omega}},\xi_{\bar{\omega}})$ be the GNS triple associated with $\bar{\omega}$. It is easy to see that $\pi_{\bar{\omega}} : A \curvearrowright \mathcal{H}_{\bar{\omega}}$ is locally normal. See the proof of Theorem \ref{T5.7}(1). In particular, we obtain that $\pi_{\bar{\omega}}(A)'' = \pi_{\bar{\omega}}(\mathfrak{A})''$. By \cite[Proposition 5.3.3]{BratteliRobinson:Book2} (also see the discussion above Lemma \ref{L5.4}) we have $\bar{\omega}\circ\alpha^t(a) = \omega\circ\underline{\alpha}^t(a) = \omega(a) = \bar{\omega}(a)$ for all $t \in \mathbb{R}$ and $a \in \mathfrak{A}$. By the local normality of $\bar{\omega}$ and $\alpha^t(A_n) = A_n$, it follows that $\bar{\omega}\circ\alpha^t = \bar{\omega}$ for all $t \in \mathbb{R}$. Thanks to these two facts, a standard argument like the proof of Proposition \ref{P4.3} (iii) $\Rightarrow$ (iv) with the same approximation procedure as in the proof of Lemma \ref{L5.4}(2) enables us to see that $\bar{\omega}$ actually falls into $K_\beta(\alpha^t)$. (See also the proof of Theorem \ref{T5.7}(2), which contains a very similar claim.) 

Finally, Lemma \ref{L8.1}(2) guarantees that $r_\mathfrak{A} : K_\beta^\mathrm{ln}(\alpha^t) \to K_\beta^\mathrm{ln}(\underline{\alpha}^t)$ is homeomorphic in the weak$^*$ topologies. The latter part is clear from the discussion above.

\medskip
Item (5): Let $(\pi_\omega : A \curvearrowright \mathcal{H}_\omega,\xi_\omega)$ be the GNS triple associated with any $\omega \in K_\beta^\mathrm{ln}(\alpha^t)$. Observe that $\pi_\omega(A)'' = \pi_\omega(\mathfrak{A})''$ by local normality. Thus $[\pi_\omega(\mathfrak{A})\xi_\omega] = \mathcal{H}_\omega$, which implies the first part of the desired assertion. The second part follows from that $\mathfrak{A}$ is separable. 
\end{proof}

It is well known that any $G_\delta$ subset of a complete metric space becomes a complete metric space. Here is a corollary of the above lemma. 

\begin{corollary}\label{C8.3} The convex set $K_\beta^\mathrm{ln}(\underline{\alpha}^t)$ equipped with the weak$^*$ topology becomes a complete metric space, and so does the extreme points $\mathrm{ex}(K_\beta^\mathrm{ln}(\underline{\alpha}^t))$ too. The same assertion also holds for $\mathrm{ex}(K_\beta^\mathrm{ln}(\alpha^t)) \subset K_\beta^\mathrm{ln}(\alpha^t)$. 
\end{corollary}
\begin{proof} 
The first part follows from Lemma \ref{L8.2}(2),(3) and \cite[Proposition 4.1.6]{BratteliRobinson:Book1} together with the remark just before this corollary. The second part follows from the first part by Lemma \ref{L8.2}(4). 
\end{proof}     

Here is the desired spectral decomposition result for locally normal $(\alpha^t,\beta)$-KMS states on $A$ as well as locally bi-normal $(\alpha^t,\beta)$-spherical representations of $A$. The result should be regarded as a generalization of the spectral decomposition result for characters of inductive limits of compact groups due to Voiculescu \cite[Th\'eor\`eme 2]{Voiculescu:JMathPuresAppl76}, Olshanski \cite[section 9]{Olshanski:JFA03} (and also Enomoto--Izumi \cite[subsection 2.2]{EnomotoIzumi:JMSJ16}). A similar disintegration is also considered by Longo--Tanimoto \cite[subsection 3.2.1]{LongoTanimoto:CMP18} in the context of algebraic quantum field theory. 

In what follows, we will use only Borel measures rather than Baire measures entirely in what follows. See e.g.\ \cite[subsection 4.1.2]{BratteliRobinson:Book1} for their relationship. 

\begin{theorem}\label{T8.4} The following hold true{\rm:} 
\begin{itemize} 
\item[(1)] Any $\omega \in K_\beta^\mathrm{ln}(\alpha^t)$ admits a unique decomposition into locally normal $(\alpha^t,\beta)$-KMS factor states of $A$. Namely, there is a unique Borel probability measure $\mu_\omega$ on $\mathrm{ex}(K^\mathrm{ln}_\beta(\alpha^t))$, that we call \emph{the spectral measure} of $\omega$, such that 
\[
\omega(a) = \int_{\mathrm{ex}(K^\mathrm{ln}_\beta(\alpha^t))} \chi(a)\,\mu_\omega(d\chi) 
\]
for every $a \in A$, where $\mathrm{ex}(K_\beta^\mathrm{ln}(\alpha^t))$ is equipped with the weak$^*$ topology. 
\item[(2)] The spectral measure $\mu_\omega$ gives the direct intergral decomposition 
\[
(\pi_\omega : A \curvearrowright \mathcal{H}_\omega) = \int^\oplus_{\mathrm{ex}(K_\beta^\mathrm{ln}(\alpha^t))} (\pi_\chi : A \curvearrowright \mathcal{H}_\chi)\,\mu_\omega(d\chi)
\]
with respect to the center $\mathcal{Z}(\pi_\omega(A)'')$, where $\pi_\chi : A \curvearrowright \mathcal{H}_\chi$ is the GNS representation associated with $\chi \in \mathrm{ex}(K_\beta^\mathrm{ln}(\alpha^t))$. This gives the central decomposition of $\pi_\omega(A)''$.  
\item[(3)] The spectral measure $\mu_\omega$ gives the direct integral decomposition 
\[
(\pi_\omega^{(2)} : A^{(2)} \curvearrowright \mathcal{H}_\omega) = 
\int^\oplus_{\mathrm{ex}(K_\beta^\mathrm{ln}(\alpha^t))} (\pi_\chi^{(2)} : A^{(2)} \curvearrowright \mathcal{H}_\chi)\,\mu_\omega(d\chi)
\]
with respect to the commutant $\pi_\omega^{(2)}(A^{(2)})'$. This is a unique irreducible decomposition.  
\item[(4)] Conversely, any Borel probability measure $\nu$ on $\mathrm{ex}(K_\beta^\mathrm{ln}(\alpha^t))$ gives a locally normal $(\alpha^t,\beta)$-KMS state on $A$ by the integral formula of item {\rm(1)} with the given measure $\nu$ in place of $\mu_\omega$. The mapping $\omega \mapsto \mu_\omega$ established in item {\rm(1)} gives an affine isomorphism from $K_\beta^\mathrm{ln}(\alpha^t)$ onto all the Borel probability measures on $\mathrm{ex}(K_\beta^\mathrm{ln}(\alpha^t))$. 
\end{itemize}
\end{theorem}

The proof will show that the uniqueness of $\mu_\omega$ in item (1) follows from that the integral formula holds only for any $a \in \mathfrak{A}$.

\begin{proof} 
We first remark that any $\chi \in \mathrm{ex}(K_\beta^\mathrm{ln}(\alpha^t))$ is a factor state by Proposition \ref{P5.8}. 

Item (1): (Existence) Two general facts \cite[Theorem 5.3.30(2)]{BratteliRobinson:Book2} and \cite[Theorem 4.1.15]{BratteliRobinson:Book1} show that $r_\mathfrak{A}(\omega)$ is the barycenter of a unique maximal measure $\mu_{r_\mathfrak{A}(\omega)}$ on $K_\beta(\underline{\alpha}^t)$. 

By Lemma \ref{L8.2}(1), we see that $\mu_{\underline{\omega}}$ is supported by $K_\beta^\mathrm{ln}(\underline{\alpha}^t)$. By \cite[Theorem 4.1.7]{BratteliRobinson:Book1} $\mu_{\underline{\omega}}$ is also supported by every boundary set $\partial_f(K_\beta(\underline{\alpha}^t))$ with convex continuous function $f$. This fact and Lemma \ref{L8.2}(3) show that $\mu_{\underline{\omega}}$ is supported by $\mathrm{ex}(K_\beta^\mathrm{ln}(\underline{\alpha}^t))$. See the remark following the proof of \cite[Theorem 4.1.11]{BratteliRobinson:Book1}. Hence, we may and do regard $\mu_{\underline{\omega}}$ as a measure on $\mathrm{ex}(K_\beta^\mathrm{ln}(\underline{\alpha}^t))$. We consider the push-forward (by $r_\mathfrak{A}^{-1}$) measure $\mu_\omega := \mu_{\underline{\omega}}(r_\mathfrak{A}(\,\cdot\,))$, a Borel probability measure on $\mathrm{ex}(K_\beta^\mathrm{ln}(\alpha^t))$. 

Since $r_\mathfrak{A}(\omega)$ is the barycenter of $\mu_{r_{\mathfrak{A}}(\omega)}$, we have, by the change of variable $\underline{\chi} = r_\mathfrak{A}(\chi)$,  
\[
f(r_{\mathfrak{A}}(\omega)) = \int_{K_\beta(\underline{\alpha}^t)} f(\underline{\chi})\,\mu_{r_{\mathfrak{A}(\omega)}}(d\underline{\chi})  = \int_{\mathrm{ex}(K_\beta^\mathrm{ln}(\alpha^t))} f(r_\mathfrak{A}(\chi))\,\mu_\omega(d\chi)
\]
for every real continuous affine function $f$ on $K_\beta(\underline{\alpha}^t)$. Applying this integral formula to any real affine continuous function $\underline{\chi} \in K_\beta(\underline{\alpha}^t) \mapsto \underline{\chi}(a) \in \mathbb{R}$ with $a = a^* \in \mathfrak{A}$, we obtain that 
\begin{equation}\label{Eq8.1}
\omega(a) = \int_{\mathrm{ex}(K_\beta^\mathrm{ln}(\alpha^t))} \chi(a)\,\mu_\omega(d\chi) 
\end{equation}
for all $a = a^* \in \mathfrak{A}$ and then for all $a \in \mathfrak{A}$ by linearity. Let $p_{n,i}$ be the sequence of projections at the beginning of the proof of Lemma \ref{L8.1}(1). For any $a \in A_n$, we observe that $a_i := p_{n,i} a p_{n,i} \in \mathfrak{A}_n$ and that $\chi(a_i) \to \chi(a)$ as $i\to\infty$ (by local normality) and  $|\chi(a_i)| \leq \Vert a \Vert$ for any $\chi \in K_\beta^\mathrm{ln}(\alpha^t)$. Therefore, by the dominated convergence theorem, formula \eqref{Eq8.1} is still valid for any $a \in \bigcup_n A_n$. Since $\bigcup_n A_n$ is norm-dense in $A$, formula \eqref{Eq8.1} finally holds for every $a \in A$ by the dominated convergence theorem.     

(Uniqueness) We have another Borel probability measure $\nu$ on $\mathrm{ex}(K_\beta^\mathrm{ln}(\alpha^t))$ so that 
\begin{equation}\label{Eq8.2}
\omega(a) = \int_{\mathrm{ex}(K_\beta^\mathrm{ln}(\alpha^t))} \chi(a)\,\nu(d\chi)
\end{equation}
for all $a \in A$. Thanks to Lemma \ref{L8.2}(4) the push-forward (by $r_\mathfrak{A}$) measure $\underline{\nu} := \nu(r_\mathfrak{A}^{-1}(\,\cdot\,))$ defines a Borel probability measure on $\mathrm{ex}(K_\beta^\mathrm{ln}(\underline{\alpha}^t))$. By Lemma \ref{L8.2}(3) we may and do regard $\underline{\nu}$ as a Borel probability measure on the whole $K_\beta(\underline{\alpha}^t)$. By Lemma \ref{L8.2}(1), $\mathrm{ex}(K_\beta^\mathrm{ln}(\underline{\alpha}^t)) \subset \mathrm{ex}(K_\beta(\underline{\alpha}^t))$, and hence \cite[Theorem 4.1.11]{BratteliRobinson:Book1} shows that $\underline{\nu}$ is a maximal measure. Since any real affine continuous function on $K_\beta(\underline{\alpha}^t)$ can uniformly be approximated by functions $\underline{\chi} \mapsto \underline{\chi}(a)$ with $a = a^* \in \mathfrak{A}$ (a simple application of Hahn--Banach separation theorem; see e.g.\ \cite[Corollary III.6.3]{Takesaki:Book1}), formula \eqref{Eq8.2} shows that the barycenter of $\underline{\nu}$ is the $r_\mathfrak{A}(\omega)$ thanks to \cite[Proposition 4.1.1]{BratteliRobinson:Book1}. Since $K_\beta(\underline{\alpha}^t)$ is a simplex by \cite[Theorem 5.3.30]{BratteliRobinson:Book1}, we conclude, by \cite[Theorem 4.1.15]{BratteliRobinson:Book1}, that $\underline{\nu} = \mu_{r_\mathfrak{A}(\omega)}$, implying $\nu = \mu_\omega$ as desired.

\medskip
Item (2): We remark that $\mu_\omega$ and $\mu_{r_\mathfrak{A}(\omega)}$ are  standard measures thanks to Corollary \ref{C8.3} (see the beginning of \cite[subsection 4.4.1]{BratteliRobinson:Book2}) and hence the direct integral theory based on standard Borel spaces can be used without any difficulty. 

Let $(\pi_\chi : A \curvearrowright \mathcal{H}_\omega,\xi_\omega)$ be the GNS triple associated with any $\chi\in K_\beta^\mathrm{ln}(\alpha^t)$. By Lemma \ref{L8.2}(5) its restriction (to $\mathfrak{A}$) $(\pi_\chi : \mathfrak{A} \curvearrowright \mathcal{H}_\chi,\xi_\chi)$ is the GNS triple associated with $r_\mathfrak{A}(\chi)$ and $\pi_\chi(\mathfrak{A})'' = \pi_\chi(A)''$ holds. 

The push-forward measure $\mu_{r_\mathfrak{A}(\omega)}$ of $\mu_\omega$ by $r_\mathfrak{A}$ is the central measure thanks to \cite[Theorem 5.3.30(5)]{BratteliRobinson:Book2}. As explained in \cite[subsection 4.4.2]{BratteliRobinson:Book1}, the mappings  $\chi \in K_\beta^\mathrm{ln}(\alpha^t) \mapsto \mathcal{H}_\chi, \pi_\chi(a)$ with $a \in A$ are measurable, and so are $\underline{\chi} \in K_\beta^\mathrm{ln}(\underline{\alpha}^t) \mapsto \mathcal{H}_{r_\mathfrak{A}^{-1}(\underline{\chi})}, \pi_{r_\mathfrak{A}^{-1}(\underline{\chi})}(a)$. Therefore, by (the proof of) \cite[Theorem 4.4.9]{BratteliRobinson:Book1}, we have the direct integral decomposition
\[
(\pi_\omega : \mathfrak{A} \curvearrowright \mathcal{H}_\omega) = \int^\oplus_{\mathrm{ex}(K_\beta^\mathrm{ln}(\underline{\alpha}^t))} (\pi_{r_\mathfrak{A}^{-1}(\underline{\chi})} : \mathfrak{A} \curvearrowright \mathcal{H}_{r_\mathfrak{A}^{-1}(\underline{\chi})})\,\mu_{r_\mathfrak{A}(\omega)}(d\underline{\chi}), 
\] 
with respect to $\mathcal{Z}(\pi_\omega(\mathfrak{A})'') = \mathcal{Z}(\pi_\omega(A)'')$. As in item (1), using the local normality of states involved in the disintegration, we can see that the above integral decomposition still holds true for $A$ in place of $\mathfrak{A}$. By the change of variable $\underline{\chi} = r_\mathfrak{A}(\chi)$ with $\chi \in \mathrm{ex}(K_\beta^\mathrm{ln}(\alpha^t))$ we obtain that the direct integral decomposition: 
\[
(\pi_\omega : A\curvearrowright \mathcal{H}_\omega) = \int^\oplus_{\mathrm{ex}(K_\beta^\mathrm{ln}(\alpha^t))} (\pi_\chi : A \curvearrowright \mathcal{H}_\chi)\,\mu_\omega(d\chi), 
\] 
with respect to $\mathcal{Z}(\pi_\omega(A)'')$. By \cite[Corollary 4.4.8]{BratteliRobinson:Book1}, we have the desired central decomposition of $\pi_\omega(A)''$. 

\medskip
Item (3): In what follows, we denote by $J_\chi$ the modular conjugation operator associated with $\xi_\chi$ for any $\chi \in K_\beta^\mathrm{ln}(\alpha^t)$. It is well known, see e.g.\ \cite[section 10.6]{StratilaZsido:Book},   that each $\pi_\chi(A)''\xi_\chi$ is a full left Hilbert algebra. So, we can apply \cite[Theorems VI.3.14--15]{Takesaki:Book2} to the measurable field $\chi \mapsto \pi_\chi(A)''\xi_\chi$, and then \cite[Lemma VI.3.2]{Takesaki:Book2} implies that $\chi \mapsto J_\chi$ is measurable so that 
\[
J_\omega = \int^\oplus_{\mathrm{ex}(K_\beta^\mathrm{ln}(\alpha^t))} J_\chi\,\mu_\omega(d\chi)
\]
holds. Thus, for any $a\otimes b^\mathrm{op} \in A^{(2)}$, $\chi \mapsto \pi_\chi^{(2)}(a\otimes b^\mathrm{op}) = \pi_\chi(a)J_\chi\pi_\chi(b^*)J_\chi$ is measurable, and hence so is $\chi \mapsto \pi_\chi^{(2)}(x)$ for every $x \in A^{(2)}$. Observe that  
\begin{align*}
\pi_\omega^{(2)}(a\otimes b^\mathrm{op}) 
&= 
\pi_\omega(a)J_\omega\pi_\omega(b^*)J_\omega \\
&= 
\int^\oplus_{\mathrm{ex}(K_\beta^\mathrm{ln}(\alpha^t))} \pi_\chi(a)J_\chi\pi_\chi(b^*)J_\chi\,\mu_\omega(d\chi) \\
&=
\int^\oplus_{\mathrm{ex}(K_\beta^\mathrm{ln}(\alpha^t))} \pi_\chi^{(2)}(a\otimes b^\mathrm{op})\,\mu_\omega(d\chi)
\end{align*}
for any $a\otimes b^\mathrm{op} \in A^{(2)}$. It follows, by taking linear combination and norm-approximation from simple tensors in $A^{(2)}$ to generic elements, that 
\[
(\pi_\omega^{(2)} : A^{(2)} \curvearrowright \mathcal{H}_\omega) = 
\int^\oplus_{\mathrm{ex}(K_\beta^\mathrm{ln}(\alpha^t))} (\pi_\chi^{(2)} : A^{(2)} \curvearrowright \mathcal{H}_\chi)\,\mu_\omega(d\chi) 
\] 
holds. 

Finally, we remark that $\mathcal{Z}(\pi_\omega(A)'') = \pi_\omega^{(2)}(A^{(2)})'$, which means that the above decomposition is a unique irreducible decomposition. 

\medskip
Item (4): Consider the functional 
\[
a \in A \mapsto \omega(a) := \int_{\mathrm{ex}(K_\beta^\mathrm{ln}(\alpha^t))} \chi(a)\,\nu(d\chi) \in \mathbb{C}. 
\]
Let $p_{n,i}$ be the sequence of projections in the proof of Lemma \ref{L8.1}(1). By the dominated convergence theorem we have $\omega(p_{n,i}) \to 1$ as $i\to\infty$. Thanks to Lemma \ref{L8.1}(1), we conclude that $\omega$ is locally normal. 

We then confirm that $\omega \in K_\beta(\alpha^t)$. It suffices to prove $r_\mathfrak{A}(\omega) \in K_\beta(\underline{\alpha}^t)$ by Lemma \ref{L8.2}(4). Let $a,b \in \mathfrak{A}_{\underline{\alpha}^t}^\infty$ be arbitrarily chosen. Then we easily see that $\chi(ab) = r_\mathfrak{A}(\chi)(ab) = r_\mathfrak{A}(\chi)(b\underline{\alpha}^{i\beta}(a)) = \chi(b\underline{\alpha}^{i\beta}(a))$ for all $\chi \in K_\beta(\alpha^t)$ and hence
\[ 
\omega(ab) 
= 
\int_{\mathrm{ex}(K_\beta^\mathrm{ln}(\alpha^t))} \chi(ab)\,\nu(d\chi)
= 
\int_{\mathrm{ex}(K_\beta^\mathrm{ln}(\alpha^t))} \chi(b\underline{\alpha}^{i\beta}(a))\,\nu(d\chi) 
=
\omega(b\underline{\alpha}^{i\beta}(a)),  
\]
implying $r_\mathfrak{A}(\omega) \in K_\beta(\underline{\alpha}^t)$. The last statement is now clear. 
\end{proof} 

We denote by $\mathscr{M}(\mathrm{ex}(K_\beta^\mathrm{ln}(\alpha^t)))$ the Borel probability measures on $\mathrm{ex}(K_\beta^\mathrm{ln}(\alpha^t))$. Theorem \ref{T8.4}(4) together with Theorem \ref{T5.7}(1) gives the bijection
\begin{equation}\label{Eq8.3}
\begin{matrix}
\mathscr{M}(\mathrm{ex}(K_\beta^\mathrm{ln}(\alpha^t))) & \longleftrightarrow & \mathrm{Rep}_\beta^\mathrm{lbin}(\alpha^t)\\
\mu_\omega & \longleftrightarrow & [(\pi_\omega^{(2)} : A^{(2)} \curvearrowright \mathcal{H}_\omega,\xi_\omega)]. 
\end{matrix}
\end{equation}
We are ready to continue the discussion at the end of section 6. 

\begin{theorem}\label{T8.5} The mapping
\[
\nu \in \mathscr{M}(\mathrm{ex}(K_\beta^\mathrm{ln}(\alpha^t))) \longmapsto (\Pi_\nu : A^{(2)} \curvearrowright \mathcal{H}_\nu) := \int_{\mathrm{ex}((K_\beta^\mathrm{ln}(\alpha^t))}^\oplus (\pi_\chi^{(2)} : A^{(2)} \curvearrowright \mathcal{H}_\chi)\,\nu(d\chi). 
\]
induces a one-to-one correspondence between the equivalence classes $[\nu]$ of $\nu\in\mathscr{M}(\mathrm{ex}((K_\beta^\mathrm{ln}(\alpha^t)))$ and the unitarily equivalent classes of locally bi-normal $(\alpha^t,\beta)$-spherical type representations of $A$. 
\end{theorem} 
\begin{proof}
By means of direct integrals, it is easy to see that two equivalent probability measures $\nu_i$, $i=1,2$, give two unitarily equivalent representations $\Pi_{\nu_i} : A^{(2)} \curvearrowright \mathcal{H}_{\nu_i}$. Namely, $[\nu] \mapsto [\Pi_{\nu} : A^{(2)} \curvearrowright \mathcal{H}_{\nu}]$ between equivalent classes is a well-defined mapping.  

On the other hand, assume that we have two unitarily equivalent, locally bi-normal $(\alpha^t,\beta)$-spherical type representations $\Pi_i : A^{(2)} \curvearrowright \mathcal{H}_{\Pi_i}$, $i=1,2$. By Theorem \ref{T5.7}(1) we may and do assume that $\Pi_i = \pi_{\omega_i}^{(2)}$ and $\mathcal{H}_{\Pi_i} = \mathcal{H}_{\omega_i}$ for some $\omega_i \in K_\beta^\mathrm{ln}(\alpha^t)$, $i=1,2$, and there exists a unitary transform $u : \mathcal{H}_{\omega_1} \to \mathcal{H}_{\omega_2}$ so that $u\pi_{\omega_1}^{(2)}(x) = \pi_{\omega_2}^{(2)}(x)u$ for all $x \in A^{(2)}$. However, we cannot assume $u\xi_{\omega_1} = \xi_{\omega_2}$. 

By Theorem \ref{T8.4}(3) we have
\[ 
(\pi_{\omega_i}^{(2)} : A^{(2)} \curvearrowright \mathcal{H}_{\omega_i}) 
= 
\int_{\mathrm{ex}(K_\beta^\mathrm{ln}(\alpha^t))}^\oplus (\pi_\chi^{(2)} : A^{(2)} \curvearrowright \mathcal{H}_\chi)\,\nu_{\omega_i}(d\chi), \quad i=1,2.
\]
Thus, it suffices to prove that $\nu_{\omega_i}$, $i=1,2$, are equivalent. 

As remarked at the end of section 6, we observe that 
\[
\omega_1(a) 
= 
(\varphi_{(\pi_{\omega_1}^{(2)},\xi_{\omega_1})})_L(a) 
= 
(\varphi_{(\pi_{\omega_2}^{(2)},u\xi_{\omega_1})})_L(a) = 
(\pi_{\omega_2}(a)u\xi_{\omega_1}\,|\,u\xi_{\omega_1})_{\mathcal{H}_{\omega_2}}
\]  
and 
\[
u\xi_{\omega_1} \in [\mathcal{Z}(\pi_{\omega_2}(A)'')\xi_{\omega_2}] = \int_{\mathrm{ex}(K_\beta^\mathrm{ln}(\alpha^t))}^\oplus \mathbb{C}\xi_\chi\,\nu_{\omega_2}(d\chi) \subset \int_{\mathrm{ex}(K_\beta^\mathrm{ln}(\alpha^t))}^\oplus \mathcal{H}_\chi\,\nu_{\omega_2}(d\chi) = \mathcal{H}_{\omega_2}  
\]  
(on which the central decomposition $\pi_{\omega_2}(A)'' = \int_{\mathrm{ex}(K_\beta^\mathrm{ln}(\alpha^t))}^\oplus \pi_\chi(A)'',\nu_{\omega_2}(d\chi)$ is performed, by Theorem \ref{T8.4}(2)). The latter implies that there is a function $f \in L^2(\nu_{\omega_2})$ such that 
\[
 u\xi_{\omega_1} = \int_{\mathrm{ex}(K_\beta^\mathrm{ln}(\alpha^t))}^\oplus f(\chi)\xi_\chi\,\nu_{\omega_2}(d\chi).  
\]
Therefore, the former leads to 
\[
 \int_{\mathrm{ex}(K_\beta^\mathrm{ln}(\alpha^t))} \chi(a)\,\nu_{\omega_1}(d\chi) = \omega_1(a) =  \int_{\mathrm{ex}(K_\beta^\mathrm{ln}(\alpha^t))} \chi(a)\,|f(\chi)|^2\,\nu_{\omega_2}(d\chi)
\]
for all $a \in A$. Since $u\xi_{\omega_1}$ is cyclic for $\pi_{\omega_2}^{(2)}(A^{(2)})$, we easily see that $f(\chi) \neq 0$ for $\nu_{\omega_2}$-almost everywhere, and hence $\nu_{\omega_i}$, $i=1,2$, must be equivalent by the uniqueness of spectral measures; see Theorem \ref{T8.4}(1).   
\end{proof}

The same method as above enables us to interpret the disjointness and subrepresentations for \emph{$(\alpha^t,\beta)$-spherical type} representations of $A$ in terms of spectral measures similarly to \cite[Theorem 9.2]{BorodinOlshanski:Book}. We leave their details to the reader. An interesting, related question is to interpret, in terms of spectral measures, the tensor product of two $(\alpha^t,\beta)$-spherical representations of inductive limits of compact quantum groups in the sense of Ryosuke Sato \cite{Sato:preprint19}. 

The spectral measures can completely be discussed only in terms of $K_\beta^\mathrm{ln}(\alpha^t)$. Therefore, with the help of Theorem \ref{T8.4}, the locally normal $(\alpha^t,\beta)$-spherical representation theory is reduced to the analysis of $K_\beta^\mathrm{ln}(\alpha^t)$. 
We will discuss how to investigate it within the representation-theoretic integrable probability theory. 

\medskip
We introduce a function $\kappa_{(\alpha^t,\beta)} : H_1^+(\kappa_{(\alpha^t,\beta)}) \times \mathfrak{Z} \to [0,1]$ defined by 
\begin{equation}\label{Eq8.4}
\kappa_{(\alpha^t,\beta)}(\nu,z) := \omega[\nu](z) 
\end{equation}
for every $(\nu,z) \in H_1^+(\kappa_{(\alpha^t,\beta)}) \times \mathfrak{Z}$. Here, $\nu \mapsto \omega[\nu]$ is the inverse of the mapping $\omega \mapsto \nu[\omega]$ that appeared in the proof of Proposition \ref{P7.10}. Thanks to the same proposition together with Theorem \ref{T8.4}, the spectral measure $\mu_{\omega[\nu]}$ enjoys the following disintegration:    
\begin{align*} 
\nu(z) = \omega[\nu](z)  
&= \int_{\mathrm{ex}(K_\beta^\mathrm{ln}(\alpha^t))} \chi(z)\,\mu_{\omega[\nu]}(d\chi) \\
&= \int_{\mathrm{ex}(K_\beta^\mathrm{ln}(\alpha^t))} \kappa_{(\alpha^t,\beta)}(\nu[\chi],z)\,\mu_{\omega[\nu]}(d\chi)  
\end{align*}  
for every $(\nu,z) \in H_1^+(\kappa_{(\alpha^t,\beta)}) \times \mathfrak{Z}$. By Proposition \ref{P7.10} and Corollary \ref{C8.3}, $\mathrm{ex}(H_1^+(\kappa_{(\alpha^t,\beta)}))$ becomes a complete metric space with respect to the topology of pointwise convergence.  Then, the push-forward measure $\mu_\nu$ of $\mu_{\omega[\nu]}$ by $\chi \mapsto \nu[\chi]$ defines a probability measure on $\mathrm{ex}(H_1^+(\kappa_{(\alpha^t,\beta)}))$ and enjoys that 
\begin{equation}\label{Eq8.5}
\nu = \int_{\mathrm{ex}(H_1^+(\kappa_{(\alpha^t,\beta)}))}\mu_{\nu}(d\xi)\, \kappa_{(\alpha^t,\beta)}(\xi,\,\cdot\,) 
\end{equation}
holds. (This notation of integral is more natural in this context. In fact, this formula is a natural generalization of identity \eqref{Eq7.6}) 

Let a Borel probability measure $\mu$ on $\mathrm{ex}(H_1^+(\kappa_{(\alpha^t,\beta)})$ be given. Consider the $\nu \in H_1^+(\kappa_{(\alpha^t,\beta)})$ determined by formula \eqref{Eq8.5} with the given $\mu$ in place of $\mu_\nu$. Let $\mu^*$ be the push-forward measure of $\mu$ on $\mathrm{ex}(K_\beta^\mathrm{ln}(\alpha^t))$ by $\xi \mapsto \omega[\xi]$. For any $a \in A_n$ we have 
\begin{align*} 
\int_{\mathrm{ex}(K_\beta^\mathrm{ln}(\alpha^t))} \chi(a)\,\mu^*(d\chi) 
&= 
\sum_{z \in \mathfrak{Z}_n(\alpha^t,\beta)} \int_{\mathrm{ex}(K_\beta^\mathrm{ln}(\alpha^t))} \chi(za)\,\mu^*(d\chi) \\
&= 
\sum_{z \in \mathfrak{Z}_n(\alpha^t,\beta)} \int_{\mathrm{ex}(K_\beta^\mathrm{ln}(\alpha^t))} \tau^{(\alpha_z^t,\beta)}(za)\chi(z)\,\mu^*(d\chi) \quad \text{(by Lemma \ref{L7.3})} \\
&=
\sum_{z \in \mathfrak{Z}_n(\alpha^t,\beta)} \tau^{(\alpha_z^t,\beta)}(za) \int_{\mathrm{ex}(K_\beta^\mathrm{ln}(\alpha^t))} \chi(z)\,\mu^*(d\chi) \\
&=
\sum_{z \in \mathfrak{Z}_n(\alpha^t,\beta)} \Big(\int_{\mathrm{ex}(H_1^+(\kappa_{(\alpha^t,\beta)}))} \mu(d\xi)\,\kappa_{(\alpha^t,\beta)}(\xi,z)\Big)\tau^{(\alpha_z^t,\beta)}(za) \\
&= 
\sum_{z \in \mathfrak{Z}_n(\alpha^t,\beta)} \nu(z)\,\tau^{(\alpha_z^t,\beta)}(za) = \omega[\nu](a). 
\end{align*}
It follows, by the uniqueness part of Theorem \ref{T8.4}(1), that $\mu^*=\mu_{\omega[\nu]}$  and hence $\mu=\mu_\nu$. Therefore, the measure $\mu_\nu$ is uniquely determined by formula \eqref{Eq8.5}. 

Moreover, Theorem \ref{T7.8} shows that any $\xi \in \mathrm{ex}(H_1^+(\kappa_{(\alpha^t,\beta)}))$ has a sequence $z(m) \in \mathfrak{Z}_m(\alpha^t,\beta)$ with $z(m+1)z(m)\neq0$ for all $m \geq 1$ so that 
\begin{equation}\label{Eq8.6} 
\kappa_{(\alpha^t,\beta)}(\xi,\,\cdot\,) = \omega[\xi] = \lim_{m\to\infty}\tau^{(\alpha_{z(m)}^t,\beta)}(z(m)\,\cdot\,) = \lim_{m\to\infty}\kappa_{(\alpha^t,\beta)}(z(m),\,\cdot\,)
\end{equation} 
holds. This fact gives the formula:
\begin{equation}\label{Eq8.7}
\kappa_{(\alpha^t,\beta)}(\xi,z') = \sum_{z\in\mathfrak{Z}_n} \kappa_{(\alpha^t,\beta)}(\xi,z)\kappa_{(\alpha^t,\beta)}(z,z')
\end{equation}
for all $(\xi,z')\in\mathrm{ex}(H_1^+(\kappa_{(\alpha^t,\beta)}))\times\mathfrak{Z}_m$ and $n>m$. Thus, for any $\mu \in \mathscr{M}(\mathrm{ex}(K_\beta^\mathrm{ln}(\alpha^t)))$, formula \eqref{Eq8.5} with the push-forward measure of $\mu$ by $\chi \mapsto \nu[\chi]$ in place of $\mu_\nu$ defines an element of $H_1^+(\kappa_{(\alpha^t,\beta)})$. Consequently, formula \eqref{Eq8.5} defines an affine isomorphism between $\mathscr{M}(\mathrm{ex}(H_1^+(\kappa_{(\alpha^t,\beta)})))$ and $H_1^+(\kappa_{(\alpha^t,\beta)})$. 

Summing up the discussion we obtain the next corollary. 

\begin{corollary}\label{C8.6} 
The following assertions hold{\rm:}
\begin{itemize}
\item[(1)] 
The set of extreme points $\mathrm{ex}(H_1^+(\kappa_{(\alpha^t,\beta)}))$ always becomes a complete metric space with respect to the topology of pointwise convergence. Moreover, we have a continuous function $\kappa_{(\alpha^t,\beta)} : \mathrm{ex}(H_1^+(\kappa_{(\alpha^t,\beta)}))\times \mathfrak{Z} \to [0,1]$ such that equations \eqref{Eq8.6},\eqref{Eq8.7} hold. 
\item[(2)]
For any $\nu \in H^+_1(\kappa_{(\alpha^t,\beta)})$ there is a unique $\mu_\nu \in \mathscr{M}(\mathrm{ex}(H_1^+(\kappa_{(\alpha^t,\beta)}))) $ such that formula \eqref{Eq8.5} holds. 
\item[(3)]
The correspondence $\nu \mapsto \mu_\nu$ in item {\rm(2)} defines an affine isomorphism $\nu \in H_1^+(\kappa_{(\alpha^t,\beta)}) \mapsto \mu_\nu \in \mathscr{M}(\mathrm{ex}(H_1^+(\kappa_{(\alpha^t,\beta)})))$. Moreover, this correspondence gives the spectral measure $\omega \mapsto \mu_\omega$ in Theorem \ref{T8.4}{\rm(1)} through the isomorphism provided in Proposition \ref{P7.10}. 
\end{itemize}
\end{corollary}

Item (3) means that it is sufficient to look at $H_1^+(\kappa_{(\alpha^t,\beta)})$, a \emph{commutative} object, to understand the classification of (locally bi-normal) $(\alpha^t,\beta)$-spherical (type) representations of $A$. 

As we will see in the next section, this corollary holds for the projective chain arising from any link. Further study of $H_1^+(\kappa_{(\alpha^t,\beta)})$ should probably be done within the theory of Martin boundaries. See e.g.\ \cite{GorinOlshanski:JFA16} for such an attempt in a concrete example.    

\section{Ergodic method by example} 

This section concerns how to construct and analyze (along the ergodic method) a certain class of flows on $C^*$-inductive limits of atomic $W^*$-algebras with separable predual as in sections 7-8. 

\subsection{Bratteli diagram description}
We start with $A = \varinjlim A_n$ as in section 7. Then, for every $n$, we have  the countable set $\mathfrak{Z}_n$ of central projections in $A_n$. Set $\mathfrak{Z} = \bigsqcup_{n\geq0} \mathfrak{Z}_n$, a disjoint union, which plays the role of vertex set in what follows. Following Bratteli's famous idea \cite{Bratteli:TAMS72}, we will construct a graph whose vertex set is $\mathfrak{Z}$. The edge set $\mathfrak{E} = \bigsqcup_{n\geq1} \mathfrak{E}_n$, a disjoint union, is defined by $\mathfrak{E}_n := \{ zz' \neq 0\,; (z,z') \in \mathfrak{Z}_n\times\mathfrak{Z}_{n-1}\}$, a set of non-zero projections in the setup, for every $n \geq 1$. The range and the source maps $s, r : \mathfrak{E} \to \mathfrak{Z}$ is defined by $r(e)=z$ and $s(e)=z'$ for $e=zz' \in \mathfrak{E}$ with $(z,z') \in \mathfrak{Z}_n\times\mathfrak{Z}_{n-1}$. Namely, 
\[
\mathfrak{Z}_n \ni z = r(e)\overset{e}{\longleftarrow} s(e)=z' \in \mathfrak{Z}_{n-1} \quad \text{if $e = zz' \neq 0$}.
\]    

By construction, the graph $(\mathfrak{Z},\mathfrak{E},r,s)$ abstractly satisfies the following:
\begin{itemize}
\item[(i)] $\mathfrak{Z}=\bigsqcup_{n\geq0}\mathfrak{Z}_n$, $\mathfrak{E}=\bigsqcup_{n\geq1}\mathfrak{E}_n$ are countable with $\mathfrak{Z}_0 = \{1\}$, a singleton. Moreover, $(r(e),s(e)) \in \mathfrak{Z}_n \times\mathfrak{Z}_{n-1}$ for every $e \in \mathfrak{E}_n$ and $n\geq1$. 
\item[(ii)] $e \in \mathfrak{E} \mapsto (r(e),s(e)) \in \mathfrak{Z}\times\mathfrak{Z}$ is injective.  
\item[(iii)] Each $z \in \mathfrak{Z}$ is the source of an edge $e \in \mathfrak{E}$. 
\item[(iv)] Each $z \in \mathfrak{Z}$ is the range of an edge $e \in \mathfrak{E}$. 
\end{itemize} 
Item (i) means that the graph is graded. Item (ii) is a consequence from that we ignore the multiplicity of each component into another. The multiplicities will be understood as a function on the edges. This is a bit different from Bratteli's formalism, but this difference is rather minor. Item (iii) is a consequence of embeddings $A_{n-1} \hookrightarrow A_n$, while item (iv) follows from that all the embeddings $A_{n-1} \hookrightarrow A_n$ are unital. 

\medskip
We have an additional structure over the graph $(\mathfrak{Z},\mathfrak{E},r,s)$. The multiplicity function $m:\mathfrak{Z}\to\mathbb{N}\cup\{\infty\}$ is defined by $m(e):=\dim(e(A_n\cap (A_{n-1})'))^{1/2} \in \mathbb{N}\cup\{\infty\}$ with $e \in \mathfrak{E}_n$. Here, one notices that $\mathfrak{E}_n$ is a complete family of minimal projections of the center $\mathcal{Z}(A_n\cap(A_{n-1})')$ and that $e(A_n\cap(A_{n-1})') \cong B(\ell_2^{m(e)})$ holds for every $e \in \mathfrak{E}_n$, where $\ell_2^{m(e)}$ denotes the $m(e)$-dimensional Hilbert space. We can reconstruct the inductive system $A_{n-1} \hookrightarrow A_n$ from the graph $(\mathfrak{Z},\mathfrak{E},r,s)$ together with multiplicity function $m$ as follows. 

The zeroth algebra is $A_0 = \mathbb{C}$, $1$-dimensional. The first one is $A_1 = \bigoplus_{z\in\mathfrak{Z}_1} B(\ell_2^{m(z)})$ ({\it n.b.}, each $z \in \mathfrak{Z}_1$ also becomes an edge in $\mathfrak{E}_1$ as $z=z1$). Assume that we have constructed $A_1 \hookrightarrow A_2 \hookrightarrow \cdots \hookrightarrow A_{n-1}=\bigoplus_{z\in\mathfrak{Z}_{n-1}}B(\mathcal{H}_z)$ up to the $(n-1)$th algebra. The next $A_n$ with normal embedding $A_{n-1} \hookrightarrow A_n$ can be reconstructed in the following manner: For each $z \in \mathfrak{Z}_n$, we define a Hilbert space 
\begin{equation}\label{Eq9.1}
\mathcal{H}_z := \bigoplus_{e \in \mathfrak{E}_n; r(e)=z} \mathcal{H}_{s(e)}\otimes\ell_2^{m(e)},  
\end{equation} 
and set $A_n = \bigoplus_{z\in\mathfrak{Z}_n} B(\mathcal{H}_z)$. The embedding $A_{n-1} \hookrightarrow A_n$ is given by 
\begin{equation}\label{Eq9.2}
(x_z)_{z\in\mathfrak{Z}_{n-1}} \mapsto \left(\bigoplus_{e \in \mathfrak{E}_n;r(e)=z} x_{s(e)}\otimes 1_{\ell_2^{m(e)}}\right)_{z \in \mathfrak{Z}_n}, 
\end{equation}
which is clearly a unital normal embedding ({\it n.b.}, any $z \in \mathfrak{Z}_n$ is $r(e)$ of an $e \in \mathfrak{E}_n$). In this way, we can construct an inductive system $A_{n-1}\hookrightarrow A_n$ from the graph $(\mathfrak{Z},\mathfrak{E},r,s)$ with multiplicity function $m$.  This inductive system is clearly the same as the original one, since 
\begin{align*}
\big(e A_{n-1} \subset e A_n e\big) 
&\cong \big(B(\mathcal{H}_{s(e)})\otimes\mathbb{C}1_{\ell_2^{m(e)}} \subset B(\mathcal{H}_{s(e)}\otimes\ell_2^{m(e)})\big) 
\end{align*} 
with $e \in \mathfrak{E}_n$ by a well-known fact on unital normal embeddings of $B(\mathcal{H})$ into $W^*$-algebras (see e.g.\ \cite[Proposition V.1.22]{Takesaki:Book1}). 

Clearly, starting with a graded graph satisfying (i)-(iv) with a multiplicity function, we can construct an inductive system of atomic $W^*$-algebras with separable predual in the above way. Consequently, there is a one-to-one correspondence between the graphs satisfying (i)--(iv) with multiplicity functions and the inductive systems of atomic $W^*$-algebras with separable predual starting at the trivial algebra $\mathbb{C}$. 

\medskip
With the description of the inductive system $A_{n-1} \hookrightarrow A_n$ in the previous paragraph, we can construct a flow on $A=\varinjlim A_n$ that satisfies the equivalent continuity conditions in Lemma \ref{L7.1}. We denote by $\mathrm{Flow}(\varinjlim A_n)$ all such flows. 

The construction starts with a system $(w_e^t)_{e \in \mathfrak{E}}$ of strongly continuous one-parameter unitary groups on $\ell_2^{m(e)}$ over $(\mathfrak{Z},\mathfrak{E},r,s,m)$, that is, each $w_e^t$ is a strongly continuous one-parameter unitary group on the Hilbert space $\ell_2^{m(e)}$. Then, two other systems $(u_z^t)_{z\in\mathfrak{Z}}$, $(u_n^t)_{n\geq 0}$ of strongly continuous one-parameter unitary groups in $zA_n$ and $A_n$, respectively, are inductively constructed as follows. On $A_1$, we set $u_{r(e)}^t := w_e^t$ ($e \in \mathfrak{E}_1$), and then
\begin{equation}\label{Eq9.3}
u_z^t = \bigoplus_{e; r(e)=z} u_{s(e)}^t \otimes w_e^t \in \bigoplus_{e;r(e)=z} B(\mathcal{H}_{s(e)}\otimes\ell_2^{m(e)}) \subset B(\mathcal{H}_z) =zA_n, \quad 
u_n^t = (u_z^t)_{z \in \mathfrak{Z}_n} \in A_n 
\end{equation}
for all $z \in \mathfrak{Z}_n$ and $n \geq 2$. The desired flow is obtained as $\alpha^t:=\varinjlim \alpha_n^t$ with $\alpha_n^t := \mathrm{Ad}u_n^t$ that trivially satisfies the equivalent continuity conditions in Lemma \ref{L7.1}. Note that each $\alpha_z^t$ is nothing but $\mathrm{Ad}u_z^t$. 

By the associativity of Hilbert space tensor products and direct sums, we can re-write \eqref{Eq9.1} and \eqref{Eq9.3} as follows. For each $z \in \mathfrak{Z}_n$ ($\subset \mathfrak{Z}$), we have
\begin{equation}\label{Eq9.4}
\begin{gathered} 
\mathcal{H}_z = \bigoplus_{(e_i)_{i=1}^n; r(e_n)=z} \ell_2^{m(e_1)}\otimes\cdots\otimes\ell_2^{m(e_n)}, \\
u_z^t = \bigoplus_{(e_i)_{i=1}^n; r(e_n)=z} w_{e_1}^t\otimes\cdots\otimes w_{e_n}^t, 
\end{gathered}
\end{equation}
where $(e_i)_{i=1}^n$ denotes an edge-path starting at $1$, that is, $e_i \in \mathfrak{E}_i$, $s(e_i) = r(e_{i-1})$ ($2 \leq i \leq n$) and $s(e_1)=1$. 

\medskip
We have given a construction of a suitable flow $\alpha^t$ out of a collection $(w_e^t)_{e\in\mathfrak{E}}$ of strongly continuous one-parameter unitary groups. Namely, we have a mapping $(w_e^t)_{e\in\mathfrak{E}} \mapsto \alpha^t$. 

\begin{lemma} \label{L9.1} Any member $\alpha^t$ in $\mathrm{Flow}(\varinjlim A_n)$ can be obtained in this way. 
\end{lemma}
\begin{proof} One can write $\alpha_z^t = \mathrm{Ad}u_z^t$ for some strongly continuous one-parameter unitary group in $zA_n = B(\mathcal{H}_z)$; see the discussion above Lemma \ref{L7.2}. Then, $u_n^t = \bigoplus_{z \in \mathfrak{Z}_n} u_z^t$ implements $\alpha_n^t$. Observe that $w_n^t := u_n^t u_{n-1}^{-t}$ falls into $A_n\cap A_{n-1}'$ and implements the restriction of $\alpha^t$ to $A_n\cap A_{n-1}'$. By (the same reason as) Lemma \ref{L7.1}, $\alpha^t$ acts on the center $\mathcal{Z}(A_n\cap A_{n-1}')$ trivially, and hence $w_n^t$ and any $e \in \mathfrak{E}_n$ must commute with each other. Hence, $w_e^t := ew_n^t \in e(A_n \cap A_{n-1}') \cong B(\ell_2^{m(e)})$ defines a strongly continuous one-parameter unitary group on $\ell_2^{m(2)}$ for each $e \in \mathfrak{E}_n$. Then, $w_n^t = \bigoplus_{e\in\mathfrak{E}_n} w_e^t$ holds, since $\mathfrak{E}_n$ is the complete family of minimal central projections of $A_n\cap A_{n-1}'$. Since $u_n^t = w_n^t u_{n-1}^t$, we can recursively reconstruct a sequence $u_n^t$ of strongly continuous one-parameter unitary groups from the $w_e^t$, and the $u_n^t$ clearly give back the given flow $\alpha^t$ as the inductive limit of $\mathrm{Ad}u_n^t$.  
\end{proof} 

Clearly, the mapping $(w_e^t)_{e\in\mathfrak{E}} \mapsto \alpha^t$ is not injective. Hence we will introduce an equivalence relation among collections of strongly continuous one-parameter unitary groups, making the mapping injective modulo the equivalence relation. Each $w_e^t$ is uniquely determined up to character $\chi_e^t$ as an implementation of the restriction $\alpha_e^t$ of $\alpha^t$ to $e(A_n\cap (A_{n-1})') \cong B(\ell_2^{m(e)})$. Thus, the original $(w_e^t)_{e\in\mathfrak{E}}$ and another $(\chi_e^t w_e^t)_{e\in\mathfrak{E}}$ with characters $(\chi_e^t)_{e \in \mathfrak{E}}$ define the same collection $(\alpha_e^t)_{e\in\mathfrak{E}}$ of local flows. Then, by equation \eqref{Eq9.4}, it easily follows that the collection $(\chi_e^t)_{e \in \mathfrak{E}}$ should satisfy that 
$\chi_{e_1}^t\cdots\chi_{e_n}^t$ depends only on the end point $r(e_n)$ for any finite edge-path $(e_i)_{i=1}^n$ starting at $1$. Since the collection $(\alpha_z^t)_{z\in\mathfrak{Z}}$ determines $\alpha^t$ uniquely, we conclude that the desired equivalence relation $(w_e^t)_{e\in\mathfrak{E}} \sim (w'_e{}^t)_{e\in\mathfrak{E}}$ should be defined by 
\begin{equation}\label{Eq9.5}
w_e^t (w'_e{}^t)^* = \chi_e^t\,1_{\ell_2^{m(e)}} 
\end{equation}
for some collection $(\chi_e^t)_{e\in\mathfrak{E}}$ of characters that satisfy that $\chi_{e_1}^t\cdots\chi_{e_n}^t$ depends only on the end point $r(e_n)$ for any finite edge-path $(e_i)_{i=1}^n$ starting at $1$ (see equations \eqref{Eq9.4} for the precise meaning). 

\medskip
By the proof of Lemma \ref{L9.1}, one may observe that a given flow $\alpha^t$ in $\mathrm{Flow}(\varinjlim A_n)$ is completely understood by looking at its restriction to each $A_n \cap A_{n-1}'$. Namely, $\{\alpha^t\!\upharpoonright_{A_n\cap A_{n-1}'}\}_{n=1}^\infty$ is more important than $\{\alpha_n^t\}_{n=1}^\infty$.   

\subsection{General scheme of analysis along ergodic method} 

Let $A = \varinjlim A_n$ be an inductive limit $C^*$-algebra of atomic $W^*$-algebras with separable predual, and $(\mathfrak{Z},\mathfrak{E},r,s)$ with $m : \mathfrak{E} \to \mathbb{N}\cup\{\infty\}$ be the associated graph with multiplicity function. As saw before, any member in $\mathrm{Flow}(\varinjlim A_n)$ can be obtained from a collection $(w_e^t)_{e\in\mathfrak{E}}$ of strongly continuous one-parameter unitary groups on $\ell_2^{m(e)}$. By Stone's theorem, each $w_e^t$ has a unique (possibly unbounded) self-adjoint operator $H_e$ on $\ell_2^{m(e)}$ so that  $w_e^t=\exp(itH_e)$ holds. The equivalence relation \eqref{Eq9.5} is encoded in terms of \emph{edge Hamiltonians} $(H_e)_{e\in\mathfrak{E}}$ as follows. Two $(H_e)_{e\in\mathfrak{E}}$ and $(H'_e)_{e\in\mathfrak{E}}$ are equivalent, if there is a collection $(\lambda_e)_{e\in\mathfrak{E}}$ of positive scalars such that $H'_e = H_e + \lambda_e1$ for every $e \in \mathfrak{E}$ and that $\sum_{i=1}^n \lambda_{e_i}$ depends on $r(e_n)$ for any finite edge-path $(e_i)_{i=1}^n$ starting at $1$ (see equations \eqref{Eq9.4} for the precise meaning). 

\medskip
We introduce the \emph{edge partition function} at $e \in \mathfrak{E}$ as follows.  
\begin{equation}\label{Eq9.6}
Z_\beta(e) := \begin{cases} 
\mathrm{Tr}_e(\exp(-\beta H_e)) & (\text{$\exp(-\beta H_e)$ is bounded}), \\
+\infty & (\text{$\exp(-\beta H_e)$ is unbounded}) 
\end{cases} 
\end{equation}
for every $e \in \mathfrak{E}$ ({\it n.b.}\ the canonical trace on a Hilbert space is defined on all the bounded positive operators but allows to take values in $[0,+\infty]$). We also write
\begin{equation}\label{Eq9.7}
Z_\beta(z) := \sum_{(e_i)_{i=1}^n;r(e_n)=z}Z_\beta(e_1)\cdots Z_\beta(e_n)
\end{equation}
for every $z \in \mathfrak{Z}_n$ and $n\geq1$, where $(e_i)_{i=1}^n$ denotes a finite edge-path starting at $1$ as before. The $Z_\beta(z)$, $z \in \mathfrak{Z}$, should be called the \emph{vertex partition function} at $z$. 

\begin{lemma}\label{L9.2} A positive self-adjoint operator 
\[
\rho_{\beta,z} := \bigoplus_{(e_i)_{i=1}^n;r(e_n)=z} \exp(-\beta H_{e_1})\,\bar{\otimes}\,\cdots\,\bar{\otimes}\,\exp(-\beta H_{e_n})
\]
is of trace-class if and only if $Z_\beta(z) <+\infty$. In the case, $\mathrm{Tr}_z(\rho_{\beta,z}) = Z_\beta(z)$, and moreover, if $\beta\neq0$, then 
\begin{align*}
\rho_{\beta,z}^{-it/\beta} 
&= 
\bigoplus_{(e_i)_{i=1}^n;r(e_n)=z} \exp(-\beta H_{e_1})^{-it/\beta}\otimes\cdots\otimes\exp(-\beta H_{e_n})^{-it/\beta} \\
&
= \bigoplus_{(e_i)_{i=1}^n;r(e_n)=z} w_{e_1}^t\otimes\cdots\otimes w_{e_n}^t = u_z^t
\end{align*}
holds. 
\end{lemma}
\begin{proof} Assume that some $\exp(-\beta H_{e_i})$ appearing in the expression of $\rho_{\beta,z}$ is unbounded. Then, $Z_\beta(z) = +\infty$ and $\rho_{\beta,z}$ is unbounded. Hence we may and do assume that all $\exp(-\beta H_{e_i})$ are bounded. One can easily calculate $\mathrm{Tr}_z(\rho_{\beta,z}) = Z_\beta(z)$ under the assumption. 

The last formula can be easily seen by taking the Schmidt decomposition of each $\exp(-\beta H_{e_i})$ and then by using the function calculus of each $H_{e_i}$ for example. 
\end{proof} 

Similarly to the usual Fock space analysis, one can compute the \emph{vertex Hamiltonian} $H_z$ at $z$, i.e., the generator of $u_z^t$, as follows.  
\[
H_z = \bigoplus_{(e_i)_{i=1}^n;r(e_n)=z} \Big(H_{e_1}\,\bar{\otimes}\,1\,\bar{\otimes}\cdots\bar{\otimes}\,1 + \cdots + 1\,\bar{\otimes}\cdots\bar{\otimes}\,1\,\bar{\otimes}\,H_{e_n}\Big),   
\] 
and we observe that $\rho_{\beta,z} = \exp(-\beta H_z)$. However, the description of $\rho_{\beta,z}$ in the above lemma is probably easier to use than the above description of $H_z$ for our purpose.  

Remark that, if all the edge-paths starting at $1$ and ending at $z$ on the graph $(\mathfrak{Z},\mathfrak{E},r,s)$ form a finite set for every $z\in\mathfrak{Z}$ and if all the $m(e)$ are finite, then all $Z_\beta(z)$ must be finite. 

\medskip
Denote the resulting flow on $A = \varinjlim A_n$ by $\alpha^t$ as usual. In what follows, we will freely use the notations in section 7. By Lemma \ref{L9.2} we observe that $\mathfrak{Z}_n(\alpha^t,\beta)=\{z \in \mathfrak{Z}_n\,; Z_\beta(z) < +\infty\}$ holds for every $n$. 

Thanks to the last formula in Lemma \ref{L9.2}, the local density operator $\rho(\alpha_z^t,\beta)$ at $z\in\mathfrak{Z}(\alpha^t,\beta)$ in section 7 is computed as 
\begin{equation}\label{Eq9.8}
\rho(\alpha_z^t,\beta) = \frac{1}{Z_\beta(z)}\rho_{\beta,z},
\end{equation}
giving the explicit description of $\tau^{(\alpha_z^t,\beta)}$. Moreover, taking formula \eqref{Eq9.2} into consideration, one can easily see that the link $\kappa_{(\alpha^t,\beta)}(z,z')$ is given by
\begin{equation}\label{Eq9.9}
\begin{aligned}
&\kappa_{(\alpha^t,\beta)}(z,z') \\
&= 
\begin{cases} 
{\displaystyle \frac{1}{Z_\beta(z)}\sum_{\substack{(e)_{i=m+1}^n; \\ r(e_n)=z, s(e_{m+1})=z'}} Z_\beta(e_{m+1})\cdots Z_\beta(e_n)} & (Z_\beta(z) < +\infty), \\
0 & (Z_\beta(z) = +\infty)
\end{cases} 
\end{aligned}
\end{equation}
for any pair $(z,z') \in \mathfrak{Z}_n\times\mathfrak{Z}_m$ with $n>m\geq0$. 

This means that the structure of $K_\beta^\mathrm{ln}(\alpha^t)$ depends only on the edge partition functions $(Z_\beta(e))_{e\in\mathfrak{E}}$ rather than the edge Hamiltonians $(H_e)_{e\in\mathfrak{E}}$, while the description of each member in $K_\beta^\mathrm{ln}(\alpha^t)$ needs $(H_e)_{e\in\mathfrak{E}}$ itself. 

Remark that \emph{only positive rational numbers appear in the values of links when $\beta=0$, that is, the tracial state case}, because the edge partition functions $Z_\beta(e)$ are exactly dimensions $m(e)$.    

\medskip 
Here is a relation between the investigation so far and Ryosuke Sato's work \cite{Sato:JFA19}. 

\begin{remark}\label{R9.3} 
When the graph is the Gelfand-Tsetlin graph {\rm(}that is known to be multiplicity-free, i.e., $m\equiv 1${\rm)}, $\beta=-1$, and the $Z_\beta(e)$ are \cite[equation (3.3)]{Sato:JFA19}, the investigation so far can be applied to Sato's formulation of quantized characters of $\mathrm{U}_q(\infty)$ and his algebra is nothing but the associated $\mathfrak{A}$ {\rm(}see section 8{\rm)} of $A = \varinjlim A_n$ {\rm(}and the locally normal $W^*$-envelope $M$ of $A$ is the $W^*$-algebra in his formulation of $\mathrm{U}_q(\infty)${\rm;} see \cite{Sato:preprint19}{\rm)}, where $A_n = W^*(\mathrm{U}_q(n))$, the group $W^*$-algebra of the quantum unitary group $\mathrm{U}_q(n)$ of rank $n$, whose algebraic structure is known to be independent of the choice of $0 < q \leq 1$. The flow generated by the {\rm(}non-unitary{\rm)} antipode {\rm(}as pointed out in section 1{\rm)} is nothing but $\alpha^t$ in the context here. It is known that the flow certainly remembers the $q$-deformation. We remark that each vertex partition function $Z_\beta(z)$ plays a role of the $q$-dimension in his setting. 

As Sato explained, the extreme points $\mathrm{ex}(H_1^+(\kappa_{(\alpha^t,\beta)})) \cong \mathrm{ex}(K_\beta^\mathrm{ln}(\alpha^t))$ were already determined explicitly by Gorin \cite{Gorin:AdvMath12} in the case of $\mathrm{U}_q(\infty)$. 

Note that any $\pi_\chi(A)'' = \pi_\chi(\mathfrak{A})''$ with $\chi\in \mathrm{ex}(K_\beta^\mathrm{ln}(\alpha^t))$ is a hyperfinite factor in the setting here, and hence its algebraic structure is completely determined by the well-known algebraic invariants {\rm(}Murray--von Neumann--Connes type, $S$ and  $T$-sets, and flows of weights{\rm)} thanks to the classification of hyperfinite factors due to Murray--von Neumann, Connes, Krieger and Haagerup. Sato \cite{Sato:ETDS2x} determined the type in the case of $\mathrm{U}_q(\infty)$ completely. His consequence is that only the type $I_1$ {\rm(}i.e., the trivial algebra case{\rm)}, the unique type $I_\infty$ and the unique hyperfinite type III$_{q^2}$ factors arise. 

The $(q,t)$-central probability measures can also be treated in the same way, see \cite[Appendix B]{Sato:JFA19}, but the computation of $\mathrm{ex}(H_1^+(\kappa_{(\alpha^t,\beta)}))$ has only been made for special $(q,t)$ by Cuenca \cite{Cuenca:SIGMA18}.    
\end{remark}

For the reader's convenience, we explain how to apply the general theory of $(\alpha^t,\beta)$-spherical representations to inductive limits of compact quantum groups as a complement to the previous remark.   

\begin{remark}\label{R9.4} Let $G_n$ be an inductive system of compact quantum groups {\rm(}e.g.\ $G_n = \mathrm{U}_q(n)${\rm)}. For the discussion of unitary representations, it is convenient to use the group $W^*$-algebras $A_n:=W^*(G_n)$ associated with $G_n$ like \cite{Yamagami:CMP95} whose formulation fits the discussion here most. They are nothing but the `dual Hopf $W^*$-algebras' of  Woronowicz's coordinate algebras $C(G_n)$. Each $A_n$ is equipped with a unitary antipode $\tau_n$ {\rm(}a $*$-anti-automorphism{\rm)} and a flow $\theta_{n,t}$, called the `deformation flow', {\rm(}which is continuous in the $u$-topology{\rm)}, where we use the notation and parametrization in \cite{Yamagami:CMP95} that are slightly different from \cite{MasudaNakagami:PRIMS94} and \cite{Sato:JFA19,Sato:preprint19}. For each $n$ there is a canonical unital $*$-subalgebra, say $\mathcal{A}_n=\mathbb{C}[G_n]$, of $A_n$, whose elements are all $\theta_{n,t}$-analytic. Moreover, $\sigma_n := \tau_n\circ\theta_{n,-i} = \theta_{n,-i}\circ\tau_n$ defines the `standard' antipode on $\mathcal{A}_n$ {\rm(}the `inverse operation' on $G_n${\rm)}. We can take the inductive limits $\tau, \theta_t$ of $\tau_n, \theta_{n,t}$, respectively, which are defined on the whole $A = \varinjlim A_n$. See \cite{Sato:JFA19}.

Let $\Pi : A\otimes_\mathrm{max}A \curvearrowright \mathcal{H}_\Pi$ be a $*$-representation. The mapping $a\otimes b^\mathrm{op} \in A^{(2)} \mapsto \Pi(a\otimes\tau(b)) \in B(\mathcal{H}_\Pi)$ defines a $*$-representation $\widetilde{\Pi} : A^{(2)} \curvearrowright \mathcal{H}_\Pi$. For each $a \in \mathcal{A}_n \subset A$, we have 
\[
\Pi(1\otimes \sigma_n^{-1}(a)) = \Pi(1\otimes (\tau_n\circ\theta_{n,i}(a))) = \Pi(1\otimes (\tau\circ\theta_{i}(a)))= \widetilde{\Pi}(1\otimes(\theta_{i}(a))^\mathrm{op}),  
\]
and hence 
\[
\Pi(\sigma_n(a)\otimes1)\xi=\Pi(1\otimes a)\xi \quad \Longleftrightarrow \quad \widetilde{\Pi}(a\otimes1)\xi=\widetilde{\Pi}(1\otimes(\theta_{i}(a))^\mathrm{op})\xi
\]
for a vector $\xi \in \mathcal{H}_\Pi$. In this way, we can translate the relation $\pi(g^{-1},e)\xi = \pi(e,g)\xi$ for a unitary representation $\pi : \Gamma\times\Gamma \curvearrowright \mathcal{H}_\pi$ explained in \S3 into the context of inductive limits of compact quantum groups as a $(\theta_t,-2)$-spherical vector in $\widetilde{\Pi} : A^{(2)} \curvearrowright \mathcal{H}_\Pi$ {\rm(}{\it n.b.}, $\beta$ should be $-1$ for the parametrization of deformation flow in \cite{MasudaNakagami:PRIMS94} and \cite{Sato:JFA19,Sato:preprint19}{\rm)}. Moreover, the discussion here explains how to apply the general theory we have developed so far to inductive limits of compact quantum groups. This will further be discussed in \cite{Ueda:Preprint} in detail.
\end{remark} 

\subsection{A realization of projective chains}

The content of this subsection overlaps with a consideration of Kishimoto \cite[section 4]{Kishimoto:RepMathPhys00}.

Let a link $\mathfrak{Z}_{n-1} \overset{\kappa}{\dashleftarrow} \mathfrak{Z}_{n}$ be given over a graph $(\mathfrak{Z},\mathfrak{E},r,s)$ as in section 7, that is, the graph as in subsection 9.1 and $\kappa$ is a function from $\mathfrak{E}$ to $[0,1]$ with $\sum_{e \in \mathfrak{E};r(e)=z}\kappa(e) = 1$ for every $z \in \mathfrak{Z}$. Then, the link induces a map from $\mathscr{M}(\mathfrak{Z}_n)$ to $\mathscr{M}(\mathfrak{Z}_{n-1})$ sending each $\nu \in \mathscr{M}(\mathfrak{Z}_{n})$ to $\nu\cdot\kappa \in \mathscr{M}(\mathfrak{Z}_{n-1})$ defined by
\begin{equation}\label{Eq9.10}
(\nu\cdot\kappa)(z) = \sum_{e\in\mathfrak{E}_{n};s(e)=z}\nu(r(e))\kappa(e)
\end{equation}
for every $z \in \mathfrak{Z}_{n-1}$. In what follows, \emph{we assume that $\kappa(e) > 0$ for all $e \in \mathfrak{E}$} for simplicity. We have obtained a projective chain $\mathscr{M}(\mathfrak{Z}_{n-1}) \overset{\kappa}{\longleftarrow} \mathscr{M}(\mathfrak{Z}_n)$. 

A multiplicity function $m : \mathfrak{E} \to \mathbb{N}\cup\{\infty\}$ and a \emph{non-zero} real number $\beta$ are also given. The purpose here is to give a construction of flow $\alpha^t$ in such a way that $\kappa = \kappa_{(\alpha^t,\beta)}$ and investigate the degree of freedom of the procedure $\kappa \rightsquigarrow \alpha^t$. 

\medskip
For each $e \in \mathfrak{E}$ we can choose a non-singular positive trace class operator $\rho_e$ on $\ell_2^{m(e)}$ in such a way that $\mathrm{Tr}_e(\rho_e) = \kappa(e)$. Here $\mathrm{Tr}_e$ stands for the canonical trace on $\ell_2^{m(e)}$. The graded graph $(\mathfrak{Z},\mathfrak{E},r,s)$ with multiplicity function $m$ gives $A = \varinjlim A_n$, and the $(w_e^t:=\rho_e^{-it/\beta})_{e \in \mathfrak{E}}$ gives a flow $\alpha^t$ on $A$ as explained in sections 9.1--9.2. 

Remark that $H_e := -\frac{1}{\beta}\log\rho_e$ defines a (possibly unbounded) self-adjoint operator, and $\exp(it H_e) = \rho_e^{-it/\beta}$ as well as $\exp(-\beta H_e) = \rho_e$ hold by function calculus. Thus, $Z_\beta(e) = \mathrm{Tr}_e(\rho_e) = \kappa(e)$ and 
\begin{align*}
Z_\beta(z) 
&= 
\sum_{(e_i)_{i=1}^n; r(e_n)=z} \kappa(e_n)\cdots\kappa(e_1) = 
\sum_{(e_i)_{i=2}^n; r(e_n)=z} \kappa(e_n)\cdots\kappa(e_2)\Big(\sum_{e\in\mathfrak{E}_1; r(e)=s(e_2)}\kappa(e)\Big) 
= \cdots = 1.
\end{align*}
Furthermore, by formula \eqref{Eq9.9} we obtain that 
\begin{equation}\label{Eq9.11} 
\kappa_{(\alpha^t,\beta)}(z,z') = 
\begin{cases}
\kappa(e) & (\text{$(z,z')=(r(e),s(e))$ for some $e \in \mathfrak{E}_n$}), \\
0 & (\text{otherwise})
\end{cases} 
\end{equation}
for every $(z,z')\in\mathfrak{Z}_n\times\mathfrak{Z}_{n-1}$. 

Consequently, we have obtained the following: 

\begin{proposition}\label{P9.5} Let $(\mathfrak{Z},\mathfrak{E},r,s)$ be a graph with properties (i)--(iv) in subsection 9.1, and $\mathfrak{Z}_{n-1} \overset{\kappa}{\dashleftarrow} \mathfrak{Z}_{n}$ be a link over the graph such that $\kappa(e) > 0$ for all $e \in \mathfrak{E}$. Let $m : \mathfrak{E} \to \mathbb{N}\cup\{\infty\}$ be an arbitrary multiplicity function and $\beta$ be an arbitrary non-zero real number. Then, the inductive system $A_{n-1} \hookrightarrow A_n$ of atomic $W^*$-algebras with separable predual associated with $(\mathfrak{Z},\mathfrak{E},r,s,m)$ admits an inductive flow $\alpha^t$ on the inductive limit $C^*$-algebra $A = \varinjlim A_n$ such that 
\begin{itemize}
\item[(i)] $\alpha^t$ satisfies the equivalent continuity conditions in Lemma \ref{L7.1}, and \item[(ii)] the link $\mathfrak{Z}_{n-1} \overset{\kappa_{(\alpha^t,\beta)}}{\dashleftarrow} \mathfrak{Z}_{n}$ associated with the $C^*$-flow $(A,\alpha^t)$ is exactly $\mathfrak{Z}_{n-1} \overset{\kappa}{\dashleftarrow} \mathfrak{Z}_{n}$ by encoding $ (r(e),s(e)) \leftrightarrow e$. 
\end{itemize} 
In particular, item {\rm(ii)} implies that $K_\beta^\mathrm{ln}(\alpha^t)$ is isomorphic to $H_1^+(\kappa) = \varprojlim (\mathscr{M}(\mathfrak{Z}_{n-1}) \overset{\kappa}{\longleftarrow}\mathscr{M}(\mathfrak{Z}_{n}))$ as completely metric, convex spaces.
\end{proposition} 

We would like to point out that specializing the multiplicity function $m : \mathfrak{E} \to \mathbb{N}\cup\{\infty\}$ suitably enables us to relax the requirement $\kappa(e) > 0$. 

\medskip
The discussion in subsection 9.1 enables us to observe that a link $\mathfrak{Z}_{n-1} \overset{\kappa}{\dashleftarrow} \mathfrak{Z}_{n}$ over a graph $(\mathfrak{Z},\mathfrak{E},r,s)$ with $\kappa(e)>0$ for all $e \in \mathfrak{E}$ uniquely determines the inductive system $(A_n,\alpha_n^t)$ under the hypothesis of multiplicity-free, i.e., $m\equiv1$. Namely, the links bijectively correspond to the inductive systems of atomic $W^*$-algebras with separable predual and continuous flows. \emph{Hence, some part of the harmonic analysis on infinite-dimensional groups from the viewpoint of integrable probability theory can naturally be understood in a wider framework based on operator algebras than that based on inductive limits of compact groups.} This is a main consequence of this paper. 

\section*{Acknowledgements} The present work is indebted to many communications with Shigeru Yamagami, Ryosuke Sato and Yoshiki Aibara. In particular, casual conversations with Yamagami and Aibara's seminar talks altogether led me to revisiting Woronowicz's old works, and also Sato's quastion to me about a draft of the paper led me to add Corollary \ref{C8.6}. The present work would never be completed without those benefits from them. So, I would like to express my  sincere gratitude to them. 

Although the main ideas of this work came to me during the last few years, I started to write down those ideas during the semi-lockdown of Nagoya University in 2020 due to the pandemic of Covid19. I would like to express my gratitude to my wife, Mariko, who has always kept and maintained an enjoyable environment in my family, which enabled me not to give up writing this paper (and to give lectures on-line) even under the pandemic. 

}

\end{document}